\documentclass[11pt]{article}

\usepackage{setspace}
\usepackage{lmodern}			% Usa a fonte Latin Modern			
\usepackage[T1]{fontenc}		% Selecao de codigos de fonte.
\usepackage{color}	
\usepackage{graphicx}			% InclusÃ£o de grÃ¡ficos--
\usepackage{epsfig,psfrag}
\usepackage{empheq}
\usepackage{lmodern}			% Usa a fonte Latin Modern			
\usepackage{amsmath,amsthm,amsfonts,amssymb,amscd,amstext,overpic}
\usepackage[top=2.0cm,bottom=2.0cm,left=2.5cm,right=2.5cm]{geometry}
\graphicspath{{Figures/}} %Comando faz o programa procurar as figuras na pasta Figuras onde todas elas estÃ£o guardadas.

\usepackage{comment}
\theoremstyle{plain}
\newtheorem{maintheorem}{Theorem}

\newtheorem{teo}{Theorem}[section]
\newtheorem{lema}[teo]{Lemma}
\newtheorem{rema}[teo]{Remark}
\newtheorem{prop}[teo]{Proposition}

\newtheorem{ques}{Question}
\newtheorem{clai}{Claim}
\newtheorem{defi}[teo]{Definition}

\newcommand{\RR}{{\mathbb R}}
\newcommand{\NN}{{\mathbb N}}
\newcommand{\ZZ}{{\mathbb Z}}
\newcommand{\DD}{{\mathbb D}}

\newcommand{\TT}{{\mathbb T}}
\newcommand{\CC}{{\mathbb C}}

\newcommand{\eqdef}{\stackrel{\scriptscriptstyle\rm def}{=}}

\renewcommand{\epsilon}{\varepsilon}

\newcommand{\interior}{\operatorname{int}}

\newcommand{\bR}{\mathbb{R}}

\newcommand{\Si}{\Sigma}

\def \AA {{\mathbb A}}

\def \CC {{\mathbb C}}
\def \DD {{\mathbb D}}

\def \NN {{\mathbb N}}

\def \RR {{\mathbb R}}
\def \SS {{\mathbb S}}
\def \TT {{\mathbb T}}
\def \UU {{\mathbb U}}

\def \XX {{\mathbb X}}

\def \ZZ {{\mathbb Z}}

\def \cA {{\mathcal A}}

\def \cC {{\mathcal C}}

\def \cE {{\mathcal E}}
\def \cF {{\mathcal F}}

\def \cH {{\mathcal H}}

\def \cL {{\mathcal L}}

\def \cO {{\mathcal O}}

\def \cR {{\mathcal R}}

\def \cU {{\mathcal U}}

\begin{document}

%\tableofcontents
%%%%%%%%%%%%%%%%%%%%%%%%%%%%%%%%%%%%%%%%%%%%%%%%%%%
%% Math Classification: 37D45, 37C70, 58F14, Secundary: 37D25, 34C23
%\input{1-Corta-Intro}
\title{Upper, down, two-sided Lorenz attractor, collisions, merging and switching}

\author{Diego Barros, Christian Bonatti and Maria Jos\'e Pacifico
\footnote{
{DB and MJP were partially supported by CAPES-Finance Code 001. CB and MJP were  partially supported by the Brazilian-French Network in Mathematics,
MJP was partially supported  by CNPq-Brazil Grant No. 302565/2017-5, FAPERJ (CNE) Grant-Brazil No. E-26/202.850/2018(239069), Pronex: E-26/010.001252/2016.}}
}
\maketitle

\begin{abstract}\footnotesize{
We present a  modified version of the well-known geometric Lorenz attractor. 
It consists of an $C^1$ open set $\cO$ of vector fields
in $\RR^3$ having an attracting region $\cU$ satisfying three properties. 
Namely, a unique singularity $\sigma$; 
a unique attractor $\Lambda$ including the singular point and
the maximal invariant in $\cU$ has at most $2$ chain recurrence classes, which are $\Lambda$ and (at most) one hyperbolic horseshoe.
The horseshoe and the singular attractor have a collision along with the union of $2$ 
codimension $1$ submanifolds which split 
$\cO$ into $3$ regions.
By crossing this collision locus, the attractor and the horseshoe may
merge into a two-sided Lorenz attractor,
or they may exchange their nature: the Lorenz attractor expels the singular point $\sigma$ and becomes a horseshoe, and the horseshoe absorbs $\sigma$ becoming a Lorenz attractor.}
\end{abstract}

\section{Introduction}

%%%%%%%%%%%%%%%%%%%%%%%%%%%%%%%%%%%%%%%%%%%%%%%%%%

%ai que saco
%%%%%%%%%%%%%%%%%%%%%%%%%%%%%%%%%%%%%%%%%%%%%%%%
Lorenz presented in  \cite{Lo63} an example of a  parameterized 2-degree polynomial system of differential equations (\ref{e.Lorenz})
as a very simplified model for the convection of thermal fluid, motivated by an attempt to understand long-term weather forecasting. 
The concrete model can be stated as
\begin{equation}\label{e.Lorenz}
(\dot{x}, \dot{y}, \dot{z})= (10(y-x), 28x-y-xz, xy-8/3z).
\end{equation}
\noindent Numerical simulations for an open neighborhood of the chosen parameters suggested that almost all points in the phase space tend to a strange attractor, nowadays called the {{\em{ Lorenz attractor}}}.
Ever since its discovery in 1963, the {\em Lorenz attractor} has been playing a central role in the research of singular flows, i.e., flows generated by smooth vector fields with {\em singularities}, {i.e., points where the flow vanishes.}
But  Lorenz's equations turned out to be very resistant to rigorous mathematical analysis, from both  conceptual (existence of a singularity accumulated by regular orbits prevents the attractor  from being hyperbolic)  as well as numerical (solutions slow down as they pass near the singularity, which means unbounded return times and, thus, unbounded integration errors) points of view \cite{Sparrow}. 
Moreover, for almost every pair of nearby initial conditions, the corresponding solutions move apart from each other exponentially fast as they converge to the attractor. That is, the attractor is {\em{sensitive to the initial conditions}}. This unpredictability is a characteristic of {\em{chaos}}. 
{Most remarkably, this attractor is robust: it can not be destroyed by any small perturbation of the original flow.}

{\noindent A very fruitful approach was undertaken, independently, by  Afraimovich,  Bykov, Shil'nikov \cite{ABS82, ABS77} and  by  Guckenheimer and Williams \cite{GW, Williams}. They constructed a {\em{geometric Lorenz attractor}} that reproduces the behaviour observed by Lorenz. 
As an abstract object, the geometric Lorenz attractor is the inverse limit of a semiflow on a branched $2$-manifold (with boundary). The flow has a singularity, 
 on the boundary of the surface, and orbits leaving a neighborhood of this singularity follow either of two branches which return to (and are glued together along) an interval of branch points transverse to the stable manifold of the singularity.
From the geometrical point of view, geometric Lorenz attractors  are flows on the
$3$-dimensional space that contains a singularity  accumulated by regular orbits, i.e., 
{orbits $\gamma$ where $X^t(x)\neq 0$ for all $x\in \gamma$ and $t>0$.}}
They have a natural cross-section given by a two-dimensional square crossed by all orbits of the flow inside the attractor except the singularity. 

  While the discovery of the Lorenz attractor
leveraged fundamental developments in Dynamical Systems, 
the  equations (\ref{e.Lorenz}) themselves have continued to resist all the attempts to prove that they exhibit a sensitive attractor. 
Bunimovich and Sinai \cite{Bu79,Si81}, have indicated a program that could prove that the Lorenz equations have a transitive attractor. 

{Rychlik and Robinson gave a proof for the so-called Shilnikov criteria \cite{Sh81} providing necessary and sufficient conditions for the birth of the Lorenz attractor from  certain classes of codimension-3 bifurcation (codimension-2 bifurcations for systems with the Lorenz symmetry), \cite{Ry90,Rob89,Rob92}.}
{From the bifurcation point of view, besides these results, 
we cite the works in \cite{Rob00, LTu20, GTS09, MPS05, MPS06,K21, LZ83}. }

A successful approach for the  Lorenz flow was through
rigorous numerics. 
Many authors have performed numerical simulations
of these equations, and 
the book by Sparrow gives an account of many of these
results proved in the seventies, \cite{Sparrow}.
Later, in this way, it was proved 
\cite{HHTZ94,HT92,MM95,MM98} that the Lorenz system of
equations exhibits a suspended Smale horseshoe \cite{Sm67}.
It implies, in particular, the existence of infinitely
many closed orbits.
 However, proving the existence of an
attractor as in the geometric models is an even harder task.
Indeed, one cannot avoid the fact that solutions slow down
as they pass near the singularity, which, as said before,
means unbounded return times and unbounded integration
errors, see \cite{Sparrow, Tu2}.  This was finally settled by Tucker in
\cite{Tu00, Tu2} around the turn of the century.
  
  The theory of dynamical systems indicates that the way to analyze equations (\ref{e.Lorenz}) is
to prove that they possess a strong stable foliation. 
{For the geometric model of the Lorenz equations proposed by Guckenheimer and Williams in  \cite{GW, Williams},  
an invariant $C^1$ strong stable bundle and the resulting strong stable foliation are assumed to exist.}
 With this assumption the theory of normally hyperbolic
attractors allows one to show that an attractor exists.
It also gives a criterion to check if
it is transitive. 
Once {the strong stable foliation} is
%these manifolds are 
known to exist with the right set of
expansion rates for one set of equations, then they persist for perturbations of the
equations; see \cite{Rob84, GW, W67, W74} for a discussion of the geometric equations and the resulting attractor. Recently, in \cite{Smania18}, the authors proved the existence of 
$C^k$-invariant foliations for Lorenz-type maps. 
Nowadays, in the literature, there is an expressive number of
results about the Lorenz geometric attractor, both from the geometric as well  as the 
measure-theoretic point of view, and we refer to the reader to see \cite{GH83, Vi00,ArPa10,AP11, OT17, Gu76} and references therein. 

  In the '90s a breakthrough was obtained by Morales, Pacifico and Pujals, \cite{MPP04}, following the very original ideas developed by Ma\~n\'e in the proof of the $C^1$ stability conjecture, \cite{Ma88}. 
They  provide a characterization of {robustly} transitive  attractors for 3-dimensional flows, of which the Lorenz attractor is the more significant example.

 After this seminal work, significant advances in this theory were achieved through the work of many authors, who gradually turned to the statistical and stochastic point of view of the Lorenz attractor, see  \cite{LMP05, AP11, AGP14, Ar21, AMV15, AM16, APPV09,AV12, SP10, GPN18, PT10, PYY}.

Taking into account that the divergence of the vector field induced by the
system (\ref{e.Lorenz}) is negative, it follows that the Lebesgue measure of the Lorenz attractor
is zero. Henceforth, it is natural to ask  about its Hausdorff dimension. 
Numerical experiments give that this value is approximately
equal to 2.062 (cf. \cite{Divakar}) and also, for some parameters, the dimension of the physical invariant measure  lies in the interval $[1.24063,1.24129]$ (cf. \cite{GN}). 
In \cite{Afra4} and \cite{St00}, this  dimension is characterized in terms of the pressure of the system and in terms of the Lyapunov exponents and the entropy concerning  a good invariant measure associated with the geometric model. But, in both cases, the authors prove that the Hausdorff dimension is greater or equal than  $2$, but not necessarily strictly greater than $2$.
A first attempt to obtain the strict inequality was given in \cite{ML},
where the authors achieve this result in the particular case that  both branches of the unstable manifold of the singularity meet the stable manifold of the singularity. 
Finally, in \cite{MPR20} the authors prove 
that the Hausdorff dimension of a geometric attractor is strictly greater than $2$.

 Despite all this progress, we still are far from  a topological 
classification of singular hyperbolic attracting sets in dimension $3$. 
Moreover,  there is also a huge gap in the understanding of unfolding parameterized families of singular hyperbolic attractors in any dimension.

 The  Lorenz attractor has a unique hyperbolic singularity, whose strong stable manifold $W^{ ss}(\sigma) ${separates the stable manifold $W^{s}(\sigma) $ into components, called the upper and lower local stable sets of $\sigma$.}  One of these local stable sets is \emph{disjoint} from the attractor. 
These properties are shared by the geometrical Lorenz model.

 This paper was motivated by a question from Adriana da Luz to the second author. 
How to construct a 3-dimensional flow presenting a singular hyperbolic attractor, 
containing a unique Lorenz-like singularity $\sigma$,  
intersecting robustly upper and lower local {stable sets} of $W^{s}(\sigma) \setminus W^{ss}(\sigma)$. 
One can easily construct examples exhibiting this phenomenon. Among those, a particular one displays a very intriguing behaviour under perturbations. The goal of the present paper is to address this matter. 
That is, this work aims to build an open set of vector fields, for which the attractors have intersections with the upper, lower and both local {stable sets}. 

 We describe a $C^1$-open set $\cO_1$ of vector fields on a three-dimensional space having a common attracting region $U$ which contains a unique singularity $\sigma$   of Lorenz type, see Definition \ref{d.lorenzlike}. 
As in the  geometric Lorenz attractor, each 
flow $X \in \cO_1$ has a global cross-section $\Sigma$, which is a topological annulus and
a first return map $P:\Sigma \to \Sigma$ possessing an invariant stable foliation $\mathcal{F}^s$.
 The intersection of the stable manifold of $\sigma$ with the cross-section  splits $\Sigma$ into two connected components,
$\Sigma^1$ and $\Sigma^2,$  and the intersection of the upper (lower) stable component of
$W^{s}(\sigma)$ 
contains each a segment, $\gamma^s_+$ and $\gamma^s_-$, respectively, transverse to the boundary $\partial \Sigma$ and connecting the two boundary components of the annulus $\Sigma$. We denote the upper (lower) stable component of $W^{s}(\sigma)\setminus W^{ss}(\sigma)$ by $W^s_+(\sigma)$ ($W^s_-(\sigma)$) respectively, {see Figure \ref{figure0}}. {We assume that both $\gamma^s_+$ and $\gamma^s_-$ belong to $\mathcal{F}^s$. The unstable manifold of $\sigma$ is one-dimensional and 
$W^u(\sigma)\setminus \{\sigma\}$ consists of two separatrices $W^u_i(\sigma), \, 1 \leq i \leq 2,$
each one having a   first intersection point $q_i$ with $\Sigma$. See Figure \ref{figure0}.}

 \begin{figure}[h!]
	\centering
\includegraphics[scale=0.2]{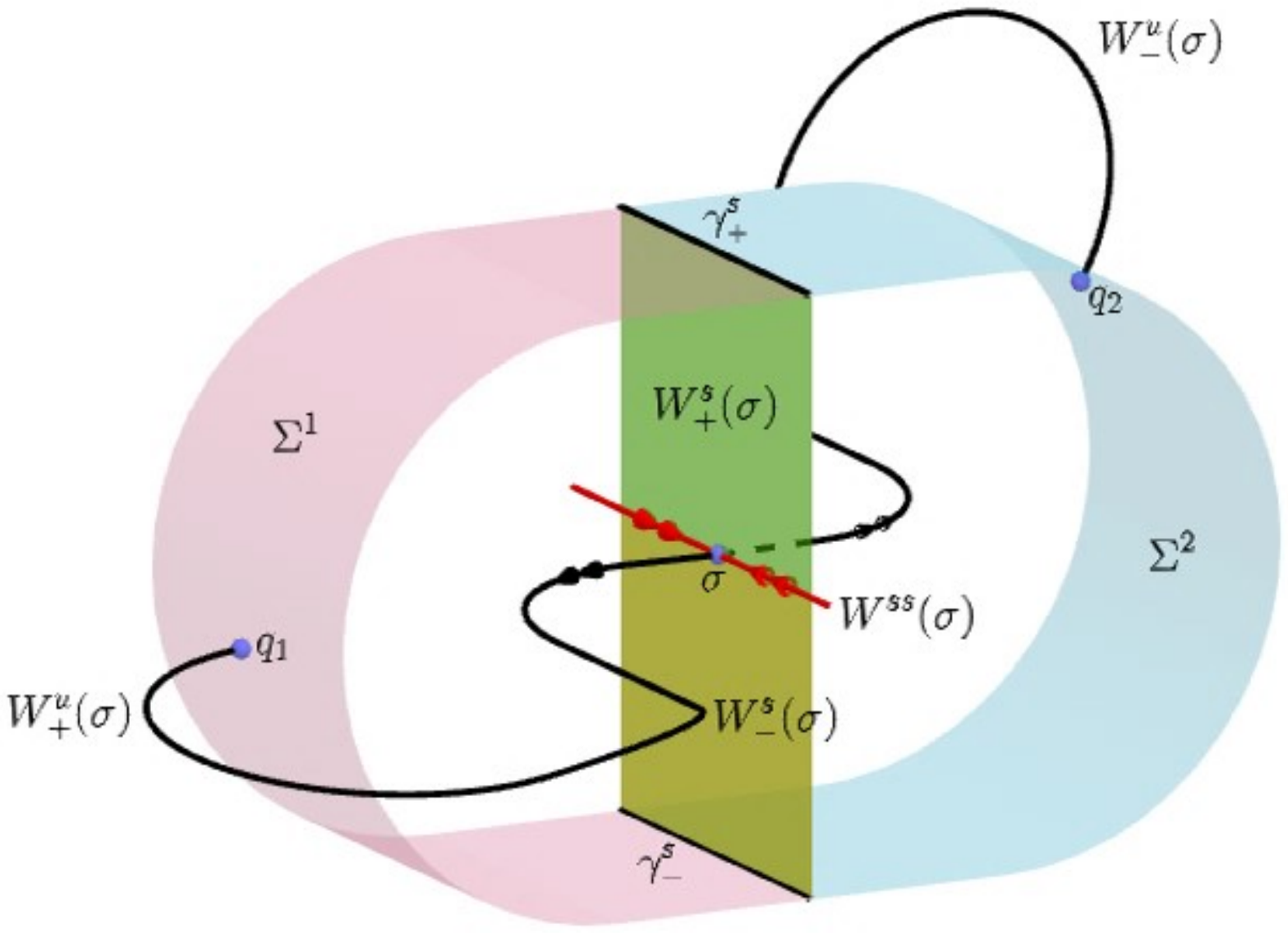}
	\caption{ The diagram displays $\Sigma = \Sigma^1 \cup \Sigma^2$ and $W^s(\sigma) \setminus W^{ss}(\sigma)=W^s_{+}(\sigma) \cup W^s_{-}(\sigma)$, and the points $q_i, 1\leq i \leq 2,$ defined above.}
	\label{figure0}
\end{figure}

{The first  result gives the topological nature of  the class of flows in $\mathcal{O}_1$, whose attractor intersects  just one or both components of the stable sets of $\sigma$, proving the existence of three disjoint open and non-empty sets in
$\mathcal{O}_1$, roughly described as:  $ \mathcal{L}^-$, whose attractors intersect just   $W^s_-(\sigma)$ , $\mathcal{L}^+$, whose attractors intersect just  $W^{s}_-(\sigma) $ and $\mathcal{L}^{-,+}$, whose attractors intersect both  $W^s_+(\sigma)$ and $W^s_-(\sigma)$.
We call the attractors in $\mathcal{L}^-$  {\em{down-Lorenz}}, the attractors in $\mathcal{L}^+$  as {\em{upper-Lorenz}} and the attractors in $\mathcal{L}^{-,+}$  {\em{two-sided Lorenz}}.
 We denote $\mathcal{L}^{-,+} \cup \mathcal{L}^+ \cup \mathcal{L^-}=\mathcal{L}$. See  Section \ref{ss.region} for the precise  definitions of these sets.}

 Now, let us explain what we mean by a fake horseshoe. It consists of a sequence of operations on the unit square, quite similar to the horseshoe, the  only difference being the way back to the square of the resulting folded rectangle: the bottom of the folded rectangle fits back like the top of the starting square. Figure \ref{fakehorseshoe} displays the main features of a fake horseshoe. {Note that a fake horseshoe   preserves the orientation, while the horseshoe does not.}

 \begin{figure}[th]
	\centering
	\includegraphics[scale=0.17]{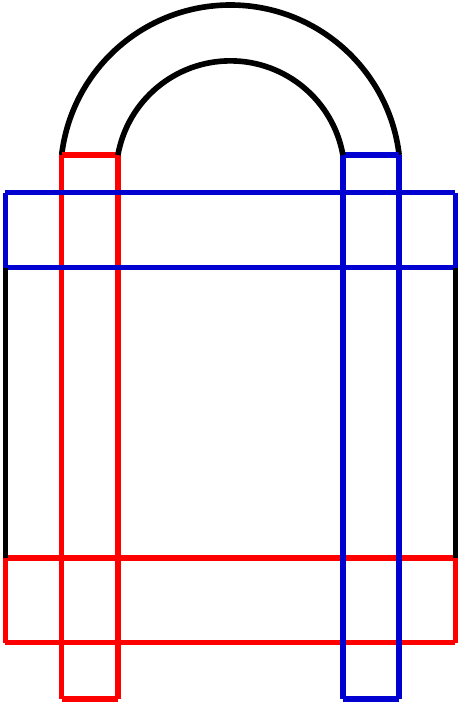}
	\hspace{0.5cm}
	\includegraphics[scale=0.15]{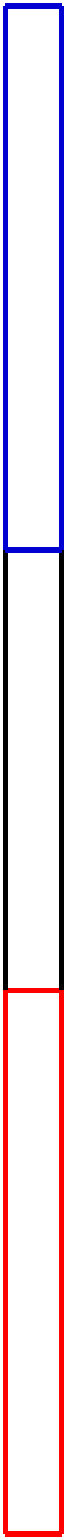}
	\hspace{0.5cm}
	\includegraphics[scale=0.17]{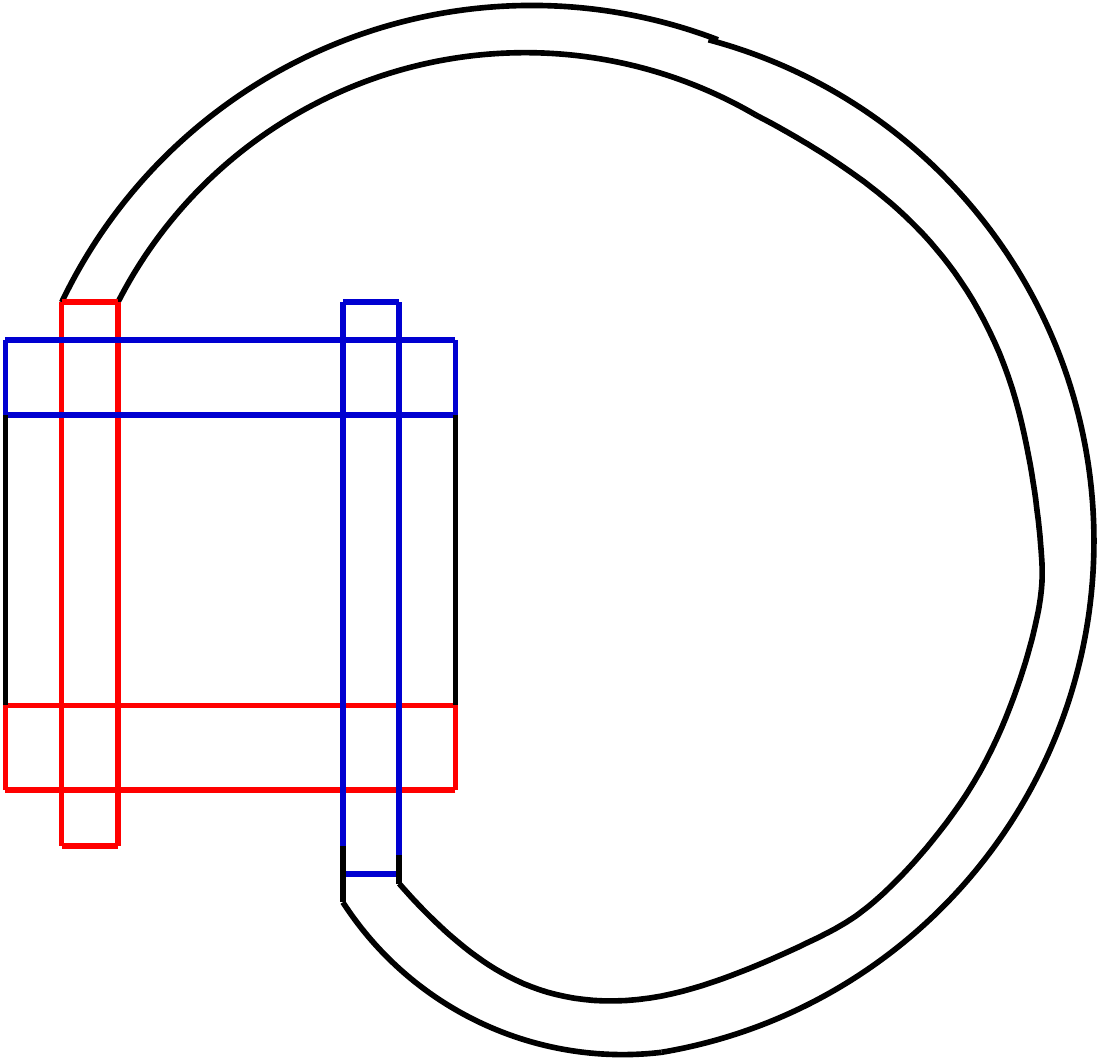}
	\caption{ The usual horseshoe and a fake horseshoe}\label{f.horseshoe-1}
	\label{fakehorseshoe}
\end{figure}

\begin{maintheorem}\label{igual-a-t.up} 
Any vector field $X\in\cL^+$ admits exactly $2$ chain recurrence classes: one is an upper-Lorenz attractor, and the other is a hyperbolic basic set, topologically equivalent to the suspension of a fake horseshoe. 
	The symmetric statement holds for $\cL^-$, interchanging the upper for the  down. 
\end{maintheorem}

{A saddle $p$  (i.e., a hyperbolic singularity point with nontrivial stable and unstable bundles) of a flow has a homoclinic tangency if its stable and unstable manifolds have some nontransverse intersection. In this case, we say that $p$ has a homoclinic loop. Two saddles with different u-indices (i.e., dimension of the unstable bundle) have a heterodimensional cycle if their stable and unstable manifolds intersect cyclically (by dimension deficiency, one of these intersections is necessarily nontransverse), see \cite{Lorenzo95}. }

{Two saddles of a $C^1$ vector field $ X $ are homoclinically related if the invariant manifolds of their orbits intersect transversely and cyclically. To be homoclinic related defines an equivalence relation on the set of saddles of $X$. Two saddles that are homoclinically related have the same u-index. The homoclinic class of a saddle $p$ of $X$ is the closure of the saddles of $X$ that are homoclinically related to $p$. A homoclinic class is a transitive set (i.e., it contains a point whose orbit is dense in the set). }

Consider $\cH^i\subset \cO_1$, $i=1, 2$, the hypersurface corresponding to the vector fields $X$ for which the first intersection  point $q_i$ of $W^u_i(\sigma)$ with $\Sigma$   
belongs to $\gamma^s_+\cup\gamma^s_-$. 
In other words, $X$ belongs to $\cH^i$ if $\sigma$ admits a homoclinic loop for the unstable separatrices $W^u_i(\sigma)$ so that the homoclinic loop cuts $\Sigma$ at a unique point $q_i$ {(see Figure \ref{figure0})}.  

\begin{maintheorem}\label{igual-prop-p.Hi}
	For any $X\in\cO_1$  in one of the hypersurfaces $\cH^1$ or $\cH^2$, the maximal invariant in $U$ is a transitive singular attractor meeting both {stable sets} $W^s_{-}(\sigma)$ and $W^s_{+}(\sigma)$. 
\end{maintheorem}

 The statement of the theorem above does not announce a two-sided Lorenz attractor, because the hypersurface  $\cH^i, i= 1, 2,$ is not an open subset, and then the robustness of the transitivity is not ensured. We split each $\cH^i$ into two $\cH^i=\cH^i_+\cup \cH^i_-$ according to $\{q_i\in\gamma^s_+\}$ and $\{q_i\in\gamma^s_-\}$ respectively. We then prove the following result, which characterizes the topological nature of the attractors for flows in the hypersurfaces $\cH^i_+$ and $ \cH^i_-$.
 There is still the case of a double homoclinic loop, that  
 occurs when both  intersection points $q_1$ and $q_2$ of
 the unstable manifold of $\sigma$ with the cross-section fall in the same segment  
 of the intersection of the upper (lower) stable component of
 $W^{s}(\sigma)$ with the cross-section.
 
 We denote with $\cH^{1,2}_+$ the collection of flows in $\mathcal{O}_1$ such that
 $q_1,q_2\in \gamma^s_+$.
 Similarly, we denote with $\cH^{1,2}_-$ the collection of flows in $\mathcal{O}_1$ such that  $q_1,q_2\in \gamma^s_-$, see Figure \ref{f.p.Hi2}.
 Thus, both unstable separatrices of $\sigma$ are homoclinic connections and are included in the same connected component of $W^{s}(\sigma) \setminus W^{ss}(\sigma)$.
 Clearly, $\cH^{1,2}_+$ and $\cH^{1,2}_-$ are codimension $2$ submanifolds included in $\cH^1\cup \cH^2$. 
 The next result ensures that for $X\in\cH^i$ out of $\cH^{1,2}_+$ and $\cH^{1,2}_-$, the transitivity of the attractor is robust, 
 showing that the attractor is a two-sided Lorenz attractor.  
 
 \begin{maintheorem}\label{L-igual a Lemma l.Hi} If $X\in \cH^i\setminus (\cH^{1,2}_+ \cup\cH^{1,2}_-)$, 
 	there is a neighborhood $\cU(X)$ of $X$ such that the maximal invariant set for every $Y\in \cU(X)$ is a two-sided Lorenz attractor.
 \end{maintheorem}

Before we announce the next result, recall that $W^s_+(\sigma)$ ($W^s_-(\sigma)$) 
is the upper (lower) stable component of $W^{s}(\sigma)\setminus W^{ss}(\sigma)$ and  
the sets $\mathcal{L}^+, \mathcal{L}^-, \mathcal{L}^{+,-}, \mathcal{L}$ are as above.

\begin{maintheorem} 
\label{igual-a-p.L+-} For any $X\in\cL^{+,-}$ the maximal invariant set in $U(X)$ is a transitive singular hyperbolic attractor meeting {$W^{s}_{-}(\sigma)$ and $W^s_{+}(\sigma)$}. That is, it is a two-sided Lorenz attractor. 
\end{maintheorem}

From our construction, we get that every  $X \in \mathcal{L}$
 contains a robust attractor (that is, there exists an open neighborhood $\mathcal{U}$ of $X$ in the space of vector fields, such that every $Y \in \mathcal{U}$ has an attractor). Besides that, although the complement of $\mathcal{L}$
 in $\mathcal{O}_1$ is nonempty it has an empty interior.	 
	 Thus, the next step is to understand the reasons 
why an attractor stops intercepting just one component of $W^{s}(\sigma)\setminus W^{ss}(\sigma)$ and suddenly  starts to intersect both (and vice versa).
 
We will see that this phenomenon occurs due to the appearance of certain types of homoclinic and heteroclinic connections related to the position of the intersection points $q_1,\,q_2$ of $W^u_i(\sigma)$ with the cross-section.

Recall that $\Sigma$ is split into two connected components $\Sigma^1$ and
  $\Sigma^2$, determined by the intersections of $W^{s}(\sigma)$ with
  $\Sigma$. Assuming that the first return map has an expansion rate greater 
  than the golden number $\varphi=\frac{1+\sqrt{5}}{2}$, when $q_i$ falls into  $\Sigma^j$ with $i \neq j$, the Poincar\'e map
has a fixed point $p_i\in \Sigma^j$.   Typically,  a heteroclinic cycle arises when 
 $q_i$ belongs to the stable leaf $W^s_j$ through $p_j$.

 We denote by $\cO_{\varphi}$ the subset of flows in
$\cO_{1}$ for which the first return map has an expansion bigger than $\varphi$,  {note that $\Sigma \setminus (W^s_1 \cup W^s_2)$ splits into two connected components $\Sigma_{-}$ and $\Sigma_{+}$, see Figure \ref{Mais-f-Up-Lorenz-em-OTilda++}.}
Let $\cH\cE_i \subset \cO_{\varphi}$ be the subset of flows corresponding to the vector fields $X$ for which $q_i\in W^s_j$.  In other words, for $X\in \cH\cE_i$ the unstable separatrix of $\sigma$ corresponding to $q_i$ is a heteroclinic connection with $p_j$.
 The subsets $\cH\cE_i$ are codimension $1$ submanifolds of $\cO_\varphi$.

   \begin{maintheorem}\label{igual-a-p.HE1HE2} For any $X\in \cH\cE_1\cap\cH\cE_2$, there is a unique chain recurrence class, which is not transitive. Both $\Si_-$ and $\Si_+$ are invariant by $P$.  
   	The maximal invariant set of the restriction of $P$ to $\Sigma^i, 1\leq i \leq 2,$ is transitive: every unstable segment in $\Sigma^i$ has its iterates which cut every segment in $\Sigma^i$. 
 The open set $U$ splits into $2$ regions, each containing a  \emph{full Lorenz} (see Definition \ref{fullLorenz}),   intersecting along $W^s(p_1)\cup W^s(p_2)\cup W^u(\sigma)$.  See Figure \ref{f.p.HE1HE2}.
 \end{maintheorem}
  \noindent {A similar construction for a full Lorenz was observed in \cite{GKKK22} where it was called conjoined Lorenz twins.}
      \begin{figure}[th]
\centering
	\includegraphics[scale=0.17]{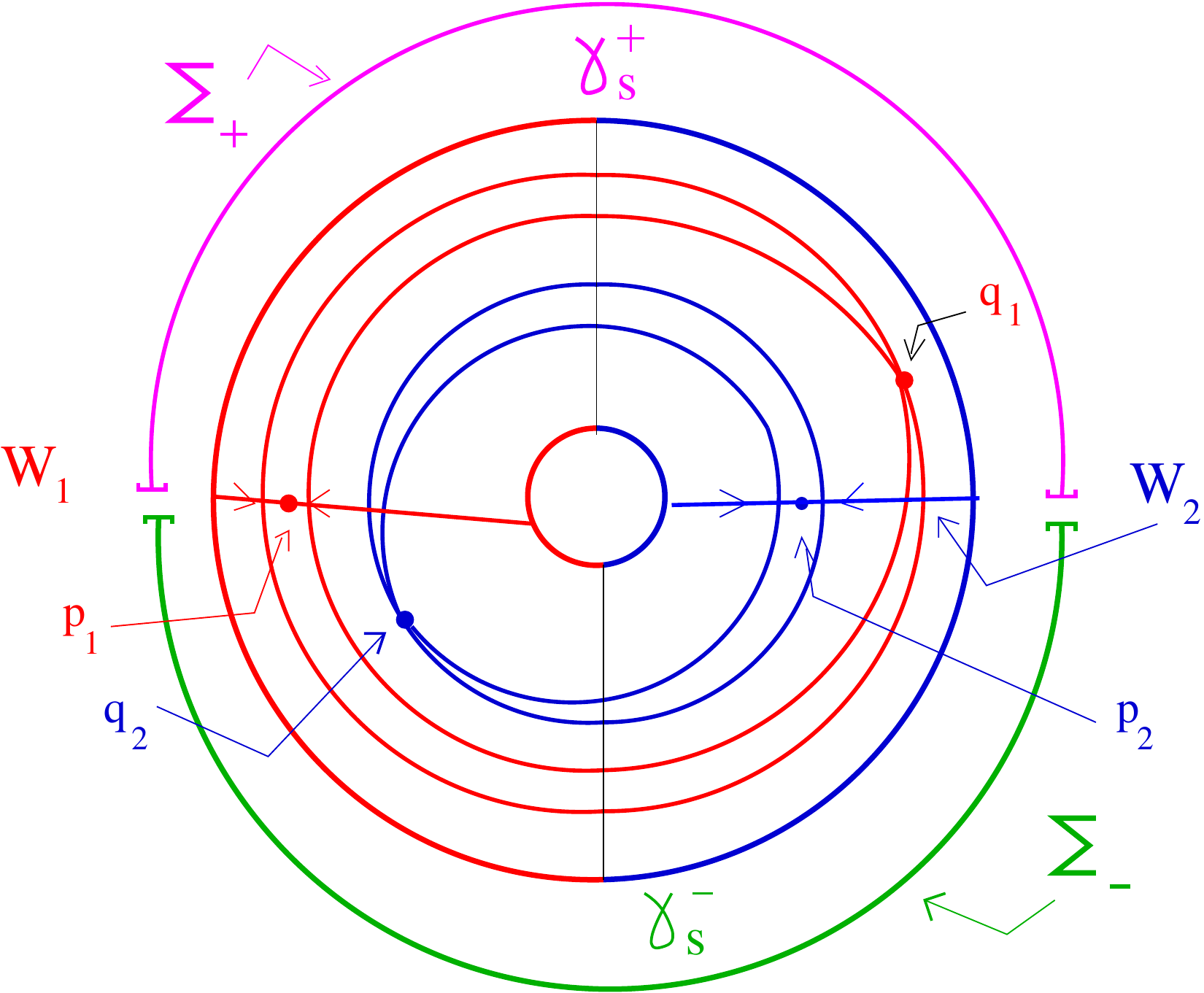}
	\caption{ $ \Sigma_+$ and $\Sigma_-$:  determined by $W_1$ and $W_2$, the stable manifolds of $p_1$ and $p_2$ respectivelly. }\label{Mais-f-Up-Lorenz-em-OTilda++}
\end{figure}
 
  The goal now is to describe the drastic changes in the behaviour appearing in the topological dynamics when a family $X_\mu,
 \,  \mu=(\mu_1,\mu_2) \in [-1,1]^2, $ crosses the boundary {of the regions considered at Theorems \ref{igual-a-t.up}, \ref{igual-prop-p.Hi}, \ref{L-igual a Lemma l.Hi}, 
 \ref{igual-a-p.L+-}, and \ref{igual-a-p.HE1HE2}  above.}
 
{We can refine the analysis of the topological behaviour of  flows in $\cO_{\varphi}$  based on the  position of $q_{1}$ and $q_{2}$ in $\Sigma$.  
 This introduces a new region, denoted by $\widetilde{\cO_\varphi}^{+,+}$, defined as 
 the open subset where $q_1$ and $q_2$ belong to different components $\Sigma_\pm$ and $\Sigma_\mp$.}
 
   The behaviour of the topological dynamics in all these regions is established in the next theorem. 		
 \begin{maintheorem}\label{igual-a-l.switch}  With the previous notations, we obtain the following:  
 	\begin{itemize}
 		\item the vector field $X_\mu$ belongs to $\cL^+$  if and only if $\mu$ belongs to the quadrant $\mu_1>0, \mu_2>0$.  In other words, the upper full Lorenz of $X_{0,0}$ becomes a  (robust) Lorenz attractor, and the down full Lorenz of $X_{0,0}$ becomes a fake horseshoe. 
 		\item the vector field $X_\mu$ belongs to $\cL^-$  if and only if  $\mu$ belongs to the quadrant $\mu_1<0, \mu_2<0$. In other words, the upper full Lorenz  of $X_{0,0}$ becomes a  (robust) Lorenz attractor, and the down full Lorenz  of $X_{0,0}$ becomes a fake horseshoe. 
 		\item the vector field $X_\mu$ belongs to $\widetilde{\cO_\varphi}^{+,+}$ if and only if $\mu$ belongs to the quadrant $\mu_1<0, \mu_2 > 0$ or  the quadrant $\mu_1>0, \mu_2<0$.  In other words, the two full Lorenz of $X_{0,0}$ merges into a two-sided Lorenz attractor. 
 	\end{itemize}
	 \end{maintheorem}
	
	The next theorem in the paper reads as follows:

\begin{maintheorem}\label{1-igual-a-l.switch}  
With the previous notations, 
	considering  a $1$-parameter family {$ X_{\mu=(\mu_1,\alpha\mu_1)}$ } for some $\alpha>0$,
	we obtain the following:
 	\begin{itemize}
 		\item for $\mu_1<0$,  the vector field admits a down Lorenz attractor and an upper fake horseshoe, 
 		\item when $\mu_1=0$ the fake horseshoe becomes a full Lorenz  and, simultaneously, enters in collision with the upper Lorenz attractor, which becomes a full Lorenz, 
 \item for $\mu_1>0$, the upper full Lorenz  becomes a fake horseshoe when the down full Lorenz becomes a  robust Lorenz attractor. 
 	\end{itemize}
 \end{maintheorem}
\noindent  Figure \ref{f.l.doubleconn}
below describes the main features established in Theorems \ref{igual-a-l.switch} and \ref{1-igual-a-l.switch}.
\vspace{0.2cm}

\begin{figure}[h!]
	\centering
	\includegraphics[scale=0.25]{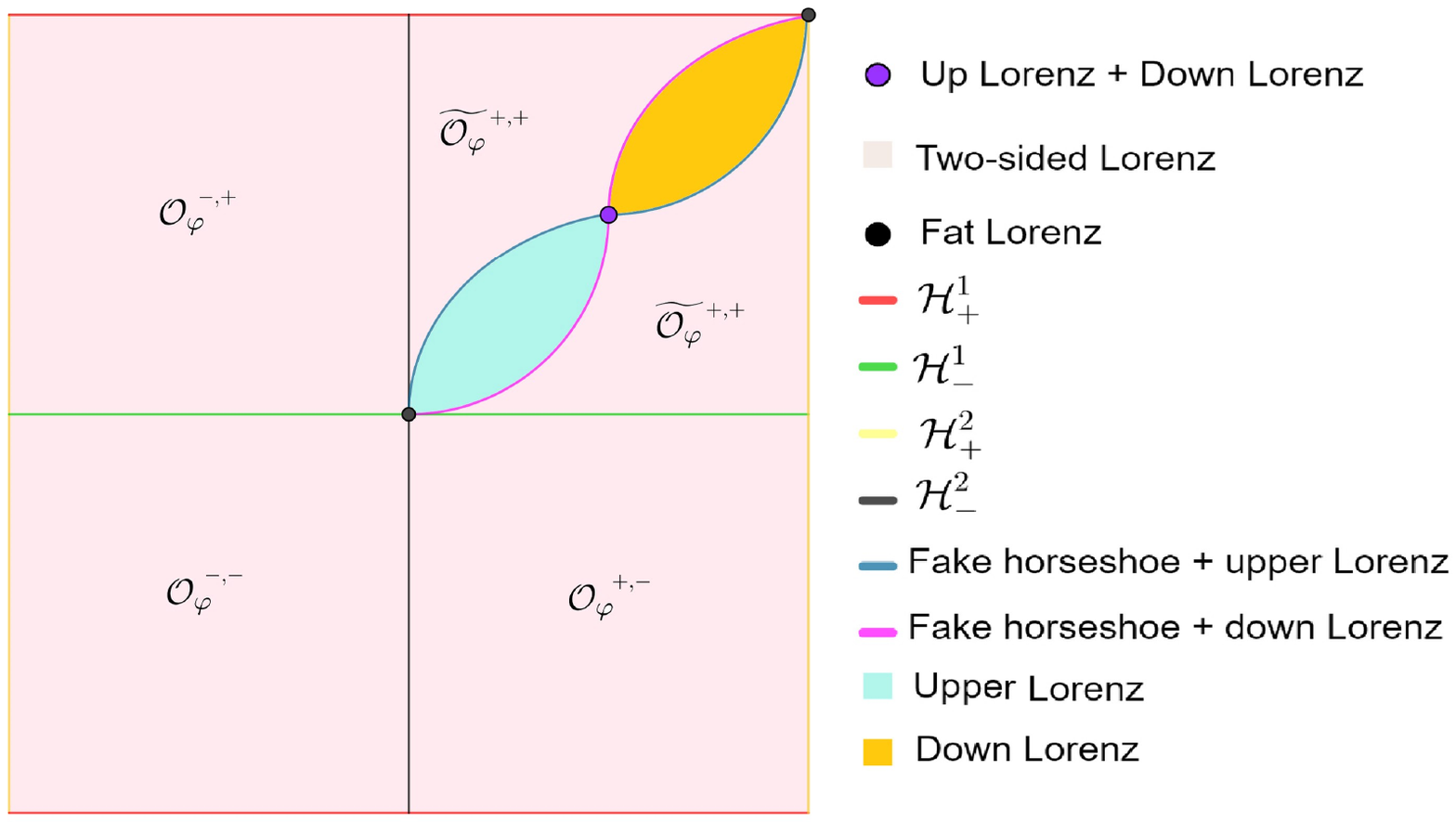}
	\vspace{0.3cm}
\caption{The bifurcation diagram of flows in $\mathcal{O}_{\varphi}$.}
\label{f.l.doubleconn}
\end{figure}

 Next,  we study in larger depth  the topological nature of a flow in $\mathcal{O}_{1}$. { Recall that if $X\in \cO_1$ then it contains a unique singularity $\sigma$   of Lorenz type, see Definition \ref{d.lorenzlike}. 
Moreover, 
 $X $ has a global cross-section $\Sigma$, which is a topological annulus and
a first return map $P:\Sigma \to \Sigma$ possessing an invariant stable foliation $\mathcal{F}^s$.
 The leaves of the foliation $\cF^s$ are segments crossing $\Sigma$ (connecting the two boundary components of $\Sigma$). 
  The leaves space  $\Sigma/\cF^s$ is a (topological) circle $\SS^1_X$.  The segments $\gamma^s_+$ and $\gamma^s_-$ are leaves of $\cF^s$ and induce each a point $c_+$ and $c_-$ respectively, on $\SS^1_X$. 
 We assume that $X$ preserves $\cF^s$ and thus the first return map $P$ preserves the foliation $\cF^s$.
As a consequence, $P$ passes to the quotient as a map $f=f_X$ defined from $\SS^1_X\setminus\{c_+,c_-\}$ to $\SS^1_X$. As in the geometric model, 
we will see that $f_X$ restricted to each interval of $\SS^1_X\setminus\{c_+,c_-\}$ is a diffeomorphism.}

As for the  geometric model of the Lorenz attractor {presented in \cite{GW},} we will see  that  to any vector field in $\cO_1$,  we can associate  combinatorial data, called the itinerary of the discontinuities. 
These itineraries   
induce a semi-conjugacy of $(\Sigma, P)$ with $(\XX,\mathfrak{S})$, where $\mathfrak{S}$ is the shift map on $\XX$, and $\XX$ is an appropriate alphabet chosen to represent the crossing points of the unstable manifold of the singularity with the cross-section. 

 Being {precise}, $P$ is defined on $\Sigma\setminus (\gamma^s_+\cup \gamma^s_-)$ which is not invariant under $P$. Thus we get conjugacy in the restriction of $\Sigma\setminus W^s(\sigma)$.
But, since the Poincar\'e map preserves a stable  foliation and the itineraries for the Poincar\'e map are functions of the stable leaf,  they pass to the quotient on the leaves' space, inducing 
itineraries for the one-dimensional quotient map. 
We use this symbolic analysis  to show that each neighborhood of a flow in $\mathcal{O}_{1}$, in the $C^{1}$ topology, contains flows whose attractor exhibits all of an uncountable family of topological types.
As in the case of the  geometric Lorenz attractor, we prove that this phenomenon is co-dimension two stable. The topological types exhibited by the various versions of this example can be distinguished by two parameters. Moreover,  we construct  a two-parameter family of flows
in $\mathcal{O}_{1}$ which is $C^{1}$ stable. This is the content of our last result below.

 \begin{maintheorem}\label{igual-a-t.conjugado} The restrictions to the maximal invariant sets  of $X$ and $ Y \in \mathcal{O}_\varphi$ are topologically equivalent by a homeomorphism close to identity if, and only if, $X$ and $Y$ have the same itineraries.
\end{maintheorem}
 
This paper is organized as follows. 
 Section \ref{s-preliminaries}  gives  the concepts, definitions, and results proved elsewhere that will be needed in the text.  
 Section \ref{s-the-open-set-1}, 
constructs  the open set of flows $\cO_{1}$ that  will be analyzed in the remainder of the text.
 Section \ref{s-the-open-set-2} gives the background and auxiliary results to prove  the main results.
 In Section \ref{s.cartografia} we prove theorems \ref{igual-a-t.up}, \ref{igual-a-p.L+-}, \ref{igual-prop-p.Hi}, \ref{L-igual a Lemma l.Hi} and
 \ref{igual-a-p.HE1HE2}.
 In Section \ref{s-collisions} we prove theorems \ref{igual-a-l.switch} and \ref{1-igual-a-l.switch}. { In Section \ref{ss-symbolic} we prove theorem \ref{igual-a-t.conjugado}}.
\vspace{0.2cm}

{\em Acknowledgements}. The authors extend their gratitude to the referee for the valuable comments which  improved the final version of the manuscript.
%%%%%%%%%%%%%%%%%%%%%%%%%%%%%%%%%%
%%%%%%%%%%%%%%%%%%%%%%%%%%%%%%%%%%
\section{Preliminaries}\label{s-preliminaries}
		
%%%%%%%%%%%%%%%%%%%%%%%%%%%%%%%%%%%%%%%%%%%%%%%%
\subsection{Basic topological notions}
%%%%%%%%%%%%%%%%%%%%%%%%%%%%%%%%%%%%%%%%%%%%%%%%

Let  $\mathfrak{X}^r(\mathbb{R}^3)$, $\,r\geq 1,$ be the space of vector fields defined on $\mathbb{R}^3$, with the $C^r$ topology. 
For $X \in \mathfrak{X}^r(\mathbb{R}^3) $, we denote by $X^t$ the  flow of $X$.	If $U\subset M$, $\overline{U}$ denotes the closure of $U$, $\interior U$ the interior of $U$ and $\partial U$ the border of $U$.
An invariant compact set $\Lambda$ of $X$ is \emph{transitive} if there is $p\in\Lambda$ so that its positive orbit is dense in $\Lambda$ (this notion is called \emph{topological ergodicity} by some authors). 

Recall that a sequence $\{p_i\}$ is a $\varepsilon$-pseudo-orbit for $X$ if there is $t_i>1$ such that $d(p_{i+1},X^{t_i}(p_i))<\varepsilon$, for all $i$. 
A point $p$ is \emph{chain recurrent} if for each $\varepsilon>0$ there is a $\varepsilon$-pseudo orbit ~$p=p_0,p_1,\dots,p_n=p$. 
The recurrent chain set $\cR(X)$ is the set of all chain recurrent points. It is a compact invariant set. 
An invariant compact set $\Lambda$ is \emph{chain recurrent} or \emph{chain transitive} if for every $\varepsilon>0$ it admits a $\varepsilon$-pseudo orbit $\{x_i\}$ , $x_i\in \Lambda$, dense in $\Lambda$.

Two points $p,q$ are chain equivalent if for any $\varepsilon>0$ there are $\varepsilon$-pseudo orbits from $p$ to $q$ and from $q$ to $p$. 
Conley's theory \cite{Co} proves that  chain-equivalence is an equivalence relation whose equivalence classes are called the \emph{chain recurrence classes}. They define a partition of $\cR(X)$ by invariant pairwise disjoint compact sets. A compact region $U$ is \emph{attracting} for a vector field $X$ of $\mathbb{R}^3$  if its boundary is a codimension $1$ submanifold transverse to $X$ and $X$ is entering in $U$. 

\begin{defi}\label{d-atrator}
\begin{itemize}
\item A compact invariant set $\Lambda_X$ of $X^t$  is   \emph{attracting}  if there exists an open neighborhood $U \supset \Lambda_X$  which is an attracting region and such that $\Lambda_X$ is the maximal invariant set in $U$, that is: 
$$\Lambda_X= \underset{t >0}{\bigcap}X^t(U).$$ 
\item  $\Lambda_X$ is \emph{an attractor} if it is  attracting and chain transitive (in particular, 
$\Lambda_X$ is a chain recurrence class). 
\item $\Lambda_X$ is  $C^r$ robust  if  there is an attracting region $U\supset \Lambda_X$ such that  $\Lambda_Y$ is an attractor of $Y$ for every vector field $Y$ which is $C^r$ close to $X$.

\item An invariant compact set $K$ is a \emph{quasi-attractor} if it is a chain recurrence class  and admits a basis of neighborhoods that are attracting regions. 
There is a decreasing sequence $U_i\subset U_{i-1}$ of attracting regions so that $K=\bigcap_i U_i$. 
\end{itemize}
\end{defi}		

Notice that the existence of an attractor is not ensured a priori.
 But, Conley, \cite{Co}, proves the existence of quasi-attractors in any attracting region (for a vector field on a compact manifold). 

\begin{defi}
	Two vector fields $X$ and $Y$ defined on $\mathbb{R}^3$ are topologically equivalent if there exists a homeomorphism $h:\mathbb{R}^3 \to \mathbb{R}^3$ such that $h(\mathcal{O}_X(p))=\mathcal{O}_Y(h(p))$ and 
 $\forall \,\, p \in M$ and $\varepsilon>0, \exists \,\, \delta>0$ so that for $t \in (0,\delta)$ there is $s \in (0, \varepsilon)$ satisfying $h(X^t(p))=Y^s(h(p))$.	
\end{defi}	

%%%%%%%%%%%%%%%%%%%%%%%%%%%%%%%%%%%%%%%%%%%
%%%%%%%%%%%%%%%%%%%%%%%%%%%%%%%%%%%%%%%%%%%%%%%%
\subsection{Singular points}
%%%%%%%%%%%%%%%%%%%%%%%%%%%%%%%%%%%%%%%%%%%%%%%%

A point $p$ is singular if $X$ vanishes at $p$,  otherwise $p$ is regular.  The set of singular points of $X$ is denoted by $Zero(X)$. 
A point $\sigma\in Zero(X)$ is hyperbolic if the real part of the eigenvalues of $DX(\sigma)$ does not vanish.

\begin{defi} \label{d.lorenzlike}  A singularity $\sigma$ of $X$ is Lorenz-like if the eigenvalues $\lambda_i \in \mathbb{R},\, i \in \{ss, s, u\},$ of $DX(\sigma)$ satisfy {$0<-\lambda_s<\lambda_{u}<-\lambda_{ss}$.} 
\end{defi}

\begin{defi}\label{naoressonancia}
In addition,  $\sigma$ is non-resonant if, for all $N>2$ and any choice of nonnegative integer numbers $m_1,m_2,m_3$ with $2\le \overset{3}{\underset{j=1}{\sum}}m_j < N$, we have 
	$\overset{3}{\underset{j=1}{\sum}} m_j\lambda_j \neq \lambda_i$.
\end{defi}

Hartman Grobman's theorem asserts that a vector field is locally topologically equivalent to its linear part in a small neighbourhood of a hyperbolic singular point, \cite{Hartman}.  Then Sternberg  provides conditions to guarantee that this local topological equivalence is true of class $C^2$:

	\begin{teo}		
		Let $X$ a vector field and let $n \in \mathbb{Z}^+$ be given. Then there exists an integer $N = N(n) \ge 2$ such
		that: if $A$ is a real non-singular $d \times d$ matrix with eigenvalues {$\lambda_1, \ldots ,\lambda_d$}  satisfying
		$$2\le \overset{d}{\underset{j=1}{\sum}}m_j < N \,\,\,\mbox{and} \,\,\,
		\overset{d}{\underset{j=1}{\sum}} m_j\lambda_j \neq \lambda_i$$ 
		for any choice of nonnegative integers $m_1,m_2 \ldots m_d$ and if the application $X(v)=A(v)+\psi(v)$ and  $\psi$ is of class  $C^n$ for small $||v||$ with  $\psi(0)=0$ {and} $\partial_v\psi(0)=0$; then there exists a $C^N$ diffeomorphism  from a neighbourhood of $v = 0$ to a  neighbourhood of $\xi = 0$ that define a topologic conjugation between $X$ and $A$.
	\end{teo}	
	
	Furthermore, the linearizing diffeomorphisms depend continuously on the $C^2$ topology from the vector field $X$, see for instance \cite[Corollary in the Appendix]{Pa84} which provides a stronger statement. 

%%%%%%%%%%%%%%%%%%%%%%%%%%%%%%%%%%%%%%%%%%%%%%%%%%%
\subsection{The shift map $\mathfrak{S}$}
%%%%%%%%%%%%%%%%%%%%%%%%%%%%%%%%%%%%%%%%%%%%%%%%%%%

We consider $\XX=\{A_0,A_1,B_0,B_1\}^\NN$ the space of infinite positive words in the alphabet $\{A_0,A_1,B_0,B_1\}$. 
We endow the alphabet $\{A_0,A_1,B_0,B_1\}$ with the total order 
$A_0<A_1<B_0<B_1<1.$
We endow $\XX$ with the corresponding lexicographic order that we denote by $\prec$ (and $\preccurlyeq$ for the non-strict order). 
The shift map $\mathfrak{S} : \XX \to \XX$ is defined as
$$ 
(w_j)_{j\geq 0} \in \XX \mapsto \mathfrak{S} (w_j)_{j\geq 0}= (w_{j+1})_{j\geq 0}.$$

\noindent We also define the {\em{ star map}} (denoted by $\star$) on $\XX$ as follows:
given a sequence $w =(w_0, w_1,\cdots ) \in \XX$ and a letter $L \in \{A_0,A_1,B_0,B_1\}$, {we set}
$
L \star w \eqdef (L, w_0, w_1, \cdots).
$

%%%%%%%%%%%%%%%%%%%%%%%%%%%%%%%%%%%%%%%%%%%%%%%%%%%	
\subsection{Hyperbolic notions}	\label{s-hyberbolic}
%%%%%%%%%%%%%%%%%%%%%%%%%%%%%%%%%%%%%%%%%%%%%%%%%%%

\begin{defi} Let $X$ be a vector field of a manifold $M$. A compact invariant set $\Lambda\subset M\setminus Zero(X)$ is hyperbolic if the tangent bundle $TM|_\Lambda$ over $\Lambda$ admits a splitting   
$TM|\Lambda= E^s\oplus \RR X\oplus E^u \quad \mbox{such that}$
\begin{itemize}
 \item the bundles $E^s$ and $E^u$ are continuous and invariant under the  derivative of the flow. 
 \item  there exists $ T>0$ so that  $\forall \,\,\, p\in\Lambda$ and  $\forall$  unit vectors $u\in E^s(p)$ and $v\in E^u(p)$, it holds 
 $\|DX^T(u)\|<\frac 12 \,\, \mbox{ and } \,\, \|DX^T(v)\|>2$  
 (one says that $E^s$ is \emph{uniformly contracted} and $E^u$ is \emph{uniformly expanded}). 
\end{itemize}
\end{defi}

A hyperbolic invariant compact set $K$ is called a \emph{ basic hyperbolic set} if it is transitive and admits a neighbourhood on which it is the maximal invariant set. 
If $K$ is a hyperbolic set, we denote $E^{cs}=E^s\oplus \RR X$ and 
$E^{cu}= E^u \oplus \RR X$, and these bundles are called, respectively, the weak stable and weakly unstable bundles. 

\begin{defi}\label{defsinghip} A compact invariant set  $\Lambda \subset M$ for $X^t$ is \emph{partially hyperbolic} if there exists a continuous and $DX^t$-invariant splitting {$T_{\Lambda}M=E^{s} \oplus E^{cu}$} and constants $\lambda\in]0,1[, K>0$ such that for all
	$x\in \Lambda$ and $t\geq 0$, the following inequalities hold:
	\begin{enumerate}
\item[(a)] 
$||DX^t|_{E^s_{x}}||\cdot \|DX^{-t}|_{E^{cu}_{X_t(x)}}\|\le K\lambda^t
		\quad \mbox{(domination)};$ 
		\item[(b)] $\|DX^t|_{E^{s}_x}\| \le K\lambda^t \quad \mbox{(uniform contraction). }$
	\end{enumerate}	
	In addition, if $E^{cu}_{\Lambda}$ is volume expanding, that is, $|\det(DX^t|_{E^{cu}_x})|>Ke^{\lambda t}$ for all $x \in \Lambda$ and $t \ge 0$, $\Lambda$ is, by definition,  a {\em{singular hyperbolic}} set.
\end{defi}

Note that any hyperbolic set is also partially hyperbolic. 
Notice that if $\Lambda_X$ an invariant compact set disjoint from $Zero(X)$ is partially hyperbolic if and only if it is hyperbolic. 
In dimension $3$, if $\Lambda_X$ is partially hyperbolic, then every singular point in $\Lambda_X$ is a Lorenz-like singularity, see \cite{MPP04}. 

%%%%%%%%%%%%%%%%%%%%%%%%%%%%%%%%%%%%%%%%%%%
\subsection{Invariant manifold and foliations}
%%%%%%%%%%%%%%%%%%%%%%%%%%%%%%%%%%%%%%%%%%%

Let $\Lambda$ be a compact invariant set for the flow $X^t$ and  $p \in \Lambda$. The stable $W^{s}_X(p$  and unstable $W^{u}_X(p)$ sets at $p$ are defined by
\begin{eqnarray*}
	W^{s}_X(p)&=&\{q \in M\,: \mbox{dist}(X^t(q),X^t(p)) \underset{t \to +\infty}{\longrightarrow} 0\}\\
	W^{u}_X(p)&=&\{q \in M\,: \mbox{dist}(X^t(q),X^t(p)) \underset{t \to -\infty}{\longrightarrow} 0\}.
\end{eqnarray*}	
If  $\mathcal{O}=\mathcal{O}_X(p) \subset \Lambda$ denotes the orbit of $p\in M$, the stable set of the orbit of $p$ is  $W^{s}_X(\mathcal{O})= \underset{t \in \mathbb{R}}{\bigcup}W^{s}(X^t(p))$. Analogously,  the unstable set of the orbit of $p$ is  $W^{u}_X(\mathcal{O})= \underset{t \in \mathbb{R}}{\bigcup}W^{u}(X^t(p))$.

If $\Lambda$ is a hyperbolic set and $p$ is a point in an orbit $\cO$ in $\Lambda$,  then $W^s(p)$ and $W^s(\cO)$ are manifolds of the same regularity as $X$ and are tangent (at $p$) to the stable bundle $E^s(p)$ and $E^{cs}(p)$, respectively. 
If $\Lambda$ is  partially hyperbolic  the stable sets of the points in $\Lambda$ are submanifolds of the same regularity as $X$, tangent to $E^s$ and varying continuously with the point. 
Assume now $U$ is an attracting region and 
that $\Lambda=\Lambda_U$, the maximal invariant  in $U$, is a partially hyperbolic attracting invariant compact set. Let $E^s$ denote the stable bundle defined over $\Lambda$.  Then the bundle  $E^s$ always extends  uniquely to $U$ as an invariant bundle (still denoted by $E^s$) tangent to a topological foliation. Recently Ara\`ujo and Melbourne, \cite[Theorem 4.12 and Remark 4.13(b)]{ararujomelbourne}, provided bunching conditions ensuring the smoothness of the stable foliation, as stated below:

\begin{teo}\label{holonomia}
 \cite[Theorem 4.12]{ararujomelbourne}:	 Let be $M$ a differentiable Riemannian manifold of dimension 3, $X$ a $C^r$ vector field,  $U$ an attracting region and $\Lambda \subset M$ a maximal invariant set in $U$  partial hyperbolic attracting set.  Let  $\{W^s_x\}$ denote the stable foliation in $U$. 
	  Let $q \in [0,r]$ and suppose that there exists $t>0$ such that
$\|DX^t|_{E^s_x}\|\cdot \|DX^{-t}|_{E^{cu}_{X^t(x)}}\|\cdot \|DX^t|_{E_x^{cu}}\|^q<1 \label{eq.cont.}\quad \mbox{ for all $x \in \Lambda$}.$	  
Then the foliation $\{W^s_x\}$ is of class $C^q$.
	\end{teo}

%%%%%%%%%%%%%%%%%%%%%%%%%%%%%%%%%%%%%%%%%%%%%%%%%%%
\subsection{Cone fields}
%%%%%%%%%%%%%%%%%%%%%%%%%%%%%%%%%%%%%%%%%%%%%%%%%%%

\noindent Let $f:M \to M$ be a $C^1$ diffeomorphism and $T_xM = F^u(x) \oplus F^s(x)$ a continuous splitting. We define the stable and unstable cone fields of size $\gamma< 1$ as
\begin{eqnarray*}
	C^s_{\gamma}(x)&=& \{(v_1,v_2) \in F^u(x) \oplus F^s(x)\,: \|v_1\|\le \gamma \|v_2\|\},\\
	C^s_{\gamma}(x)&=&\{(v_1,v_2) \in F^u(x) \oplus F^s(x)\,: \|v_2\| \le \gamma \|v_1\|\}.
\end{eqnarray*}	
We say that $C^s_{\gamma}$ (reciprocally $C^u_{\gamma}(x)$) is strictly invariant by $Df^{-1}$ (reciprocally by $Df$) if  there is $\alpha < 1$ such that $Df^{-1}(C^s_{\gamma}(f(x))) \subset C^s_{\alpha\gamma}(x)$ (reciprocally $Df(C^u_{\gamma}(x)) \subset C^u_{\alpha\gamma}(f(x))$).

$\lambda_i \in \mathbb{R},\, i \in \{ss, s, u\},$ of $DX(\sigma)$ satisfy $0<-\lambda_s<\lambda_{u}<-\lambda_{ss}$.

\subsection{Geometric Lorenz attractor}

{Here we recall the construction of the  geometric Lorenz model, following \cite{GW}. For more on this construction, the interested reader can consult \cite{SP10}, 
and \cite[Section 3.3.2]{{ArPa10}} for a detailed construction of this model.}
{ See Figure 5.}

{ Let $S\subset \bR^3$  and a parametrization such that $S = \{(x, y, 1)|\,\, x,y \in \mathbb{R}\,\,\, \mbox{with}\,\, |x|,|y| \le 1/2\}$.  Let $\Gamma=\{(0, y,1) \in S, |y| \le 1/2\}$ and
$\Sigma^\pm = \{(x, y, z)\,\,x=\pm 1 \}.$ }

Consider a $C^1$-vector field $X$ on $\mathbb{R}^3$ satisfying the  following conditions:
	
\begin{enumerate}
	\item[(1)] {For any point $(x, y,z)$ in a neighborhood of the origin $(0, 0, 0)$, $X$ is given by
	$(x',y',z') = (\lambda_1 x, \lambda_{2} y, \lambda_3 z)$
		where $0<-\lambda_3<\lambda_{1}<-\lambda_{2}$},
\item[(2)] {For all $p \in S \setminus \Gamma$, there exist the smaller positive $t_p$ such that $X^{t_p}(p) \in S$, }
	\item[(3)]  {$S$ is a cross-section to $X$ and the first return map $P:S \setminus \Gamma \to S$ is a piecewise $C^1$-diffeomorphism  given by
		$P(x,y) = (f(x), g(x,y))$, satisfying:}
	{$f:[-1/2,1/2]\setminus \{0\} \to [-1/2,1/2]$ is piecewise $C^1$ map,}
	{$|f'|>\sqrt{2}$ where it is defined,}
	{$f(0^+)=-1/2$ and $f(0^-)=1/2$,}
	{$|f'(0^{\pm})|=+\infty$} and
	{$\partial_y g(x,y) = C x^{\beta}$, for $C>0$ and $\beta>1$.}
\end{enumerate}

{Let $X^t$ the flow of $X$. 
The maximal invariant set $\Lambda= 
\bigcap_{t\leq 0} X^t(S)$ is called a geometric Lorenz attractor. Figure 5 
describes the main features of the construction of a geometric Lorenz attractor.}
  \begin{figure}[!h]
\centering
\includegraphics[width=2in]{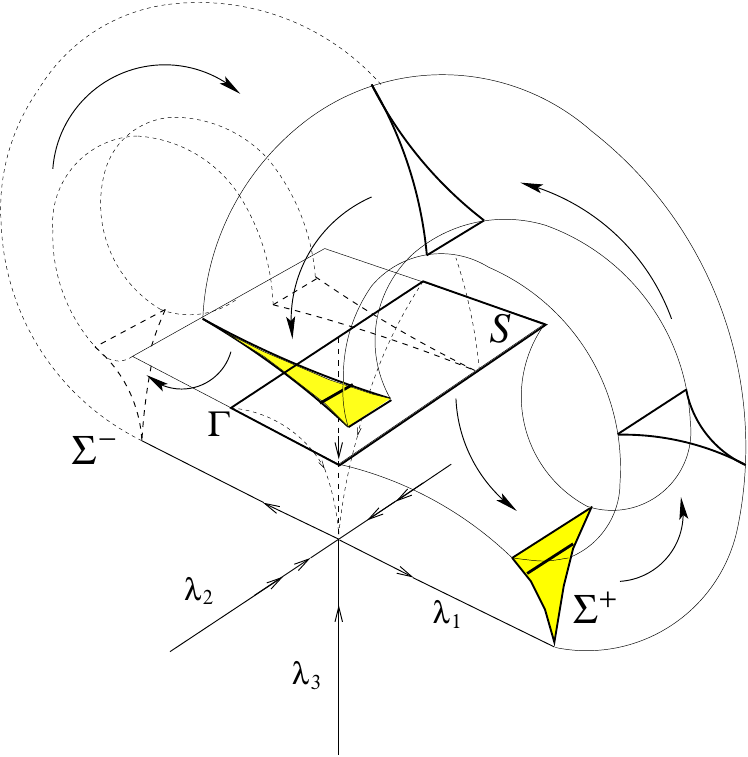}\label{f-Lorenz-Geo}
\hspace{0.4cm}
\includegraphics[scale=0.28]{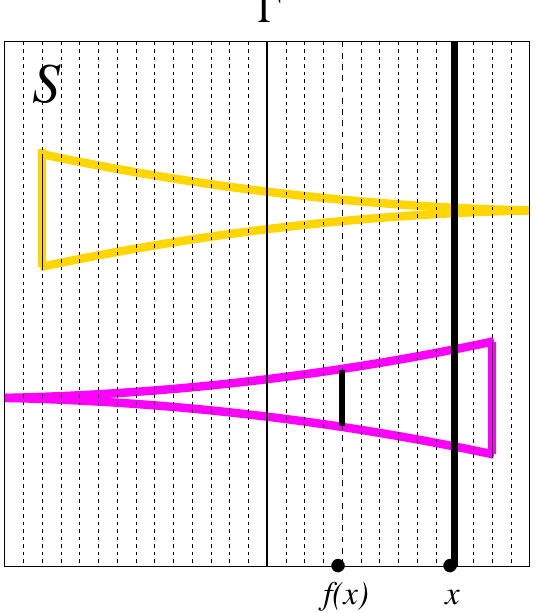}
\hspace{0.3cm}
\includegraphics[scale=0.31]{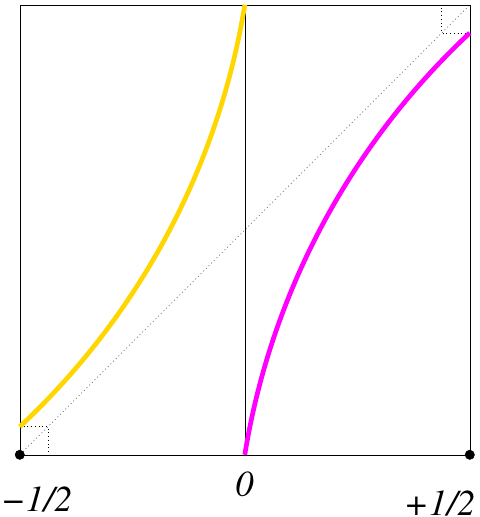}
\label{f-geometric}
\caption{A geometric Lorenz flow, projection on $I$ through the stable leaves and a sketch of the image of one leaf under the return map,  and the 1-dimensional map defined in the space of leaves of $\cF^s$.}
\end{figure}

{
 \begin{defi}\label{fullLorenz}
A positively invariant region of a vector field $X$ contains \emph{a full Lorenz} $\Lambda$ if  it satisfies items (1) and (2) of the definition of Lorenz attractors above, but  item (3) is replaced by 
\begin{enumerate}
\item[(3a)]  The first return map $P:S \setminus \Gamma \to S$ is a piecewise $C^1$-diffeomorphism  defined by
		$P(x,y) = (f(x), g(x,y))$, satisfying:
	$f:[-1/2,1/2]\setminus \{0\} \to [-1/2,1/2]$ is piecewise $C^1$ map,
	$|f'|>\sqrt{2}$ where it is defined,
	$f(0^+)=-1/2$ and $f(0^-)=1/2$,  $f(1/2)=1/2$ $f(-1/2)=(-1/2)$
	$|f'(0^{\pm})|=+\infty$ and
	$\partial_y g(x,y) = C x^{\beta}$, for $C>0$ and $\beta>1$.
\end{enumerate}
\end{defi}
}

%%%%%%%%%%%%%%%%%%%%%%%%%%%%%%%%%%%%%%%%%%%%%%%%
\section{An open set of singular hyperbolic flows}\label{s-the-open-set-1}
%%%%%%%%%%%%%%%%%%%%%%%%%%%%%%%%%%%%%%%%%%%%%%%%

We consider an open set  $\cO_1$ of vector fields $X$ on $\RR^3$ with the following properties described in the two next sections.

\subsection{A geometric model for $X \in \cO_1$}\label{ss-the-open-set-1}

\begin{enumerate}
 \item $X$ has  a Lorenz-like singularity $\sigma= \sigma(X)$ varying continuously with $X$ and satisfying the Sternberg of non-resonance conditions. We set $W^s_+(\sigma)$ and $W^s_-(\sigma)$  for the connected components of $W^s(\sigma)\setminus W^{ss}(\sigma)$. 

  \begin{figure}[h!]
 	\centering
\includegraphics[scale=0.17]{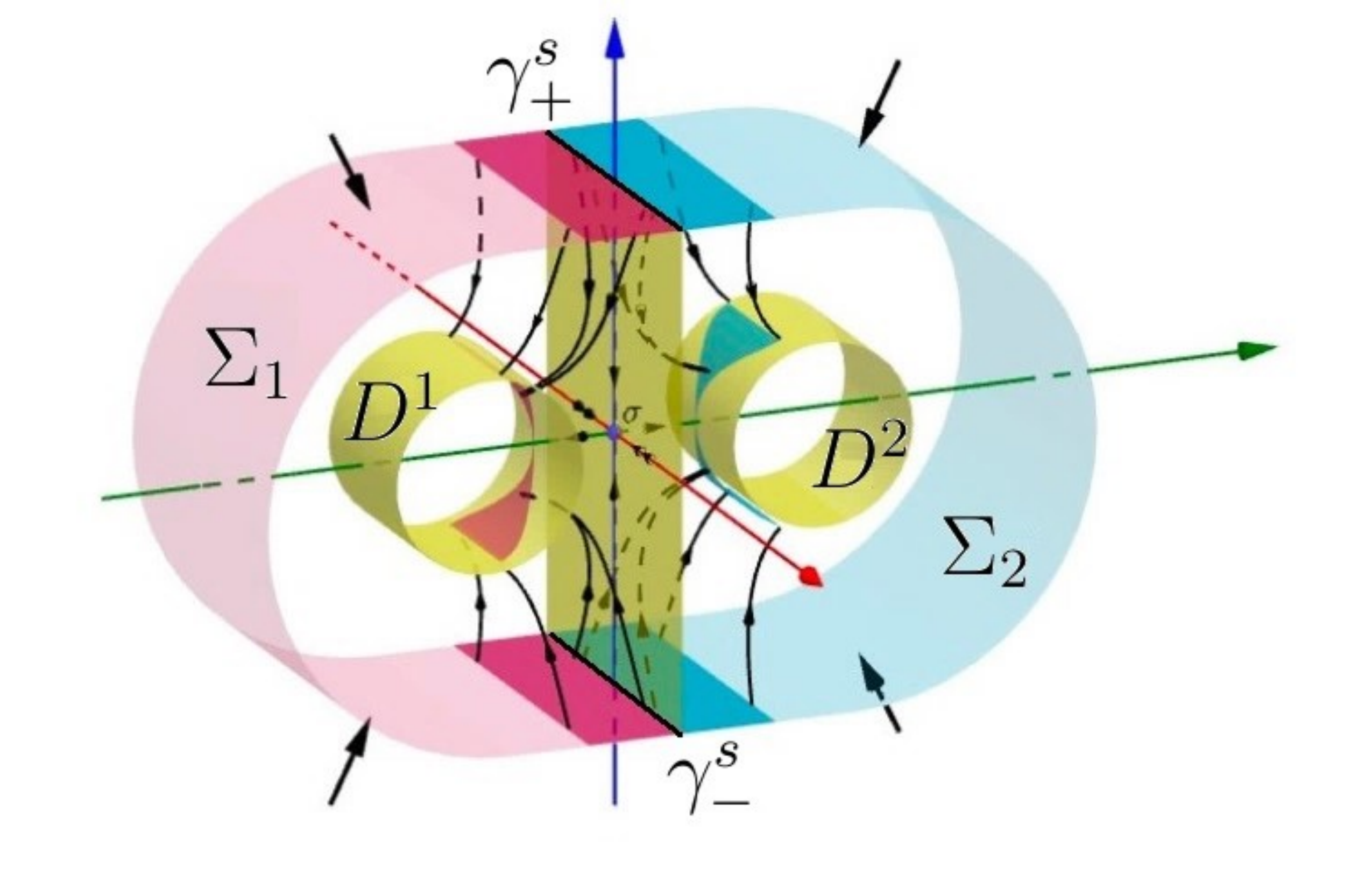}
 	\caption{{Trajectories of $X$  crossing  transversally  $\Sigma$,  $D^1$ and $D^2.$}}
 	\label{figure1}
 \end{figure}
 \item the vector field $X$ is transverse to  $3$ embedded, disjoint annuli denoted by $\Sigma$, $D^1$ and $D^2$, see Figure \ref{figure1}. 
 
 \item for any point $p\in D^i$ there is $t(p)$ depending smoothly on $p$ so that $X^{t(p)}(p)\in \Si$ and $X^s(p)\notin \Sigma\cup D^1\cup D^2$ for $s\in]0,t(p)[$. In  other words, $X^{t(p)}(p)$ is the first return of the orbit of $p$ on $\Sigma\cup D^1\cup D^2$, and we denote it $R(p)=X^{t(p)}(p)$.  
 
 \item Note $R$ is a diffeomorphism from $D^1\cup D^2$ to its image in $\Sigma$. In particular, $R(D^1)$ and $R(D^2)$ are annuli. We require that they are disjoint and each of them is \emph{an essential annulus} in $\Sigma$ (that is, it is not homotopic to a point). 
 
Note that the union of the orbit segments $\bigcup_{p\in D^1}X^{[0,t(p)]}(p)$ is the product of the annulus $D^1$ by a segment. The same holds for $D^2$. 
 
 \item $\Sigma\cap W^s_+(p)$ and $\Sigma\cap W^s_-(\sigma)$ contain each a segment, $\gamma^s_+$ and $\gamma^s_-$, respectively, transverse to the boundary $\partial \Sigma$ and connecting the two boundary components of the annulus $\Sigma$. 
  The positive orbits of points in $\gamma^s_+$ and $\gamma^s_-$ go directly to the singular point and are disjoint from $\Sigma\cup D^1\cup D^2$. 
  Then $\Sigma\setminus ( \gamma^s_+ \cup \gamma^s_-)$ consists  of two components  $\Sigma^1$ and $\Sigma^2.$ See Figure \ref{figure1}.
 
 \item For any $p\in \Sigma^1\cup \Sigma^2$ there is $t(p)$ depending smoothly on $p$ so that $X^{t(p)}(p)\in D^1\cup D^2$ and $X^s(p)\notin \Sigma\cup D^1\cup D^2$ for $s\in]0,t(p)[$. 
 Thus, $X^{t(p)}(p)$ is the first return of the orbit of $p$ on $\Sigma\cup D^1\cup D^2$, and we denote it $S(p)=X^{t(p)}(p)$.  
  Using item 3) one gets that $P=R\circ S\colon \Sigma\setminus(\gamma^s_+ \cup\gamma^s_-)\to \Sigma$ is the first return map of $X$ on $\Sigma$.   
  \begin{figure}[!h]\label{figure7}
 	\centering
	\includegraphics[scale=0.36]{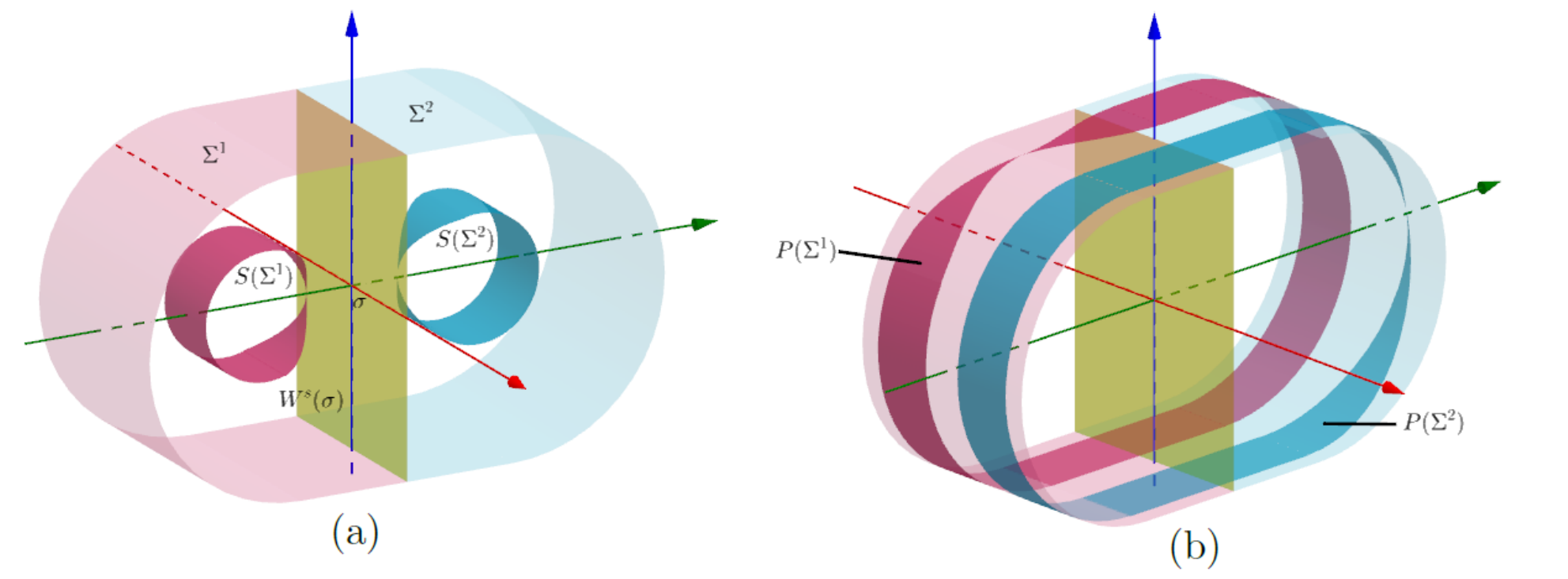}
 	\caption{(a) The  cross-section $\Sigma=\Sigma^1 \cup \Sigma^2, $ the image $S(\Sigma^1)\subset D^1$ and $S(\Sigma^2)\subset D^2$ \hspace{3cm} 
	(b) The image $P(\Sigma^1\cup\Sigma^2)\subset \Sigma$.}
 \end{figure}
 \vspace{-0.2cm}
 
 \item Let $W^u_1(\sigma)$ and $W^u_2(\sigma)$ denote the  unstable separatrices of $\sigma$, that is, the connected components of $W^u(\sigma)\setminus\sigma$. 
 Then the interior of  $D^i$, $1\leq i \leq 2,$ contains the first intersection point $\widetilde{q}_i$ of $W^u_i(\sigma)$ with $\Sigma\cup D^1\cup D^2$. 
 Furthermore, $S(\Sigma^i)\cup \{\widetilde{q}_i\}$ is an annulus pinched at $\widetilde{q}_i$. This pinched annulus is essential in $D^i$ (see Figure 7(a)): by pinched, we mean that in a neighbourhood of $\widetilde{q}_i$, the set  $S(\Sigma^i)\cup \{\widetilde{q}_i\}$ consists of two cuspidal triangles, with a cusp at $\widetilde{q}_i$ and tangent at this point to the same line but oriented in opposite direction.
  By construction, the closure of  $P(\Sigma^1\cup \Sigma^2)$ consists of $2$ parallel essential annuli in $\Sigma$, pinched at $q_i=R(\widetilde{q}_i)$. See Figure 7(b). 
\end{enumerate}

%%%%%%%%%%%%%%%%%%%%%%%%%%%%%%%%%%%%%%%%%%%%%%%
\subsection{Singular hyperbolic conditions}
%%%%%%%%%%%%%%%%%%%%%%%%%%%%%%%%%%%%%%%%%%%%%%%
 
\begin{enumerate}
\setcounter{enumi}{7}
\item The maximal invariant set $\Lambda$ in $U$ is singular hyperbolic with bundles $E^s$ and $E^{cu}$. This is equivalent to  the request that the first return map $P$ is hyperbolic. We will request  more. 

 \item There is a conefield $\cC^u$ on the annulus $\Sigma\simeq \SS^1\times [-1,1]$, transverse to the fibers $\{\theta\}\times[-1,1]$, and strictly invariant by the derivative $DP$ of $P$, and so that the vectors in $\cC^u$ are uniformly expanded by a factor $\lambda>1$.
  The cone $\cC^u(p)$ has two components: vectors cutting the fiber positively or negatively.  
  
\item  There exists $t>0$ such that
$\|DX^t|_{E^s_x}\|\cdot \|DX^{-t}|_{E^{cu}_{X^t(x)}}\|\cdot \|DX^t|_{E_x^{cu}}\|<1 \label{eq.cont.}\quad \mbox{ for all $x \in \Lambda$}.$
\end{enumerate}
As a consequence of item (10) and  \cite[Theorem 4.12]{ararujomelbourne}, 
the stable foliation is well defined in $U$ and it is of class $C^1$.  
 This foliation is not tangent to $\Sigma$. However, there is a $2$ dimensional \emph{center-stable foliation } well defined on $U\setminus \sigma$, obtained by considering the $X^t$ orbits of the stable foliation: this foliation is $C^1$ too. This center stable foliation cuts the annulus $\Sigma$ transversely along a $C^1$, one-dimensional foliation $\cF^s$, which is  the stable foliation of the return map $P$. 
 The segments $\gamma^s_+$ and $\gamma^s_-$ are leaves of $\cF^s$. 
 The foliation $\cF^s$ is transverse to the unstable cone field $\cC^u$.  The leaves of the foliation $\cF^s$ are segments crossing $\Sigma$ (connecting the two boundary components of $\Sigma$). 

%%%%%%%%%%%%%%%%%%%%%%%%%%%%%%%%%%%%%%%%%%%%%%%
\subsection{{$X\in \cO_1$ can be constructed in any $3$-dimensional manifold}}
%%%%%%%%%%%%%%%%%%%%%%%%%%%%%%%%%%%%%%%%%%%%%%%
{To see that  $X$ satisfying itens 1. to 10. can be realized in any $3$-manifold,  we first construct an attracting region $\UU\subset \RR^3$ containing $\Sigma$, $D^1$, $D^2$ and the singularity $\sigma$ so that every regular orbit of $X$ in $\UU$ crosses $\Sigma$ transversally.  The maximal invariant set $\Lambda$ in $\UU$ is an attracting invariant compact set.}

{To construct $\UU$ we proceed as follows.
Let $\DD\subset \CC$ be the unit disk in the complex plane and consider the solid cylinder $\AA_0=\DD \times [-1,1]$ with coordinates $z\in \CC$  and $ t \in [-1,1]$.
Define the solid cylinders 
$$\CC_1= \{ (z, t); \|z\|< 1-\varepsilon, t \in (-1+\varepsilon, -\varepsilon) \} 
\,\,\mbox{and}\,\,
\CC_2= \{ (z, t); \|z\|< 1-\varepsilon, t\in (\varepsilon, 1-\varepsilon) \}.$$}

 \begin{figure}[h!] 
 	\centering
\includegraphics[scale=0.14]{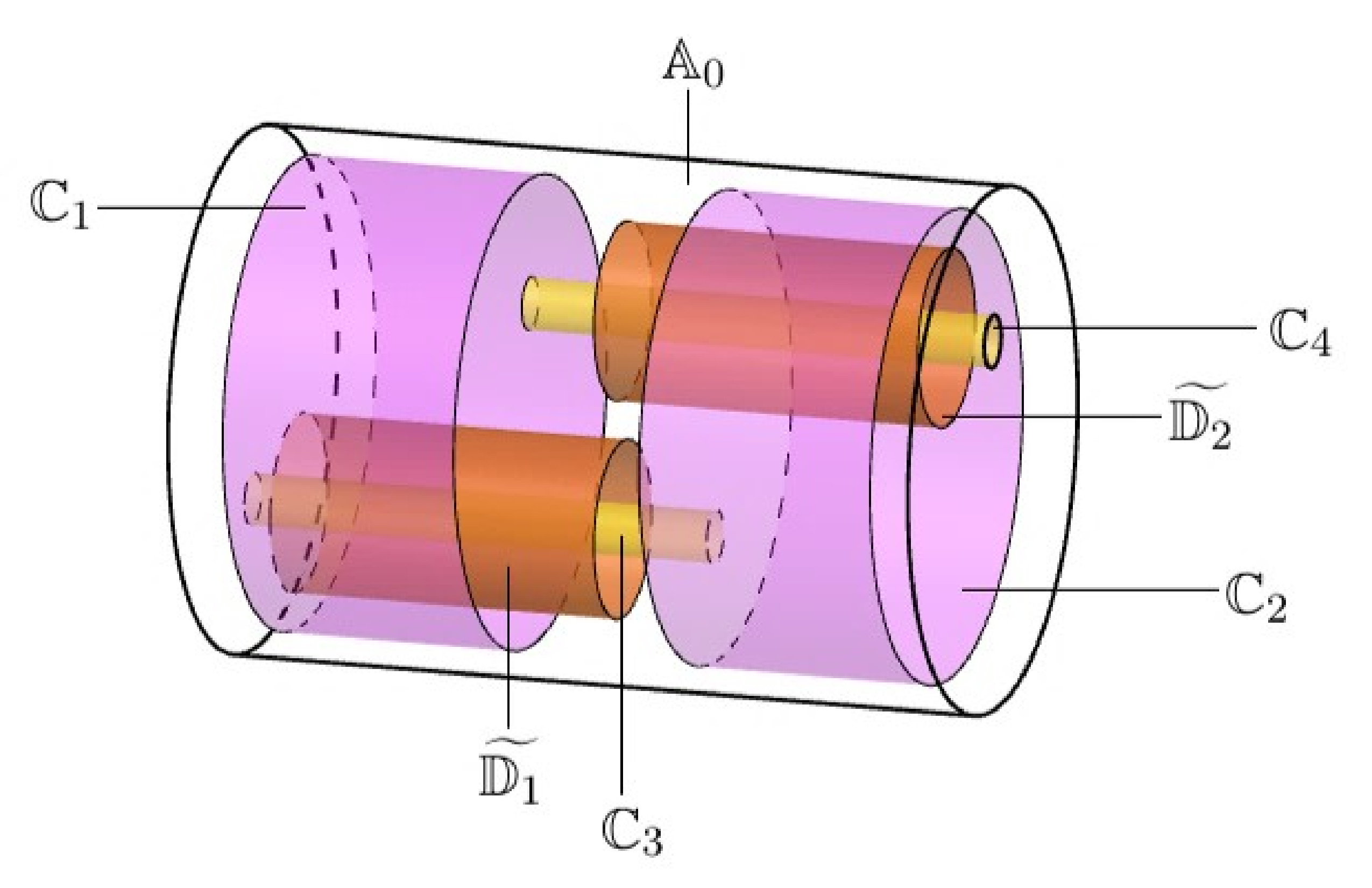} \label{atracting}
\caption{Determining the attracting region}
\end{figure}

{Let $\AA_1=\AA_0\setminus (\CC_1 \cup \CC_2)$ and pick the points
$x=1/2, y=-1/2 \in \CC$. Define
$$
\widetilde{\DD_1}=\{(z,t); \|z-x\|< 1/4,\, t \in (-1 + \varepsilon,0)\} \,\, \mbox{and}\,\,\,\, 
\widetilde{\DD_2}=\{(z,t); \|z-y\|< 1/4,\, t \in (0,1-\varepsilon)\}.
$$}

{Consider $\AA_2= \AA_1 \cup (\widetilde{\DD_1} \cup \widetilde{\DD_2}),$ and 
define
$$
\CC_3=\{(z,t); \|z-x\|< 1/8,\, t \in (-\varepsilon,1)\}, \,\,
\CC_4=\{(z,t); \|z-y\|< 1/4,\, t \in (-1,\varepsilon)\}.
$$
 Consider 
 \begin{equation}\label{region-A}
 \UU=\AA_2 \setminus \big(\CC_3 \cup\CC_4\big).
 \end{equation}}
 
{We define  $X$ in the interior of the region $\UU$  as follows.
We can assume that the singularity $\sigma$ is in the interior of $\UU$ and
that its stable set  $W^s(\sigma)$ intersects 
$\Sigma = \{(z,t) \in \UU \times [-\varepsilon,\varepsilon]\,\,;\,\, \|z\|=1-\varepsilon\}$
into the two  disjoint curves $\gamma^s_-$ and $\gamma^s_+$ that split $\Sigma$ into two regions $\Sigma^1$ and $\Sigma^2$.
Let $D^1=\{(z,t); \|z\|= 1/4, \,\,t \in (-\varepsilon/2,\varepsilon/2)\} $ and $D^2=\{(z,t); \|z\|= 1/4, \,\,t \in (-\varepsilon/2,\varepsilon/2)\}$.}

{Define $X$ in such a way that it crosses inward transversally $\Sigma$ and, for each point 
$p\in \Sigma^1$, (resp. $\Sigma^2$) there is a first positive time $t_p$ such that
$X^{t_p}(p) \in D^1$ (resp. $D^2$). We set, to keep the previous notation, $X^{t_p}(p)\eqdef S(p)$.}
  \begin{figure}[h!]
 	\centering
\includegraphics[scale=0.15]{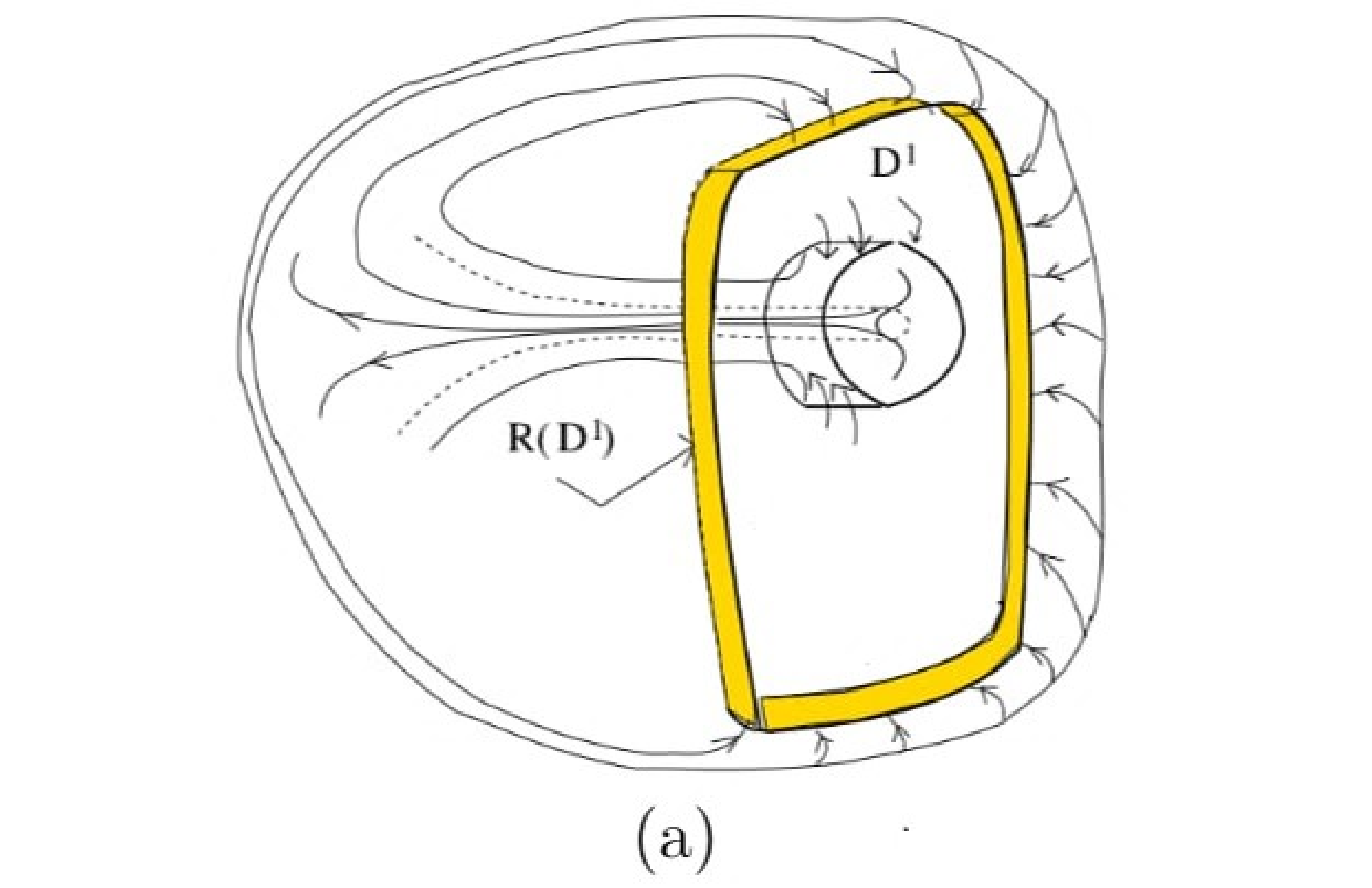} 
\hspace{-2cm}
\includegraphics[scale=0.15]{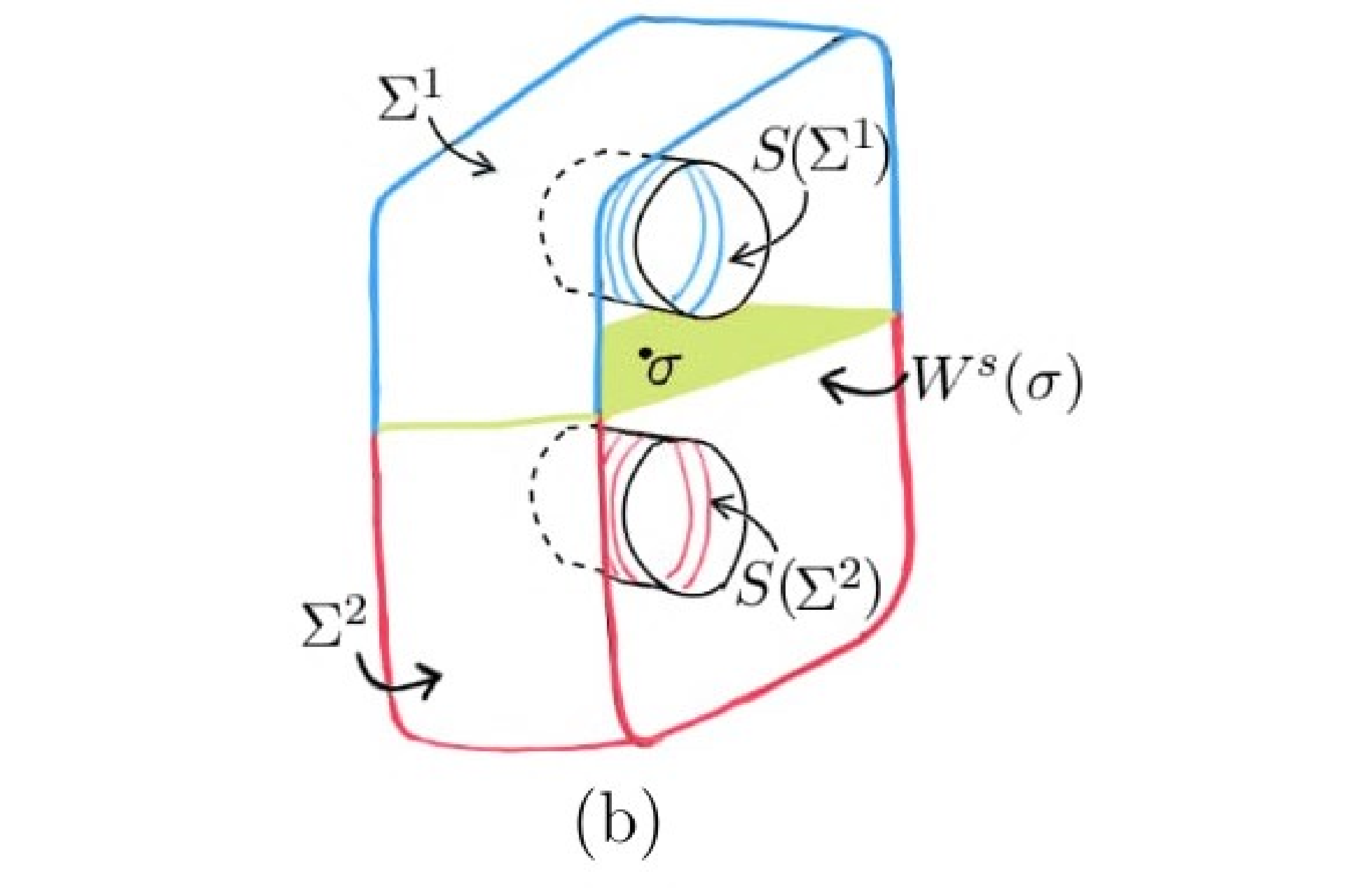}
\hspace{-2cm}
\includegraphics[scale=0.15]{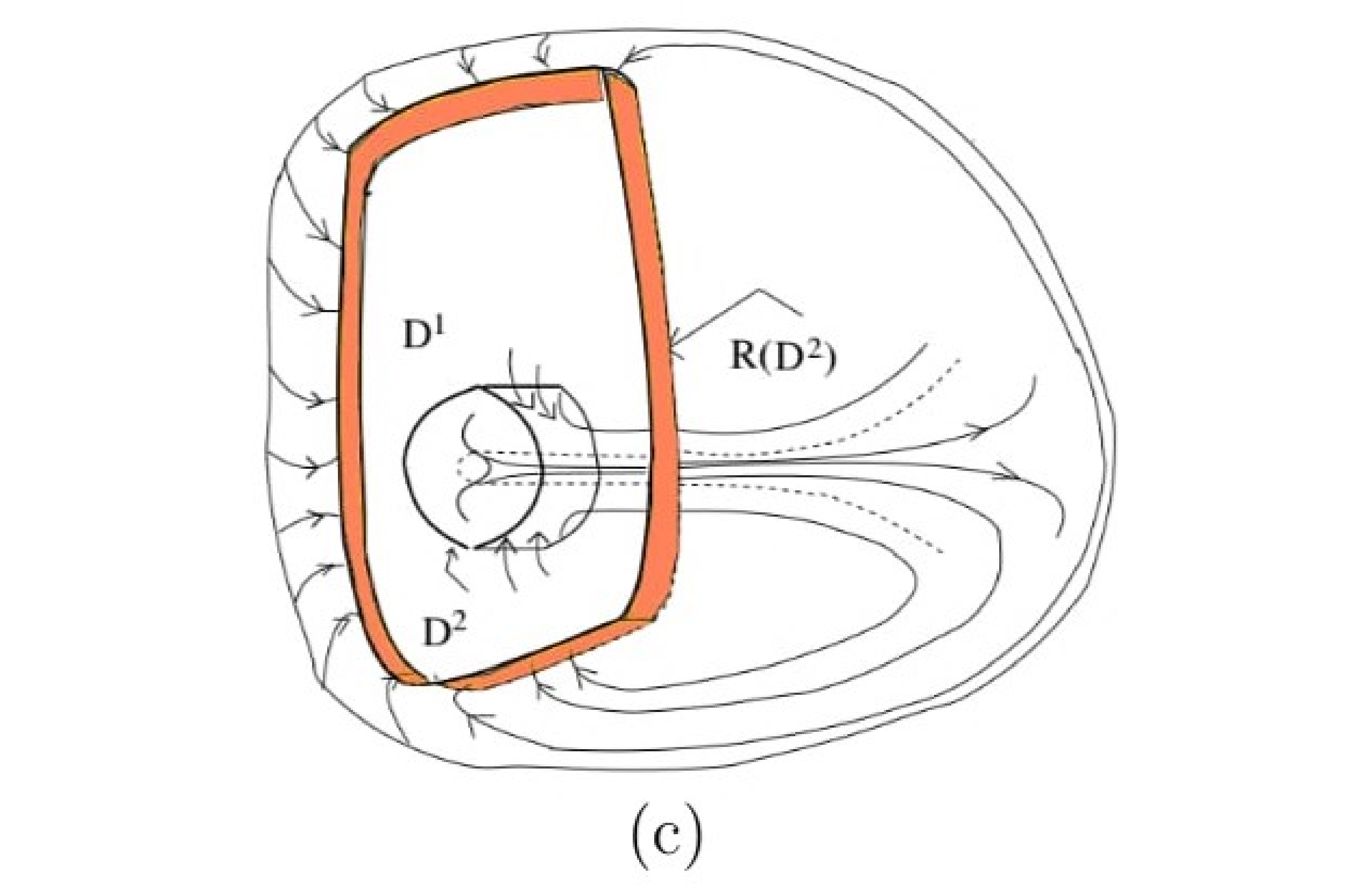}
\caption{(a) Trajectories of $X$  from $D^1$ to $\Sigma$, (b) the cross-section $\Sigma=\Sigma^1\cup \Sigma^2$ and the images $S(\Sigma^1) \subset D^1$ and $S(\Sigma^2) \subset D^2$, (c) Trajectories of $X$  from $D^2$ to $\Sigma$.}
 	\label{vacafria}
 \end{figure}

{Moreover, for each point $\tilde{q}=S(p) \in D^1$ (resp. $D^2$), there is a first positive time $t_{\tilde{q}}$
such that $X^{t_{\tilde{q}}}(q)$ crosses inward  transversally $\Sigma$. Again, to keep the previous notation, we set $X^{t_{\tilde{q}}}(q) \eqdef R(\tilde{q})$.  
This completes 
the definition of $X$ on the cross-section $\Sigma$. 
Clearly $\Sigma$ is a global cross-section to $X$ and the first return map  $P: \Sigma\to \Sigma$ is equal to  $R\circ S$. Then we easily complete the definition of $X$ on $\bigcup_{t\geq 0}\Sigma$ satisfying  all the properties described above.}

{By construction, the region $\UU$ defined in (\ref{region-A}) is homeomorphic to the attracting region for our flow $X^t$, and it is  contained in $\RR^3$.
Figure \ref{vacafria} represents the geometric realization of the union $\bigcup_{p\in D^i}X^{[0,t(p)]}(p)$, where $1\leq i\leq 2$. This illustration demonstrates that this construction can be realized within a ball contained in $\mathbb{R}^3$ and, consequently, within any $3$-dimensional compact manifold.}  \hspace{14cm} $\square$

 %%%%%%%%%%%%%%%%%%%%%%%%%%%%%%%%%%%%%%%%%%%%%%%%
\section{Topological dynamics: the attractor and the chain recurrence classes}\label{s-the-open-set-2}

%%%%%%%%%%%%%%%%%%%%%%%%%%%%%%%%%%%%%%%%%%%%%%%%

\begin{defi}\begin{itemize}\item We say that $X\in \cO_1$ exhibits a \emph{two-sided (geometric model)  Lorenz attractor} if, for any $Y$ in a neighbourhood of $X$, the maximal invariant set $\Lambda_Y$ in the attracting region $U$ is transitive and has a non-trivial intersection with both {stable sets} $W^s_+(\sigma,Y)$ and $W^s_-(\sigma, Y)$. 
\item We say that $X$ exhibits an \emph{upper-Lorenz attractor} if it admits a geometric model of Lorenz attractor  disjoint from the component $W^s_-$. 
\item We say that $X$ exhibits a \emph{down-Lorenz attractor} if it admits a geometric model of Lorenz attractor  disjoint from the component $W^s_+$.
\end{itemize}
\end{defi}

This section aims to show that, under a certain condition on the expansion in the unstable direction, the vector fields $X$ in an open and dense subset exhibit either a two-sided or an upper- or a down-Lorenz attractor, and the $3$ cases appear.  

%%%%%%%%%%%%%%%%%%%%%%%%%%%%%%%%%%%%%
\subsection{Quasi-attractor}
%%%%%%%%%%%%%%%%%%%%%%%%%%%%%%%%%%%%%%%

We started by noticing that there is non-ambiguity on what can be the attractor: 

\begin{prop}\label{p.quase-atrator} For any $X\in \cO_1$, the chain recurrence class of the singularity $\sigma$ is the unique quasi-attractor in the attracting region $U.$
\end{prop}
 As $U$ is an attracting region, Conley's theory implies that $U$ contains at least one 
 quasi-attractor. 
 Now Proposition~\ref{p.quase-atrator} is a direct consequence of the following lemma:
 \begin{lema}\label{l.quase-atrator} The stable manifold of $\sigma$ is dense in $U.$
 \end{lema}
 
 This proof is identical to the similar statement for the classical geometric model of the Lorenz attractor and is very simple. We include it here for completeness.
 \begin{proof} As any orbit cuts the cross-section $\Sigma$, we only need to prove that, for any open set $O\subset \Sigma$, there
  is $n>0$ so that $P^n(O)\cap (\gamma^s_+\cup \gamma^s_-)\neq \emptyset$.
  
  We consider a segment $S\subset O$ {such that for any $x\in S$ the tangent vector at $x\in S$ belongs to the cone} $\cC^u$. By item 9) there is $\lambda>1$ so that the vector in $\cC^u$ is expanded by a factor larger than $\lambda$. Thus:
  \begin{itemize}
   \item either $S$ cuts $(\gamma^s_+\cup \gamma^s_-)$ and we are done, 
   \item or $P(S)$ is a segment tangent to $\cC^u$ whose length is at least $\lambda$ times the length of $S$.
  \end{itemize}
  
  Iterating the process, one gets that either there is $n$ so that  $P^n(S)$ cuts $\gamma^s_+\cup \gamma^s_- $ or $P^n(S)$ is a segment tangent to $\cC^u)$ whose length tends to infinity. 
  As $\gamma^s_+$ and $\gamma^s_+$ are leaves of the foliation $\cF^s$, which is a fibration by segments on the annulus $\Sigma$,
   there is a bound for the length of any segment in the unstable cone which does not cut every leaf, ending the proof.
 \end{proof}
 %%%%%%%%%%%%%%%%%%%%%%%%%%%%%%%%%%%%%%%%%%%%%%%%%%
 \subsection{Iterating unstable segments and transitive properties}
 %%%%%%%%%%%%%%%%%%%%%%%%%%%%%%%%%%%%%%%%%%%%%%%%%%%%%%%%%%%%%%
 We fix a vector field $X\in\cO_1$, $U$  is the attracting region in the definition, $\Sigma$ is the global cross-section, and $P$ is the first return map. 
 First, notice that the closures $\bar P^n(\Sigma)$, $n\geq 0$, is a nested family of compact subsets of $\Sigma$.
 Let $\Lambda_P$ denote the intersection of the $\bar  P^n(\Sigma)$
 $$\Lambda_P=\bigcap_{n\geq 0} \bar P^n(\Sigma).$$
 
 \begin{lema} The compact set $\Lambda_P$ is the intersection of the maximal invariant set $\Lambda_X$ with $\Sigma$.  Furthermore, $\Lambda_X$ is the union of $\sigma$ with the saturated  flow of $\Lambda_P$. 
 \end{lema}
 \noindent {\em{Proof}}. Any orbit of $\Lambda_X$, but $\sigma$, cuts $\Sigma$.  Any orbit but $W^u(\sigma)$ cut $\Sigma$ for an infinite sequence of negative times (tending to $-\infty$): indeed the $\alpha$-limit set of such an orbit is not reduced to $\sigma$, and therefore cuts $\Sigma$. 
   Thus every orbit  of $\Lambda_X$ not in $W^u(\sigma)$ cuts $\Sigma$ in $\bigcap_{n\geq 0} P^n(\Sigma)$. 
    In $\Lambda_P$ we consider the closures $\bar P^n(\Sigma)$.  This consists in adding $\Sigma\cap W^u(\sigma)$ to $\bigcap_{n\geq 0} P^n(\Sigma)$. The proof follows from these observations. $\square$
 
 Note that for every $n>0$, the closure $\bar P^n(\Sigma)$ consists in the union of $2^n$ pinched essential annuli, each annulus admits a finite set of pinched (cuspidal) points, and the boundary is tangent to the unstable cone. 
 The intersection of two annuli is contained in this finite set of pinched points. 
 Finally, the thickness of these annuli is exponentially decreasing with $n$. 
 The intersection is an essential circle for any nested sequence of such pinched annuli in 
$\bar P^n(\Sigma)$, $n>0$. 
In particular, the topological dimension of $\Lambda_P$ is $1$. 
 Moreover, the boundary of any annulus is tangent to the unstable cone, and the image by $DP^n$ of this unstable cone.  The intersection of the iterates of the  unstable cones converges to a well-defined continuous, invariant line field $E^u$ on $\Lambda_P$. Each of the circles is tangent to $E^u$. 

 \begin{lema} Let $X$ be a vector field in $\cO_1$ and $P$ the first return map on the cross-section $\Sigma$.  Suppose for every segment  $S\subset \Sigma$ transverse to the stable foliation $\cF^s$, the union of stable leaves through  the iterates $P^n(S)$, $n\geq 0$, covers an open and dense subset of $\Sigma$. Then, the maximal invariant set $\Lambda_X$ is transitive.  
 \end{lema}
 
 One easily checks that is enough to prove the following 
 \begin{lema} Under the same hypothesis, given any non-empty open subsets $U\cap \Lambda_P$, $V\cap\Lambda_P$ ($U$ and $V$ open sets of $\Sigma$) there is $n>0$ so that 
 $P^n(V\cap \Lambda_P)\cap (U\cap \Lambda_P)\neq \emptyset.$
  
 \end{lema}

 \begin{proof} The open subset $U\cap \Lambda_P$ contains a segment $S_U$ tangent to the unstable {cone.}  There is $\varepsilon >0$ so that the segments of stable leaves through $S_U$ (shrinking $S_U$ if necessary) are contained in $U$. Let $W^s_\varepsilon(S_U)$ be the union of these segments. 
 Note that for any $n>0$, $P^{-n}$ is defined on $S_U$ up to a finite set (the first $n$ iterates of the unstable manifold of $\sigma$).  Now $P^{-n}(W^s_\varepsilon(S_U))$ contains the saturated by $\cF^s$ of $P^{-n}(S_U)$, which contains open segments in $\Lambda_P$. 
 Consider now a segment $S_V$ contained in $V\cap \Lambda_P$.  By assumption, there is $m>0$ so that $P^m(S_V)$ intersects stable leaves in $P^{-n}(W^s_\varepsilon(S_U))$.  In other words, $P^{m+n}(S_V)$ cuts $W^s_\varepsilon(S_U)$ at points that belong to $\Lambda_\Sigma$ (because $S_V$ is contained in $\Lambda_\Sigma$, which is positively invariant by $P$). Thus these intersection points belong to $U\cap \Lambda_\Sigma$ concluding the proof. 
 \end{proof}
 
 We now present  tools to ensure that the orbit of an \emph{unstable segment} (i.e. a segment in $\Sigma$ tangent to the unstable cone ) cuts almost every stable leaf.

 \begin{lema}\label{l.cusptogamma}  Let $S$ be an unstable segment joining a cuspidal point $q_i$ in $\bar P(\Sigma)$ to a point $q_\pm$ in $\gamma^s_\pm$. Assume that   $q_i$ belongs to $\Sigma_i$. 
 Then there is $n>0$ so that the union of the iterates $P^i(S)$, $i\in\{0,\dots,n\}$ cut every stable leaf but a finite number. 
  
 \end{lema}
 \begin{proof}The unstable cone is oriented, inducing an orientation over every unstable segment. Assume for instance that $S$ is starting (for this orientation) at $q_i\subset\Sigma_1$ and  ends on $q_-\in \gamma^s_-$: in particular, $S\subset \Sigma^1$.  Now, $P(S)$ is an unstable segment ending at $q_1$ and of length $\ell(P(S)\geq \lambda \ell(S)$. 
 So $S_1= P(S)\cup S$ is an unstable segment of length at least twice $\ell(S)$ and ending at $q_-$. 
 
 We define by induction a finite sequence $S_n$ as follows: if $S_{n-1}$ is contained in {$\Sigma^1$} then $S_n=P(S_{n-1})\cup S$. Otherwise, the sequence ends and $S_n$ is not defined. 
 Thus for every $n$, $S_n$ is an unstable segment ending at $q_-$ and of size at least $n$ times the size of $S$. In particular, this sequence ends at some $n_0$, and  $S_{n_0}$ is not contained in {$\Sigma^1$}: it cuts $\gamma^s_+$.
 Now $P(S_{n_0})$ cuts every stable leaf up to one of $q_1$. 
 \end{proof}

 \begin{rema}In Lemma~\ref{l.cusptogamma}, the unique stable leaf which may not intersect the iterates  $P^n (S)$, $n\geq 0$, is the leaf through $q_1$. Furthermore, if $q_1$ is distinct from $q_2$, then further iterates of $S$ cut $q_1$ so that every leaf intersects some iterate of $S$. 
  
 \end{rema}
 
 \begin{lema}\label{l.cortaestavefixo}  Assume now that $\Sigma_i$ contains a fixed point $p_i$ of $P$.  Let $S$ be an unstable segment  whose interior cuts the stable leaf through $p_i$.  Then there is $n_0$ so that for any $n\geq n_0$,   the iterate $P^n(S)$ cuts every stable leaf but the one through 
 $q_i$.  
 \end{lema}
 \begin{proof}By the inclination lemma (also known as $\lambda$-lemma), the positive iterates $P^n(S)$ accumulate the unstable manifold $W^u(p_i)$. Some of these iterates will cross completely {$\Sigma^i$} (crossing both $\gamma^s_+$ and $\gamma^s_-$).  Then the next iterate will cross the whole essential pinched annulus $P(\Sigma_i)$, ending the proof. 
 \end{proof}

 %%%%%%%%%%%%%%%%%%%%%%%%%%%%%%%%%%%%%%%
 \subsection{Fixing the expansion rate bigger than the golden number $\varphi= \frac{1+\sqrt5}2$}\label{s-golden}
 %%%%%%%%%%%%%%%%%%%%%%%%%%%%%%%%%%%%%%%
 For any flow $X\in \cO_1$, the return map $P$ is defined on the cross-section $\Sigma$ but  the two stable leaves $\gamma^s_+$ and $\gamma^s_-$.  
 So we get two rectangles and  send them  into $\Sigma$ as a pinched essential annulus.  Therefore the expansion rate $\lambda$ by $P$ in the unstable cone cannot be required uniformly larger than $2$.  
 However, we will see that we can require a uniform rate expansion arbitrarily close to $2$. Figure \ref{f.fixingexpansion} displays the main features of the return map $P$.
 
 \begin{figure}[th]
\centering
	\includegraphics[scale=.17]{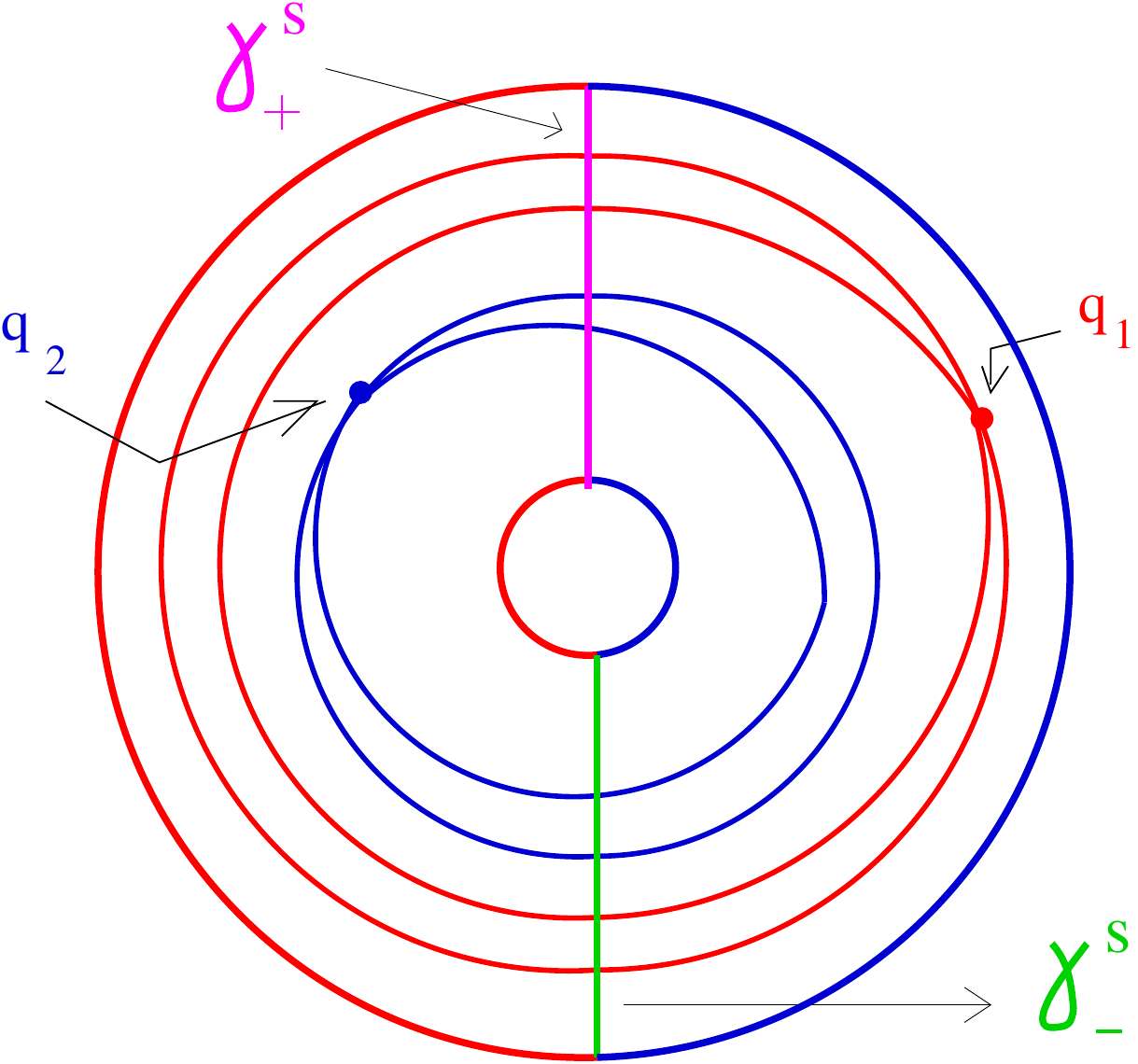}
\caption{Another view of the cross-section $\Sigma$ and the image of $\Sigma$ under the Poincar\'e map. }\label{f.fixingexpansion}
\end{figure}	
 
The proof of the main topological properties consists of iterating unstable segments (tangent to the unstable cone) and estimating the length of these iterated components. 
In this section, we will choose a rate $\lambda$ which will allow us to estimate these lengths. 
{Let us start with a very elementary observation: 
Consider a segment $[a,c]\subset \RR$ and pick $b\in ]a,c[$. Then one of the lengths 
$\lambda^2\ell([a,b])$ or $\lambda^2\ell([b, c])$ is strictly larger than $\ell([a,c]$ if $\lambda>\sqrt{2}$.  
The choice of $\lambda>\sqrt2$ ensures the increase of a segment split into two components after two iterations if those components are not split again. This is the standard  rate of expansion to guarantee the transitivity of a Lorenz attractor. 
We note that there are examples of one-dimensional Lorenz maps with a rate of expansion smaller than $\sqrt{2}$ that are not transitive, see \cite[pp. 126]{ArPa10}. 
\noindent Consider the maximum of the lengths $\lambda\ell( [a,b])$  and $\lambda^2\ell([b,c])$.  Notice that the golden ratio $\varphi= \frac{1+\sqrt5}2\in]1,2[$. 
\begin{lema}\label{l.golden}{ For any $a<b<c$, for any $\lambda\geq \varphi$ on gets}
{$\max\{\lambda\ell([a,b]),\lambda^2\ell([b,c])\}\geq \frac{\lambda}{\varphi}\ell([a,c]).$}
\end{lema}
\begin{proof}
	Taking $\alpha=\frac{\ell([b,c])}{\ell([a,c])}$, we only have to show that
	$\max\{(1-\alpha),\lambda\alpha \}\geq \frac{1}{\varphi}.$
Case $\alpha \ge \frac{\varphi - 1}{\varphi}$, we have $\lambda \alpha \ge \lambda \left(\frac{\varphi -1 }{\varphi} \right)\ge \varphi\left(\frac{\varphi -1 }{\varphi} \right)\ge \frac{\varphi^2-\varphi}{\varphi}\ge \frac{1}{\varphi}$ and we are done. 
On the other hand, if $\alpha < \frac{\varphi - 1}{\varphi}$, we get $1-\varphi > 1- \left(\frac{\varphi - 1}{\varphi}\right) = \frac{1}{\varphi}$ and we are also done.
\end{proof}
 
\noindent We consider now the open subset $\cO_\varphi\subset \cO_1$ consisting of flows $X$ so that the expansion rate $\lambda$ of the return map in the unstable cone  is larger than 
 $\varphi$. 
 %%%%%%%%%%%%%%%%%%%%%%%%%%%%%%%%%%%%%%%%%%%%%%%%
 \subsection{Cutting $\cO_\varphi$ in regions}\label{ss.region}
 %%%%%%%%%%%%%%%%%%%%%%%%%%%%%%%%%%%%%%%%%%%%%%%%
 Consider $\cH^i\subset \cO_1$, $i=1, 2$, the subset corresponding to the vector fields $X$ for which the point $q_i$ (first intersection with $\Sigma$ of the unstable separatrix $W^u_i(\sigma)$) belongs to $\gamma^s_+\cup\gamma^s_-$. 
 In other words, $X$ belongs to $\cH^i$ if $\sigma$ admits a homoclinic loop for the unstable separatrix $W^u_i(\sigma)$ so that the homoclinic loop cuts $\Sigma$ at a unique point $q_i$.  
 We split each $\cH^i$ into two $\cH^i=\cH^i_+\cup \cH^i_-$ according to $\{q_i\in\gamma^s_+\}$ and $\{q_i\in\gamma^s_-\}$, respectively. Figure \ref{f.p.Hi1mais} displays the feature of a vector field $X \in \cH^1$.
 It is well known that a homoclinic connection corresponds to a codimension $1$ phenomenon, as expressed below: 
 
 \begin{lema}\label{l.homo} The subsets $\cH^i$ are codimension $1$ submanifolds of $\cO_1$. 
  \end{lema}
  
  \begin{figure}[th]
\centering
	\includegraphics[scale=0.17]{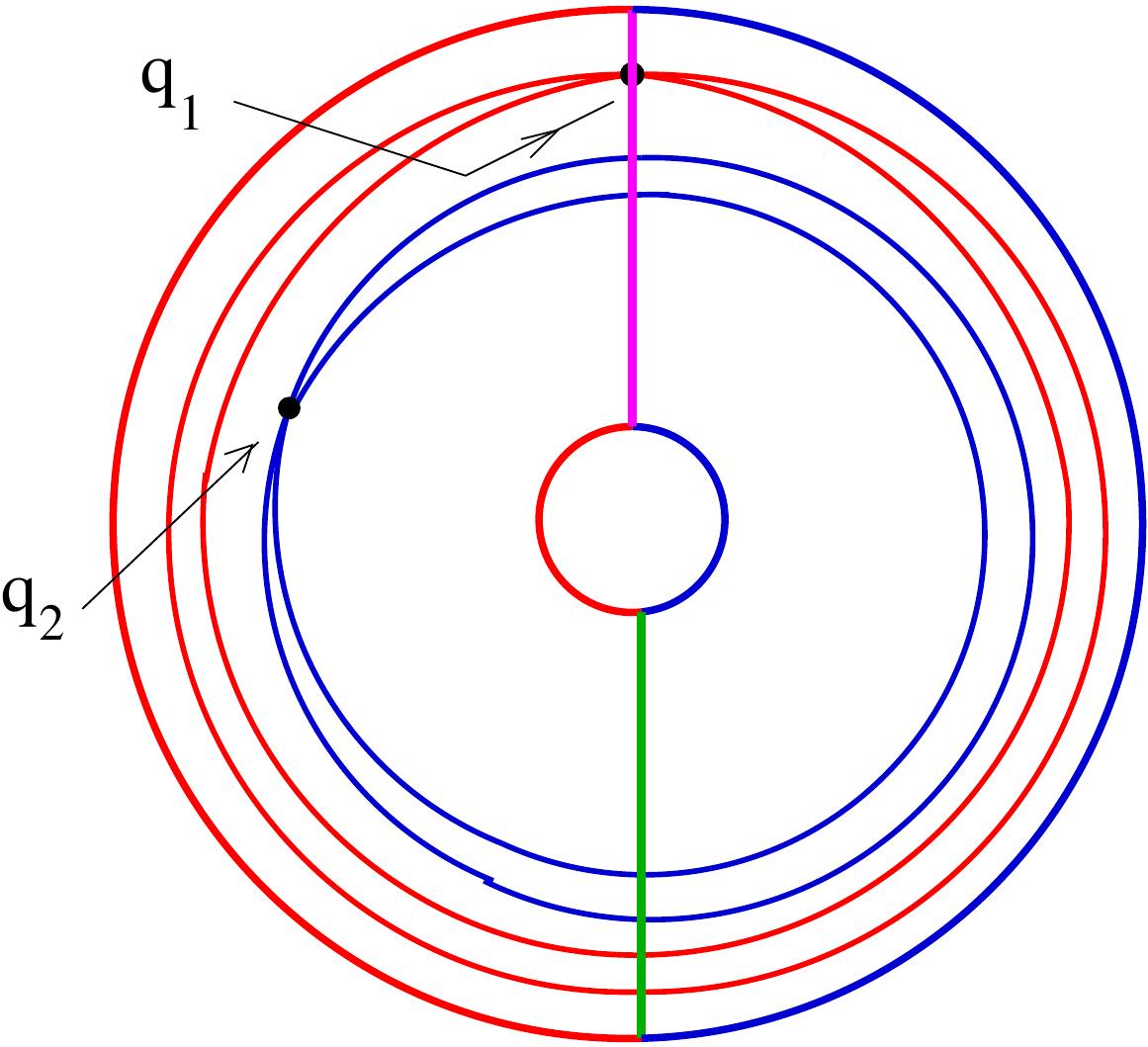}
		\hspace{0.4cm}
\includegraphics[scale=0.17]{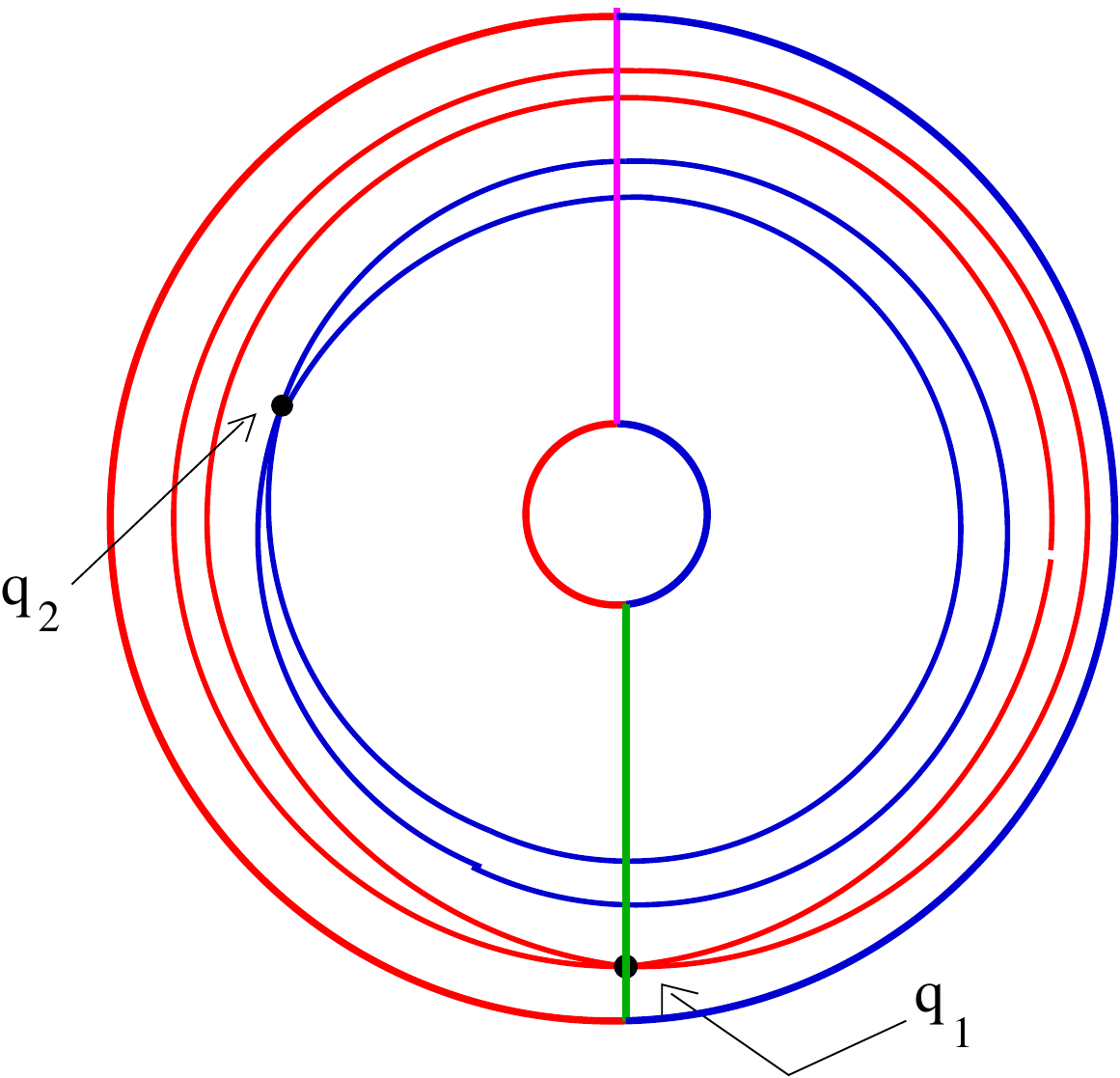}
\caption{(a) $ q_1 \in \gamma_+^s$ implying that $X \in \cH^1_+ $ and (b) $ q_1 \in \gamma_-^s$ implying that $X \in \cH^1_- $.}
\label{f.p.Hi1mais}
\end{figure}
  
\noindent The submanifolds $\cH^1$ and $\cH^2$ cut $\cO_\varphi$ in $4$ regions $\cO_\varphi^{\omega_1,\omega_2}$,
  $\omega_i\in\{+,-\}$, so that $\omega_i=-$ if and only if  $q_i\in \Sigma^i$.  
  Then: 
  
  \begin{lema}\label{f-Lema-Out-H1eH2} Let $X$ be a vector field in $\cO_\varphi$ out of $\cH^1$ and $\cH^2$.
   Then the first return map $P$ admits a  fixed point in {$\Sigma^1$  (resp. $\Sigma^2$)} if and only if $X$ belongs to $\cO_\varphi^{+,-}\cup \cO_\varphi^{+,+}$ (resp. $\cO_\varphi^{-,+}\cup \cO_\varphi^{+,+}$). 
  \end{lema}
  
  \begin{proof} If {$q_1\in\Sigma^2$} means that the cuspidal point of the pinched annulus {$P(\Sigma^1)$ belongs to $\Sigma^2$. } Thus $P(\Sigma^1)$ crosses $\Sigma^1$ in a hyperbolic way: $\Sigma^1$ is a rectangle with a stable boundary $\gamma^s_\pm$ and its image $P(\Sigma^1)$ crosses completely $\Sigma^1$ along a sub-rectangle
  with a stable side contained in each of $\gamma^s_\pm$,
    leading to a unique fixed point in $\Sigma^1$, see Figure \ref{f-Lema-Out-H1eH2}(a). 
  If $q_1\in \Sigma^1$ then $\Sigma^1\cap P(\Sigma^1)$ has $2$ connected components which are cuspidal triangles $T_1^+$ (bounded by $\gamma^s_+$), $T_1^-$ (bounded by $\gamma^s_-$).  Assume that there is a fixed point $p$ in $T_1^-$ (for instance, the other case being equivalent).  Then the region $R$ bounded by $\gamma^s_+$ and $W^s(p)$ is invariant by $P$, see Figure \ref{f-Lema-Out-H1eH2}(b). 
  But $W^u(p)$ contains a segment joining $p$ to $\gamma^s_+$.  This segment is expanded by $P$, contradicting the invariance of $R$. This ends the proof.
   \begin{figure}[th]
\centering
	\includegraphics[scale=0.16]{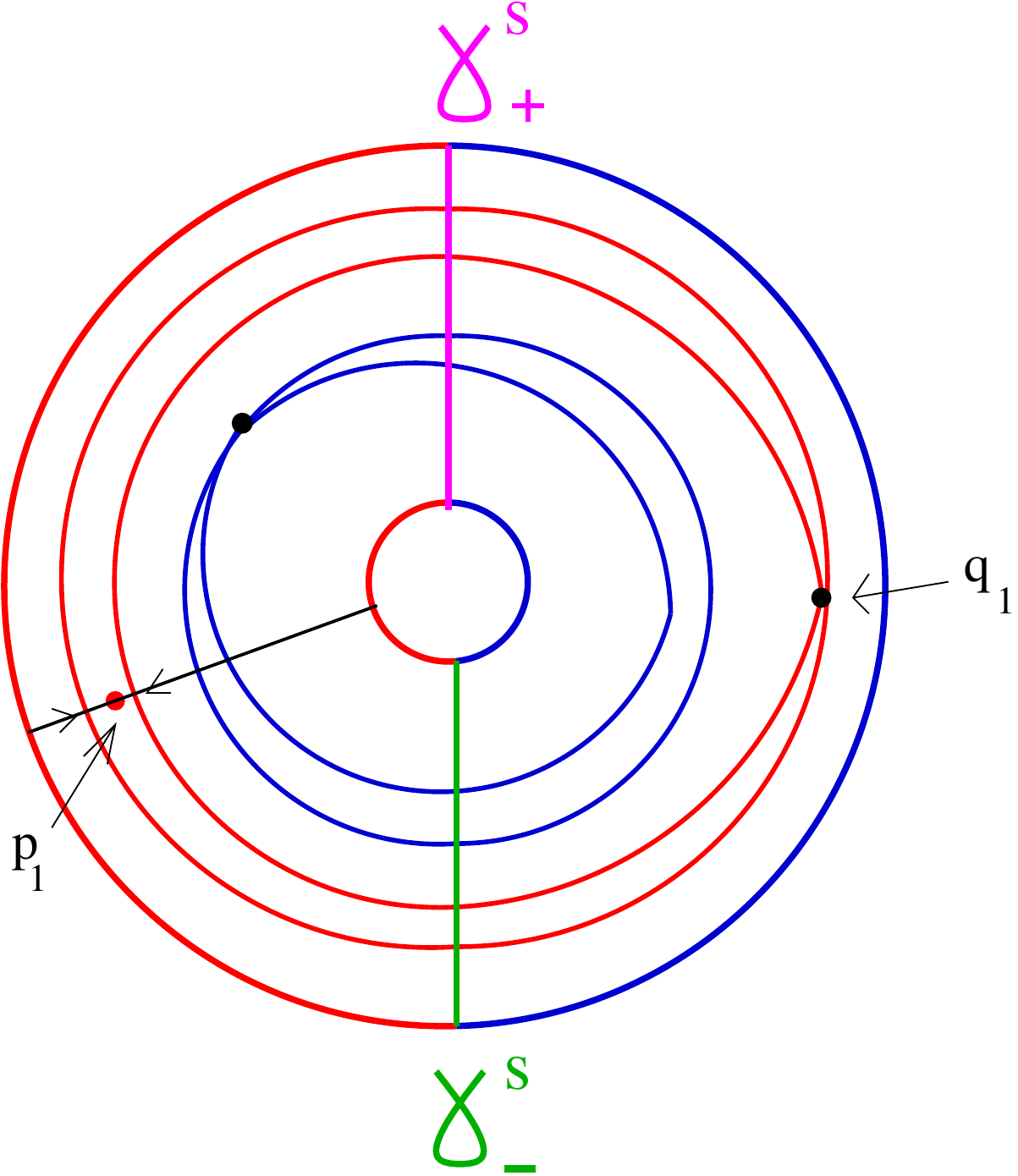}
		\hspace{0.2cm}
\includegraphics[scale=0.16]{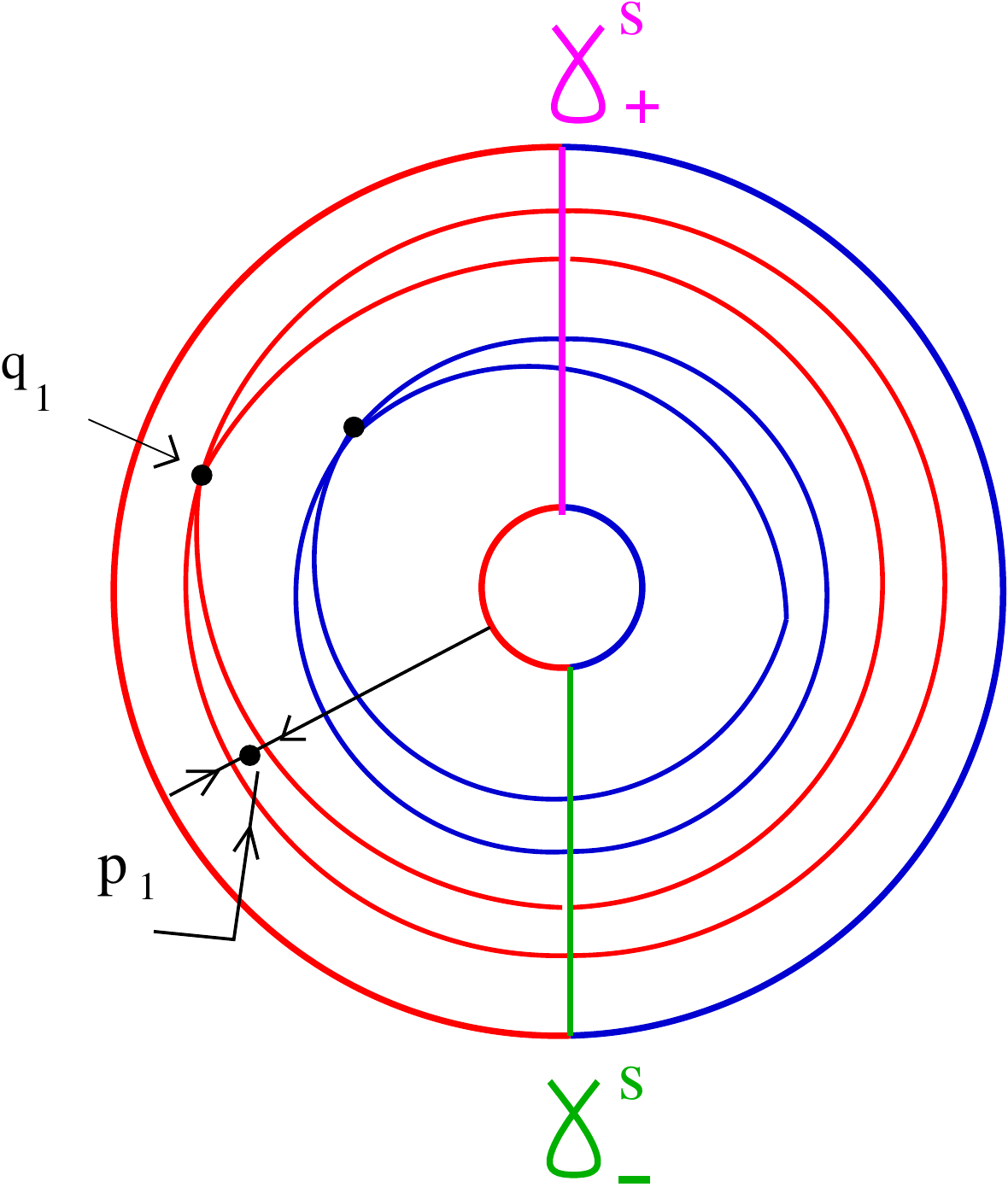}
\caption{$ p_1$ and $p_2$ are fixed points of $X \in \cO_{\varphi}\setminus \cH^{1}\cup \cH^{2}$ }
\label{f-Lema-Out-H1eH2}
\end{figure}
  \end{proof}
\vspace{-0.5cm}

Consider now the region $\cO_\varphi^{+,+}$: every vector field $X\in \cO_\varphi^{+,+}$ has exactly $1$ fixed point $p_1$ in {$\Sigma^1$ and $1$ fixed point $p_2$ in $\Sigma^2$.} Let $W^s_i$ be the stable leaf through $p_i$. 
Note that $q_1$ and $p_2$ are both in $\Sigma^2$ and in the same way $q_2$ and $p_1$ are both in $\Sigma^1$. 
This remark allows, \emph{a priori}, that $q_1$ and $p_2$ belong to the same stable leaf or that $q_2$ and $p_1$ belong to the same stable leaf.  This corresponds to our next splitting of the region $\cO_\varphi^{+,+}$.
Let us denote $\cH\cE_i$ the subset of $\cO_\varphi^{+,+}$ corresponding to the vector fields $X$ for which $q_i\in W^s_j$.  In other words, for $X\in \cH\cE_i$, the unstable separatrix of $\sigma$ corresponding to $q_i$ is a heteroclinic connection with $p_j$, see Figure \ref{f-HE1eHE2}.  As in Lemma \ref{l.homo}, one has:

    \begin{figure}[h!]
\centering
	\includegraphics[scale=0.15]{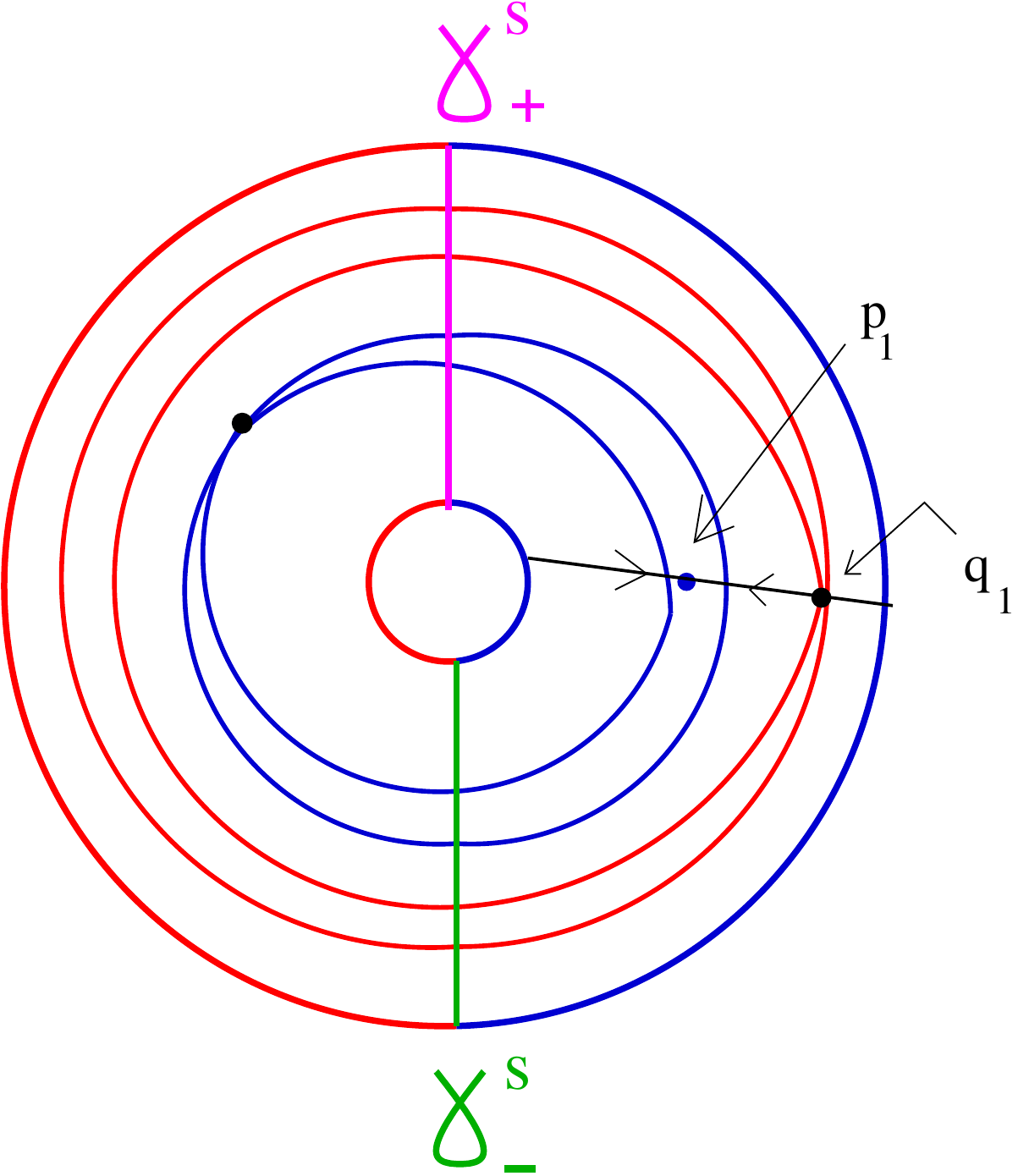}
		\hspace{0.6cm}
\includegraphics[scale=0.15]{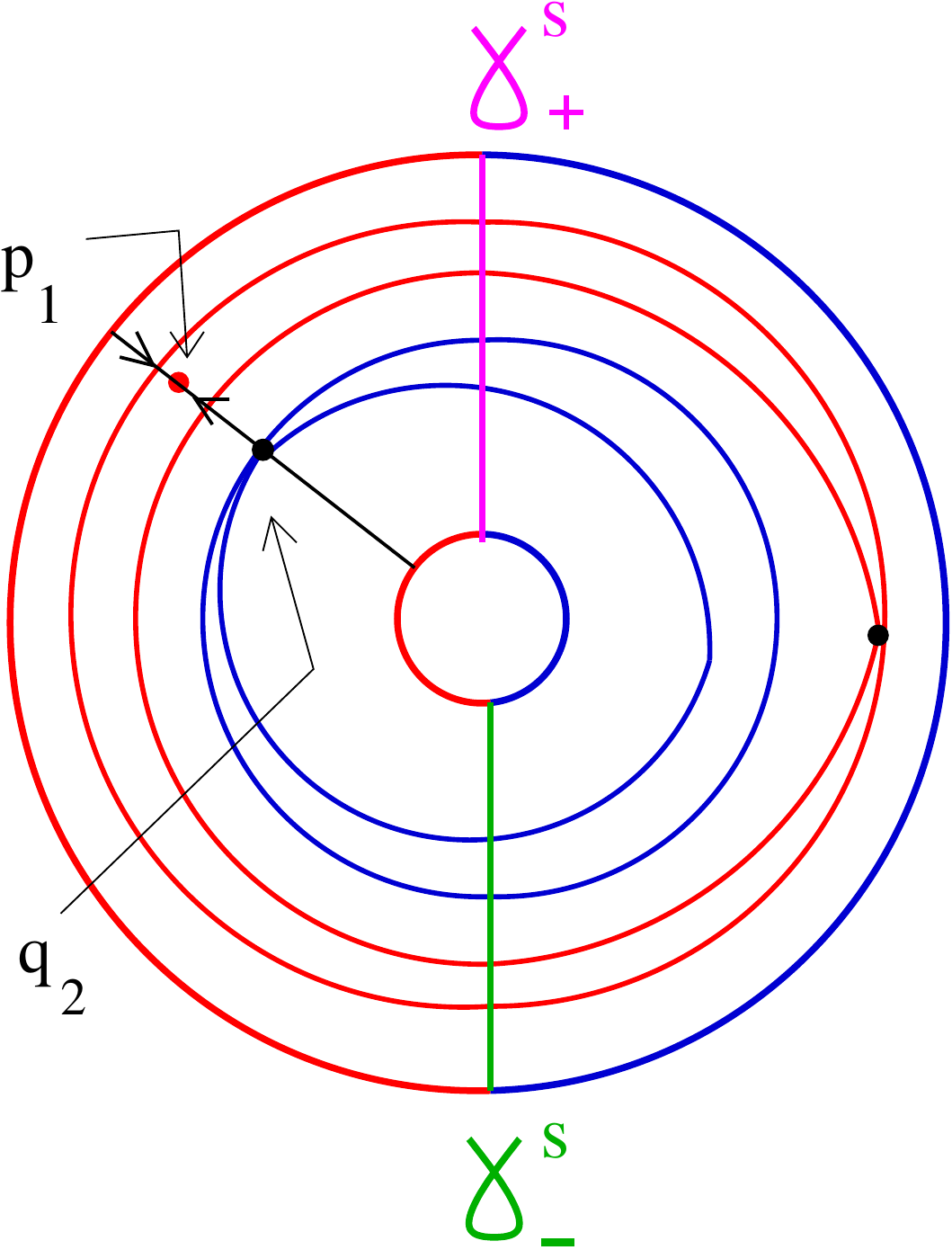}
\caption{(a) $X \in  \cH\cE_{1}$ and (b) $X \in  \cH\cE_{2}$}
\label{f-HE1eHE2}
\end{figure}
\vspace{-0.2cm}

\begin{lema}\label{l.hetero} The subsets $\cH\cE_i$ are codimension $1$ submanifolds of $\cO_\varphi$. 
  \end{lema}
  
Consider $X\in\cO_\varphi^{+,+}\setminus (\cH\cE_1\cup\cH\cE_2)$.  Then 
$\Sigma\setminus (W^s_1\cup W^s_2)$ has exactly $2$ connected components: one,  denoted by $\Sigma_+$ contains $\gamma^s_+$, and the other, denoted by $\Sigma_-$ contains $\gamma^s_-$. 
We denoted by $\cL^+$ the open subset of $\cO_\varphi^{+,+}$ where  both points $q_1,q_2$ belong to $\Sigma_+$ and $\cL^-$ the open set where  both points $q_1,q_2$ belong to $\Sigma_-$. We denote by $\widetilde{\cO_\varphi}^{+,+}$ the open subset where $q_1$ and $q_2$ belong to different components $\Sigma_\pm$ and $\Sigma_\mp$. 
We will denote $\cL^{+,-}$ the union
$$\cL^{+,-}=\cO_\varphi^{-,-}\cup \cO_\varphi^{+,-}\cup\cO_\varphi^{-,+}\cup\widetilde{\cO_\varphi}^{+,+}\cup \left((\cH^1\cup\cH^2)\setminus ((\cH^1_+\cap \cH^2_+) \cup (\cH^1_-\cap \cH^2_-)\right).$$
Recall that $\cH^i$ corresponds to a homoclinic loop and $\cH^i_\pm$ distinguishes the up or down connected components of $W^{s}(\sigma)\setminus W^{ss}(\sigma)$ involved in that loop. 

\section{The topological dynamics in the different regions of $\cO_\varphi$}\label{s.cartografia}
%%%%%%%%%%%%%%%%%%%%%%%%%%%%%%%%%%%%%%%%%%%%%%%%
\subsection{Upper and down Lorenz attractor: the regions $\cL^+$ and $\cL^-$}
%%%%%%%%%%%%%%%%%%%%%%%%%%%%%%%%%%%%%%%%%%%%%%%%

This section aims  to prove the following result:

\setcounter{maintheorem}{0}
\begin{maintheorem}\label{t.up} With the notation of Section~\ref{ss.region}, any vector field $X\in\cL^+$ admits precisely $2$ chain recurrence classes: one is an upper-Lorenz attractor, and the other is a basic hyperbolic set, topologically equivalent to the suspension of a fake horseshoe. 
The symmetric statement holds in $\cL^-$, interchanging the up and down. 
\end{maintheorem}
 \begin{figure}[h]
\centering
\includegraphics[scale=0.16]{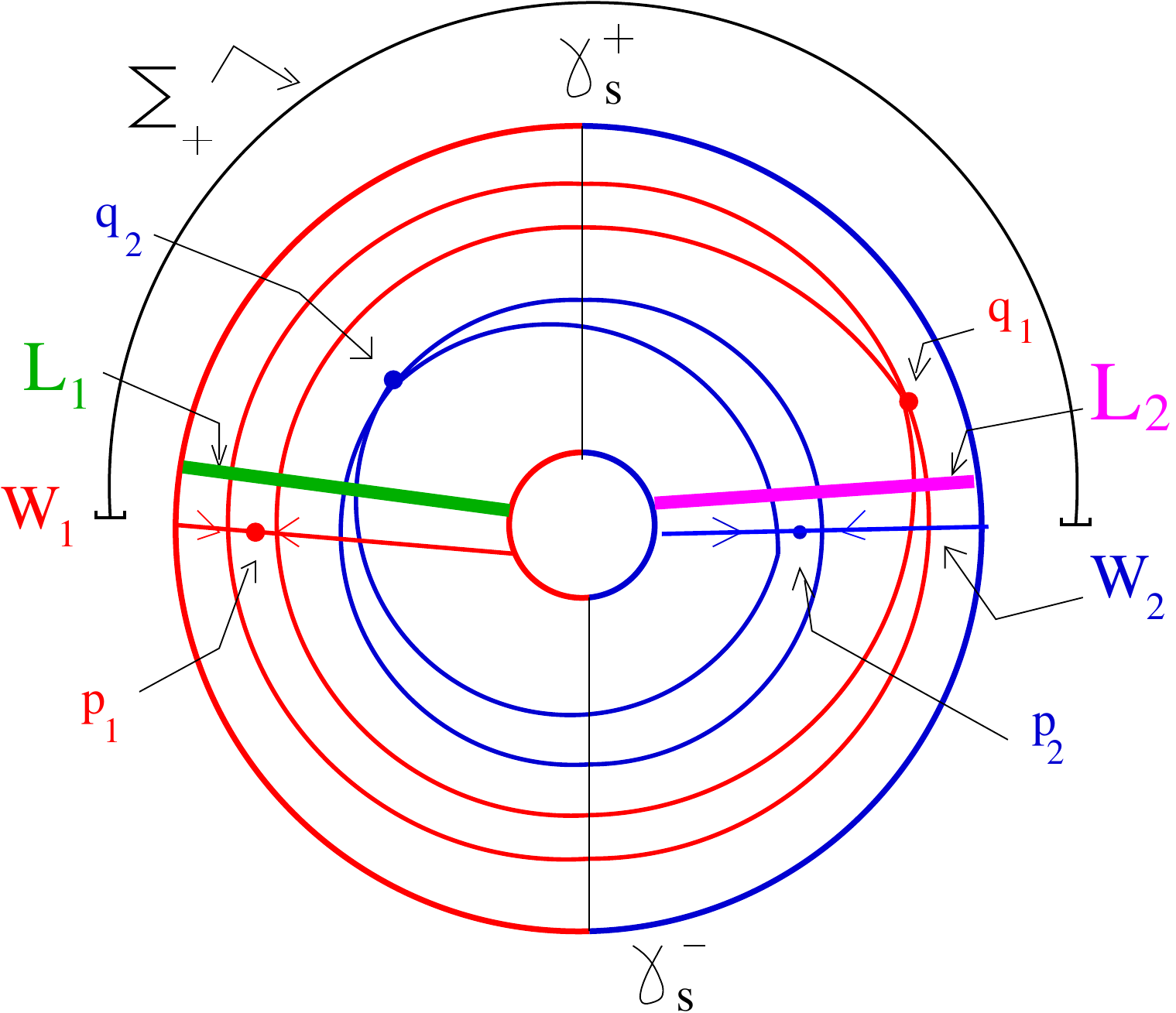}
\caption{The leaves $L_1$ and  $L_2$, and the region $\Sigma_+$.}\label{Teo-up}
\end{figure}	
 \vspace{-0.3cm}
 
 \begin{proof} The component $\Sigma_+$ is a rectangle.  The stable leaf $\gamma^s_+$ cuts $\Sigma_+$ into two components.  One is $\Sigma^1\cap\Sigma_+$ and is  bounded by $W^s_1=W^s(p_1)$ and the other is $\Sigma^2\cap\Sigma_+$ and is  bounded by $W^s_2$. Note that 
 $q_1$ belongs to $\Sigma^2\cap\Sigma_+$  and $q_2$ belongs to $\Sigma^1\cap\Sigma_+$. 
 Consider a stable leaf $L_1\in \Sigma^1\cap\Sigma_+$  separating $W^s_1$ from $q_2$,  
 and a stable leaf $L_2\in \Sigma^2\cap\Sigma_+$  separating $W^s_2$ from $q_1$, see Figure \ref{Teo-up}.
 Then $L_1\cup L_2$ cut $\Sigma^+$ in $3$ components: one is bounded by $W^s_1$ another by $W^s_2$ and the third, denoted by $R_L$ is a rectangle bounded by both $L_1$ and $L_2$ and containing $q_2$,$q_1$ and $\gamma^s_+$.  The leaf $\gamma^s_+$ cuts $R_L$ in two components, 
 $R_{L,1}\subset\Sigma_1$ and $R_{L_2}\subset \Sigma_2$. 
 
 Consider the restriction of $P$ to $R_L\setminus \gamma^+_s$. 
 The images of $R_{L_1}$ and $R_{L_2}$ are cuspidal triangles with cusps at $q_1$ and $q_2$ and are contained in $R_L$.  Recall that $P$ is hyperbolic.  
 Thus the restriction or $P$ to $R_L$ satisfies all the properties of the return map in the geometric model of the Lorenz attractor with a rate expansion larger than $\varphi>\sqrt2$. The rectangle $R_L$ is an attracting region for $P$.  
 One deduces that $U$ contains an attracting sub-region, in which $X$ is a geometric model of the Lorenz attractor. 
 
 Consider now the component $R_H$ of $\Sigma\setminus (L_1\cup L_2)$ disjoint from $R_L$, and containing $\gamma^s_-$. Consider the restriction of $P$ to $R_H\setminus\gamma^s_-$.  
 The image of each  of these components crosses $R_H$ in a Markovian way. Thus the maximal invariant in $R_H$ is far from the discontinuity and is conjugated to the \emph{fake horseshoe}, as the map $P$ preserves the orientation of the unstable cone field. This ends the proof. 
 The down case is similar.
 \end{proof}
 %%%%%%%%%%%%%%%%%%%%%%%%%%%%%%%%%%%%%%%
 \subsection{Two-sided Lorenz attractor in $\cO_\varphi^{-,-}$}
 %%%%%%%%%%%%%%%%%%%%%%%%%%%%%%%%%%%%%%%
 \begin{prop}\label{p.--} For any $X\in \cO_{\varphi}^{-,-}$, the maximal invariant set in $U$ is transitive and consists of a two-sided Lorenz attractor. 
 \end{prop}

 \noindent {\em{Proof.}} According to Lemma~\ref{l.cortaestavefixo}, for proving the transitivity of the maximal invariant, it is enough to verify that  the iterates of an unstable segment $S$ in $\Sigma$  cut all the stable leaves with the possible exception of a set with an empty interior,
 see Figure \ref{f.p--}.
 
  \begin{figure}[h]
\centering
	\includegraphics[scale=0.16]{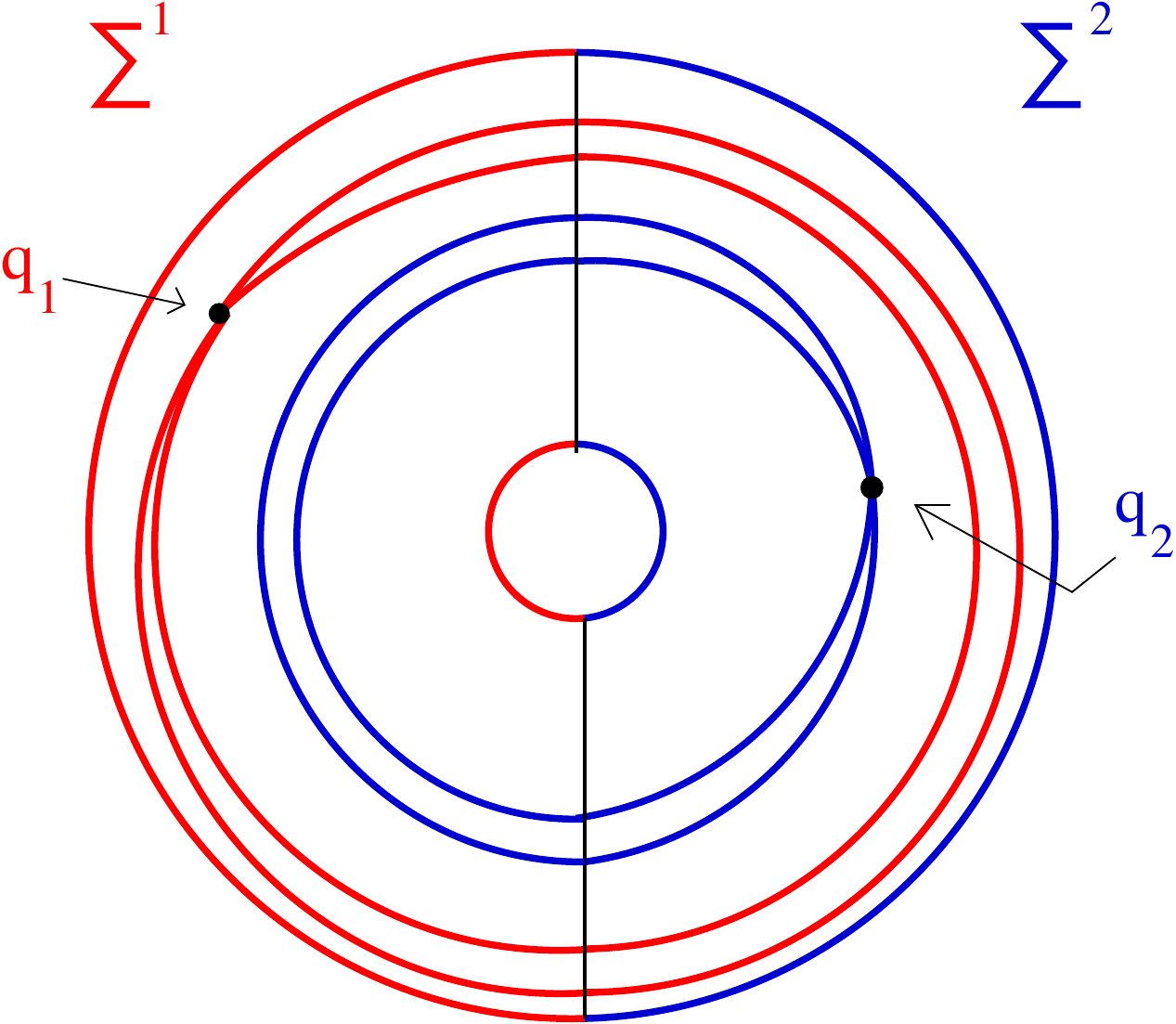}
	\caption{$q_1 \in \Sigma^1$ and $q_2 \in \Sigma^2$}\label{f.p--}
\end{figure}

 So consider an unstable segment $S\subset \Sigma$. Consider the set of lengths of the connected components of all positive iterates $P^n(S)$.  If this set of lengths is not bounded,  some component cuts every stable leaf, and we are done. 
 Otherwise, given any $\delta>1$, up to replace $S$ by a segment in one of its iterates, one may assume that any connected component $S'$ in any iterate $P^n(S)$ has a length bounded by $\delta\ell(S)$.  We now fix $\delta\in]1,\frac{\lambda^2}2[$. 
 
 If $S$ is disjoint from $\gamma^s_+\cup\gamma^s_-$, then $\ell(P(S))>\lambda \ell(S)>\delta\ell(S)$ contradicts the choice of $S$.  So $S$ cuts $\gamma^s_+$ or $\gamma^s_-$.  If it cuts both, then it crosses completely {$\Sigma^1$ or $\Sigma^2$} so that $P(S)$ cuts all the stable leaves but $1$, and we are done. 
 So we may assume that $S$ cuts exactly one of $\gamma^s_+$ or $\gamma^s_-$, say $\gamma^s_+$, for instance. Let $S_1$ be a component of $S\setminus\gamma^s_+$ with a length larger or equal to $\frac12\ell(S)$. The length of $P(S_1)$ is at least $\frac\lambda2 \ell(S)$. 
 
 If $P(S_1)$ is disjoint from $\gamma^s_+\cup\gamma^s_-$ then 
 $$\ell(P(S_1)>\lambda \ell(S_1)> \frac{\lambda^2}2\ell(S)> \delta\ell(S).$$
 contradicting the choice of $S$.  So $P(S_1)$ cuts $\gamma^s_+$ or $\gamma^s_-$.
 Once again, if it cuts both of them, we are done, so one may assume that $P(S_1)$ cuts exactly one of $\gamma^s_+$ or $\gamma^s_-$.  Note that one of the endpoints of $P(S_1)$ is a cuspidal point $q_i\in \Sigma^i$ as $X\in \cO_1^{-,-}$.
 Thus Lemma~\ref{l.cusptogamma} applies and concludes the proof of the transitivity of the maximal invariant set $\Lambda_X$. 
 Note that the compact set $\Lambda_P$ consists of a union of essential circles in $\Sigma$ and therefore always cuts $\gamma^s_+$ and $\gamma^s_-$, so that the maximal invariant set $\Lambda_X$ intersects non-trivially both stable sets $W^s_+(\sigma)$ and $W^s_-(\sigma)$. Thus $\Lambda_X$ is a two-sided Lorenz attractor, ending the proof. See Figure \ref{f.p--}. $\square$
 %%%%%%%%%%%%%%%%%%%%%%%%%%%%%%%%%%%%%%%%%%%%%%%%
\subsection{Two-sided Lorenz attractor in $\cO^{+,-}_\varphi$ and $\cO^{-,+}_{\varphi}$}
%%%%%%%%%%%%%%%%%%%%%%%%%%%%%%%%%%%%%%%%%%%%%%%%

\begin{prop}\label{p.+-} For any $X\in \cO_{\varphi}^{+,-}\cup \cO_{\varphi}^{-,+}$, the maximal invariant set in $U$ is transitive and consists of a two-sided Lorenz attractor. 
 \end{prop}
 The proof of the proposition for $X\in \cO_{\varphi}^{-,+}$ is identical to the proof for
 $X\in\cO_{\varphi}^{+,-} $ interchanging the components {$\Sigma^1$ and $\Sigma^2$}, so we will write the proof only for $X\in \cO_{\varphi}^{+,-}$. See Figure \ref{f-Up-Lorenz2-em-O++}.

 \noindent {\em{Proof.}} Recall that the rate of expansion of $X$ is $\lambda>\varphi$. 
  \begin{figure}[th]
\centering
	\includegraphics[scale=0.17]{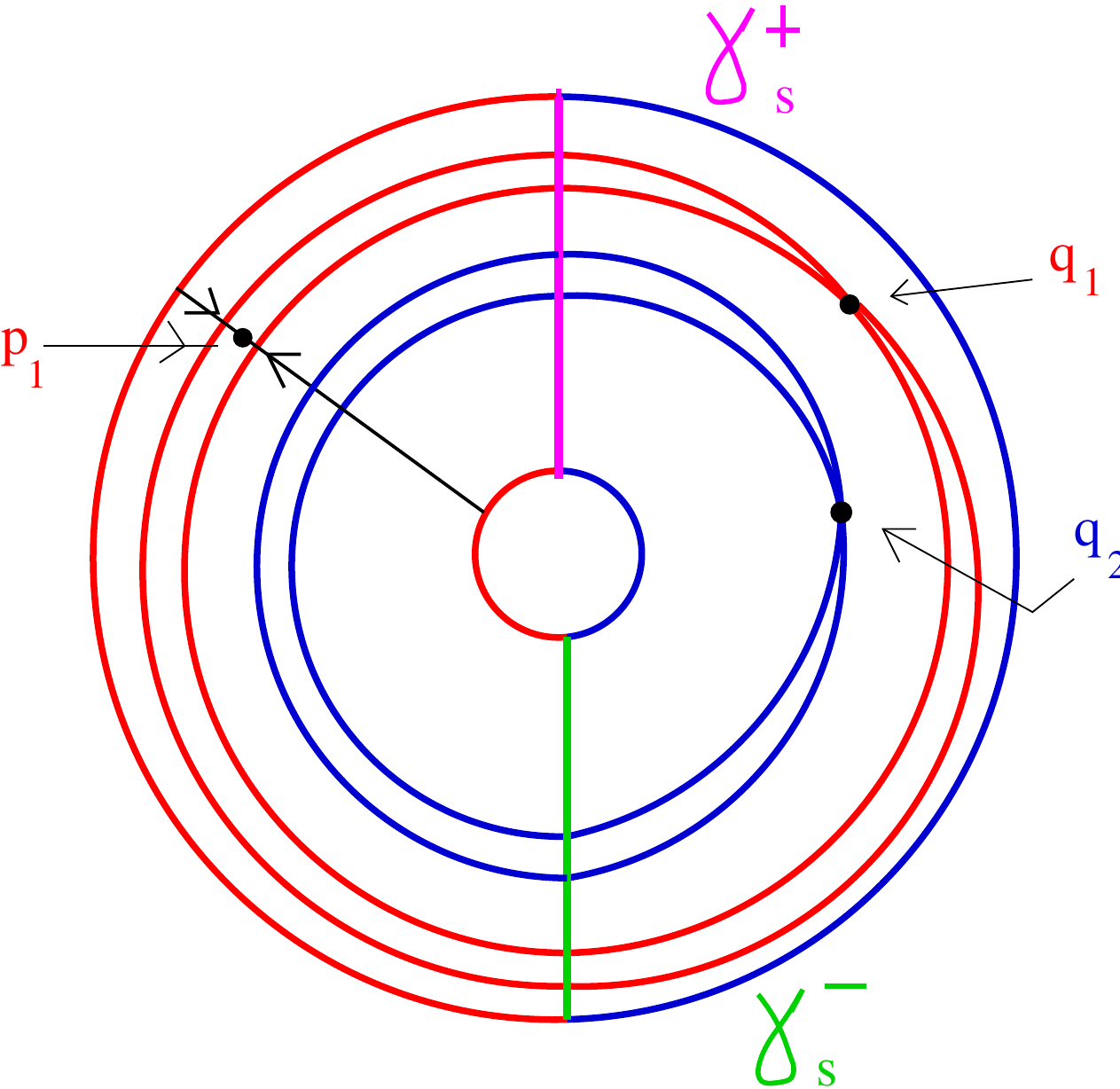}
	\caption{$q_1, q_2 \in \Sigma^2$ for $X\in \cO^{-,+}$}\label{f-Up-Lorenz2-em-O++}
\end{figure}
 
 As for Proposition~\ref{p.--}, it is enough to prove that  the iterates of an unstable segment $S$ in $\Sigma$  cut all the stable leaves with a possible exception of a set with an empty interior. See Figure \ref{f-Up-Lorenz2-em-O++}.
 So consider an unstable segment $S\subset \Sigma$. Consider the set of lengths of the connected components of all positive iterates $P^n(S)$.  If this set of lengths is not bounded, some component cuts every stable leaf, and we are done. 
 Otherwise, given any $\delta>1$, up to replace $S$ by a segment in one of its iterates, one may assume that any connected component $S'$ in any iterate $P^n(S)$ has a length bounded by $\delta\ell(S)$.  We now fix $\delta\in]1,\frac{\lambda}\varphi[$. 
 
 If $S$ is disjoint from $\gamma^s_+\cup\gamma^s_-$, then $\ell(P(S)>\lambda \ell(S)>\delta\ell(S)$ contradicts the choice of $S$.  So $S$ cuts $\gamma^s_+$ or $\gamma^s_-$.  If it cuts both, then it crosses completely $\Sigma^1$ or $\Sigma^2$ so that $P(S)$ cuts all the stable leaves but $1$, and we are done. 
 
 So we may assume that $S$ cuts exactly one of $\gamma^s_+$ or $\gamma^s_-$, say $\gamma^s_-$, for instance.  Thus $S$ is cut by $\gamma^s_-$ into two components 
{$S_i=S\cap \Sigma^i$. }
 Consider $P(S_2)$. If $P(S_2)\cap(\gamma^s_-\cup \gamma^s_+)\neq \emptyset$ then Lemma~\ref{l.cusptogamma} applies (because $q_2\in \Sigma^2$ by our assumption $X\in\cO^{+,-}_\varphi$).  Thus the iterate of $S_2$ cuts every stable leaf but a finite number of them so that we are done. 
 
 Thus we may assume that $P(S_2)$ is disjoint from $\gamma^s_-\cup  \gamma^s_+$. Hence $P(S_2)$ is an unstable segment of length larger than $\lambda^2\ell(S_2)$.  On the other hand $\ell(P(S_1))\geq\lambda \ell(S_1)$. Now Lemma~\ref{l.golden} implies that 
 $\max\{\ell(P(S_1)),\ell(P^2(S_2))\}\geq \frac\lambda\varphi \ell(S)>\delta\ell(S)$,
contradicting the choice of the segment $S$, finishing the proof. $\square$

 %%%%%%%%%%%%%%%%%%%%%%%%%%%%%%%%%%%%%%%%%%%%%%%%
 \subsection{Two-sided Lorenz attractor in $\widetilde{\cO_\varphi}^{+,+}$}
 %%%%%%%%%%%%%%%%%%%%%%%%%%%%%%%%%%%%%%%%%%%%%%%%
\begin{prop}\label{p.++} For any $X\in \widetilde{\cO_{\varphi}}^{+,+}$ the maximal invariant set in $U$ is transitive and consists of a two-sided Lorenz attractor. 
 \end{prop}
 
 Vector fields $X$ in  $\widetilde{\cO_\varphi}^{+,+}$ are characterized by the fact that the points $q_1$ and $q_2$ belong to different components $\Sigma_+$ (containing $\gamma^s_+$) and $\Sigma_-$ (containing $\gamma^s_-$) of $\Sigma\setminus (W^s_1\cup W^s_2)$ where $W^s_i=W^s(p_i)$ and $p_i$ is the fixed point of $P$ in $\Sigma^i$. 
 Thus $\widetilde{\cO_\varphi}^{+,+}$ is the union of two disjoint open subsets, defined by $q_1\in \Sigma_+$ or $q_1\in \Sigma_-$.  
 The proof of the proposition is symmetrical in these two open sets.  We provide here the proof where $q_1\in \Sigma_+$. See Figure \ref{f-Up-Lorenz-em-OTilda++}.

  \begin{figure}[th]
\centering
	\includegraphics[scale=0.17]{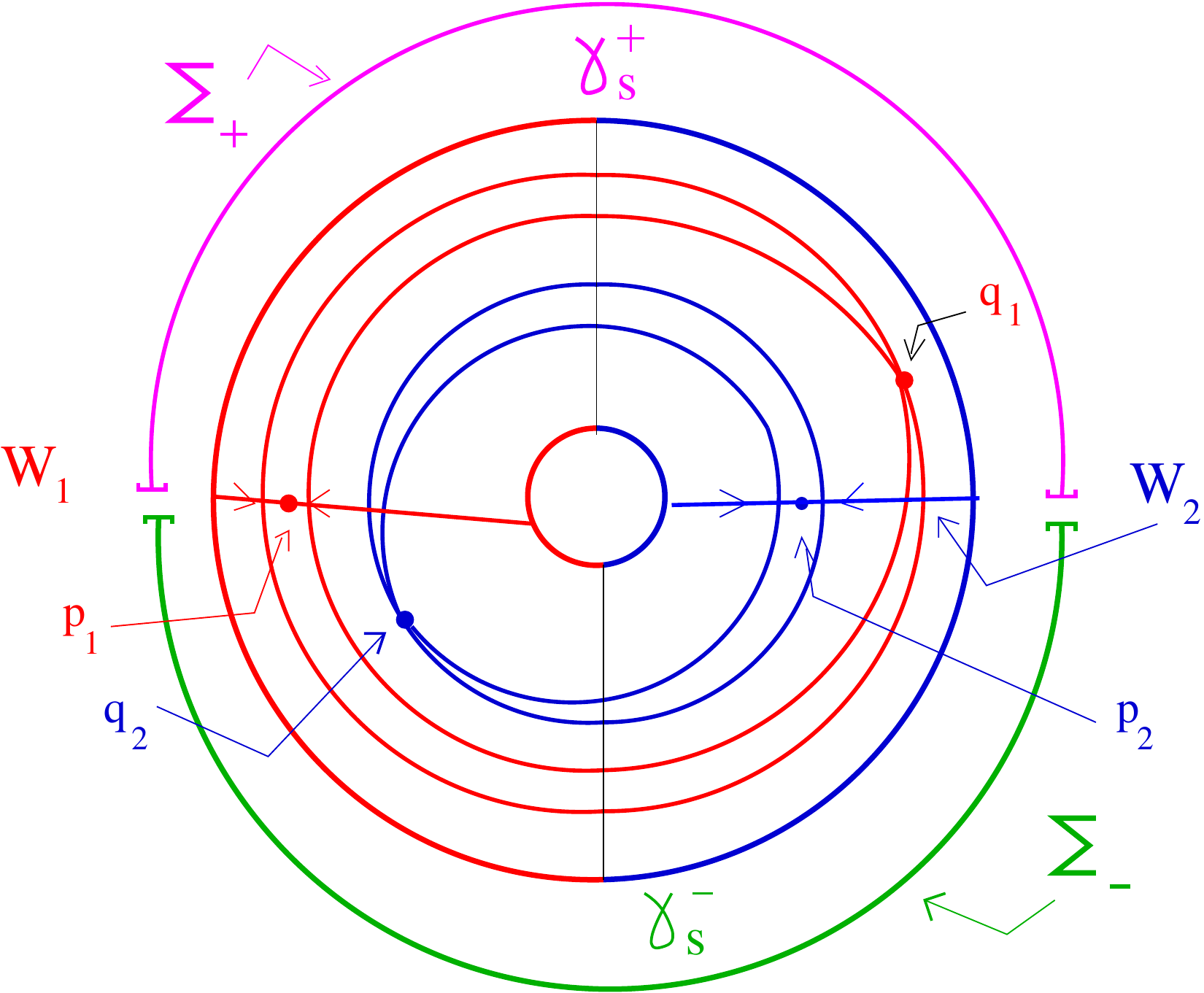}
	\caption{$q_1 \in \Sigma_+$ for $X\in \widetilde{\cO_\varphi}^{+,+}$}\label{f-Up-Lorenz-em-OTilda++}
\end{figure}
 
 \noindent {\em{Proof.}} Most of the proof is identical to the proofs of Propositions~\ref{p.--} and~\ref{p.+-} and allows us to consider a segment $S$ so that any component in any iterate $P^n(S)$ has a length bounded by $\delta\ell(S)$ with $1<\delta<\frac\lambda\varphi$ where $\lambda$ is the expansion rate of $X$. Furthermore, $S$ cuts exactly one of the stable leaves $\gamma^s_+$ or $\gamma^s_-$. 
 Let us assume that $S$ cuts $\gamma^s_-$, and denote $S_i=S\cap \Sigma_i$. If one of $P(S_1)$ or $P(S_2)$ is disjoint from $\gamma^s_+ \cup\gamma^s_-$, then one concludes in the same way as in the proof of Proposition~\ref{p.+-} that, using Lemma~\ref{l.golden}, one of the iterates
 $P(S_1)$, $P^2(S_1)$, $P(S_2)$ or $P^2(S_2)$ contains a segment of length larger than $\frac\lambda\varphi$ contradicting the choice of $S$.  
 Thus one may assume that both $P(S_1)$ and $P(S_2)$ cut $\gamma^s_+ \cup\gamma^s_-$. 
 For the orientation of $\cC^u$, $S_1$ has its endpoint at $\gamma^s_-$.  Thus $P(S_1)$ ends at $q_1\in \Sigma_+$. As seen above, $P(S_1)$ cuts $\gamma^s_+ \cup\gamma^s_-$ and  its orientation implies that indeed it cuts $\gamma^s_-$. 
 In particular, it goes out of $\Sigma_+$.  One deduces that $P(S_1)$ cuts transversely $W^s_2=W^s(p_2)$, where $p_2$ is the fixed point in $\Sigma_2$  (recall that $X\in\cO_\varphi^{+,+}$). Thus Lemma~\ref{l.cortaestavefixo} ensures that the iterates of $P(S_2)$ cross all stable leaves but a finite number of them, concluding that case. 

The case  where $S$ cuts $\gamma^s_+$ is similar, just replacing $S_1=S\cap \Sigma_1$ and $W^s(p_2)$ by $S_2=S\cap \Sigma_2$ and $W^s(p_1)$. This concludes the proof. $\square$
  
 %%%%%%%%%%%%%%%%%%%%%%%%%%%%%%%%%%%%%%%%%%%%%%%%
 \subsection{Two-sided Lorenz attractor in  hypersurfaces $\cH^i$ of homoclinic loops}
%%%%%%%%%%%%%%%%%%%%%%%%%%%%%%%%%%%%%%%%%%%%%%%%
 \begin{maintheorem}\label{p.Hi}   For any $X\in\cO_1$  in one of the hypersurfaces $\cH^1$ or $\cH^2$, the maximal invariant in $U$ is a transitive singular attractor meeting both stable {sets} $W^s_\pm(\sigma)$. 
 \end{maintheorem}
 \begin{proof} We assume that one of the points $q_1$,$q_2$, say $q_1$ belongs to one of the stable leaves $\gamma^s_+$ or $\gamma^s_-$, say $q_1\in \gamma^s_+$ (all the cases admit an identical proof, \emph{mutatis mutant}). %See Figure \ref{f.p.Hi1}.

\begin{figure}[th]
\centering
	\includegraphics[scale=0.17]{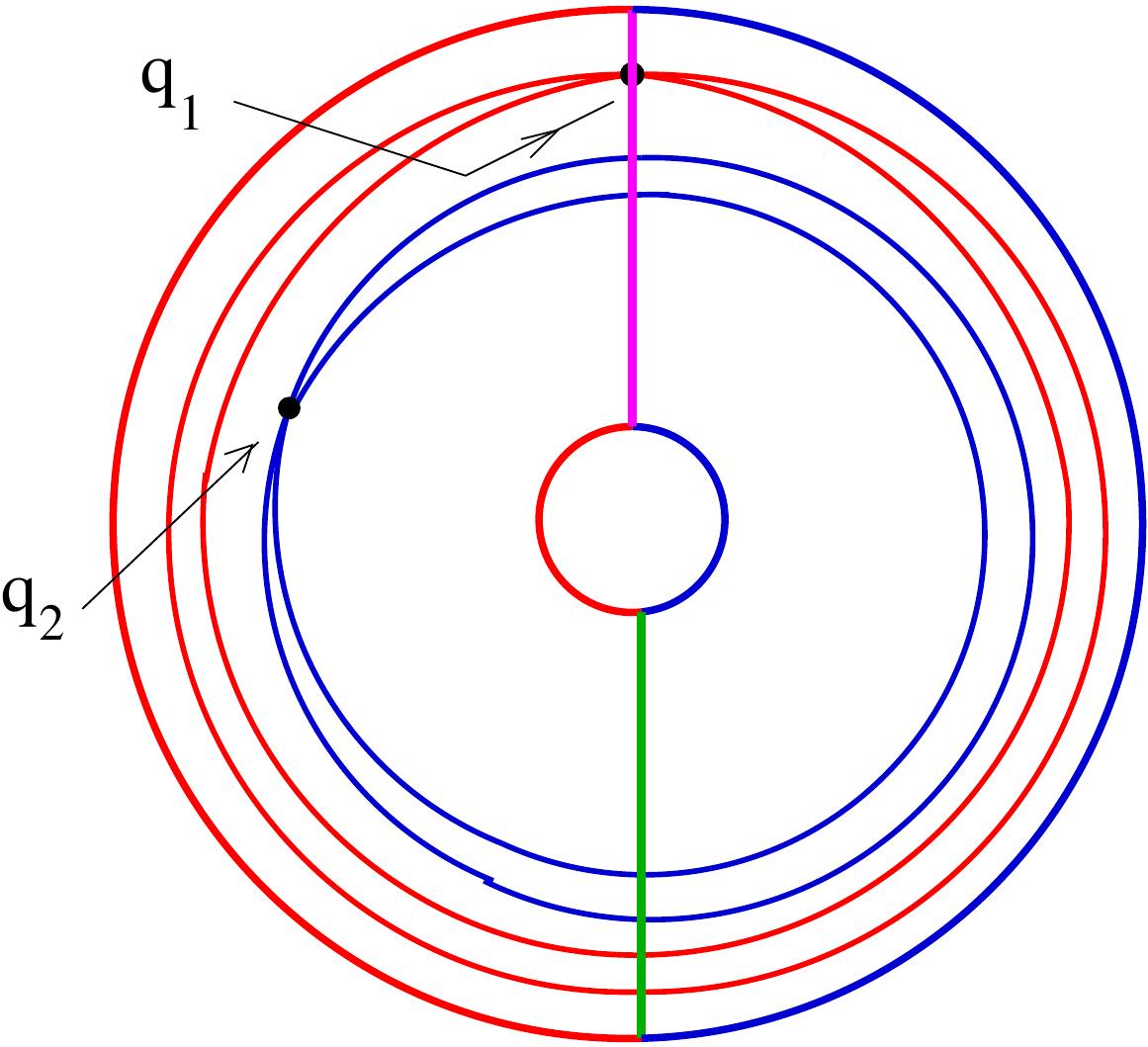}
		\hspace{0.2cm}
		\includegraphics[scale=0.17]{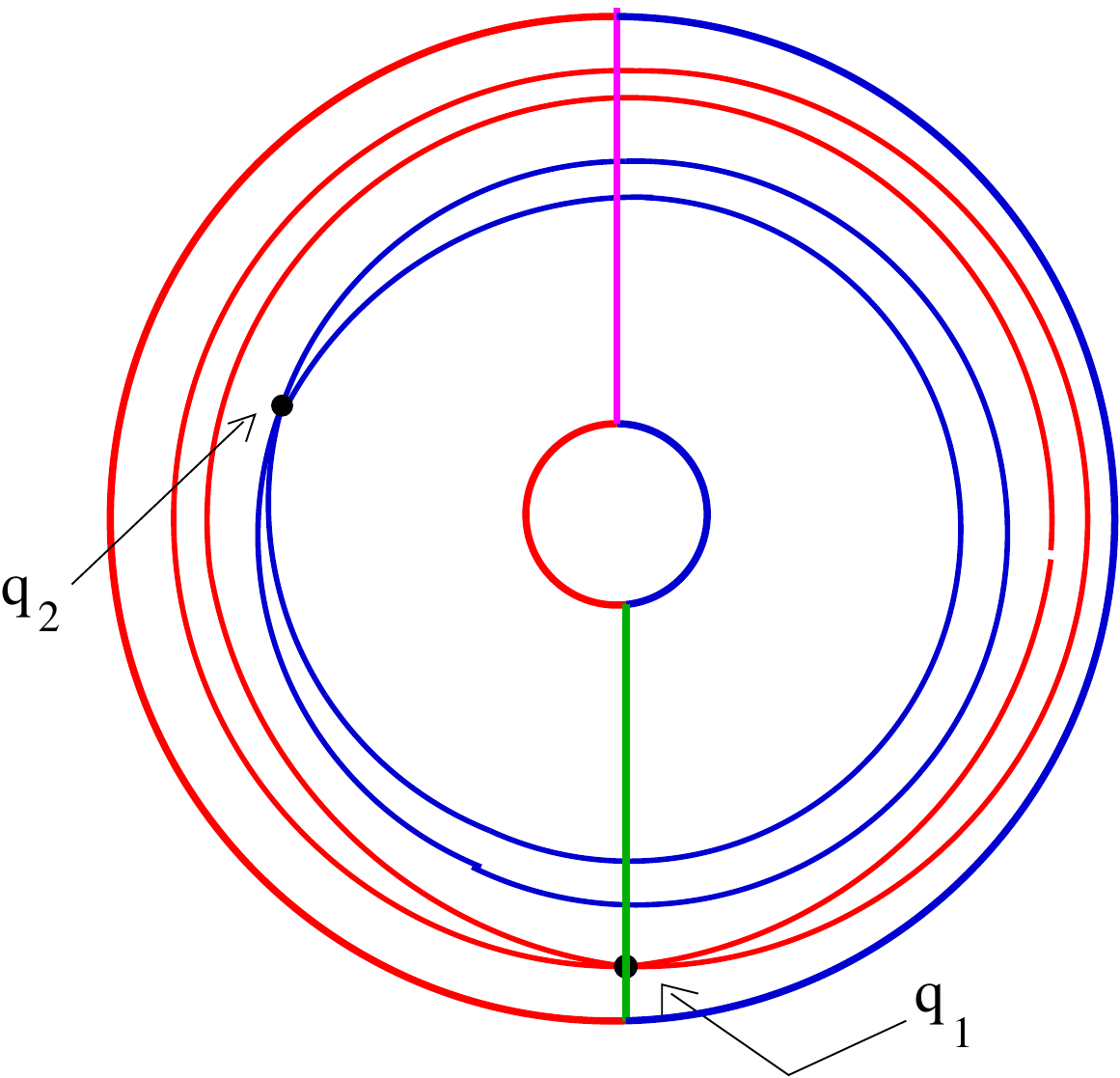}
\caption{$X \in \cH^1 \cup \cH^2$}%\label{f.p.Hi2}
\label{f.p.Hi1}
\end{figure}

Similarly to the proof of the Propositions~\ref{p.--}, ~\ref{p.+-} and ~\ref{p.++},  the proof consists in considering an unstable segment $S$  which does not admit any segment $\tilde S$ of length larger than $\delta\ell(S)$, $1<\delta<\frac\lambda\varphi$, so that $\tilde S$, excepted a finite subset,  is contained in the union of the iterate $P^n(S)$, $n\geq0$. One needs to prove that the iterates of such a segment $S$ cut any stable leaves but a finite number of them.   
  Again, the choice of $S$ implies that $S$ cuts $\gamma^s_+$ or $\gamma^s_-$ and if it cuts both, then $P(S)$ already cuts all stable leaves but one.  So we assume that $S$ cuts only $1$ of these leaves.

 Assume first that $S$ cuts $\gamma^s_+$. Consider {$S_1=S\cap \Sigma^1$}. It is a segment starting at a point in $\gamma^s_+$ and contained in $\Sigma^1$.   Then $P(S_1)$ is a segment a length larger than $\lambda\ell(S)$ and starting at $q_1\in \gamma^s_+$. If $P(S_1)$ is not included in  {$\Sigma^1 $}, then  it cuts both $\gamma^s_+$ and $\gamma^s_-$ and $P^2(S_1)$ cuts all leaves but one.  
 If $P(S^1)\subset \Sigma^1$, then $P^2(S_1)$ is a segment of length larger than $\lambda^2\ell(S)$ and starting at $q_1$. 
 Iterating the process, one gets that one of the iterates $P^k(S_1)$ crosses completely {$\Sigma^1$} so that $P^{k+1}(S_1)$ cuts all stable leaves but one, and we are done. 
 
 Assume now that $S$ cuts $\gamma^s_-$. Consider {$S_i=S\cap \Sigma^i$}. Then $P(S_1)$ is a segment ending at $q_1\in\gamma^s_+$, and $P(S_1)$ either crosses completely {$\Sigma^2$} (and we are done) or is contained in {$\Sigma^2$}. Now $P^2(S_1)$ is a segment ending at $q_2$. On the other hand, $P(S_2)$ is a segment starting at $q_2$. 
 Now $P^2(S_1)\cup\{q_2\}\cup P(S_2)$ is a segment of length at least $\lambda \ell(S)$.  This contradicts the choice of $S$, ending the proof.
 \end{proof}
 
 The statement of Theorem~\ref{p.Hi} does not announce a two-sided Lorenz attractor, because the $\cH^i$ are not open subsets, and then the robustness of the transitivity is not ensured.  However, we will see below that $X\in \cH^i$ exhibits indeed  a two-sided Lorenz attractor excepted for $X$ in a codimension $2$ submanifold. 
 Let $\cH^{1,2}_+$ and $\cH^{1,2}_-$ the codimension $2$ submanifolds included in $\cH^1\cup \cH^2$ consisting in vector fields $X$ so that $q_1,q_2\in \gamma^s_+$  or $q_1,q_2\in \gamma^s_-$, respectively, see Figure \ref{1-f.p.Hi2}. Thus both unstable separatrices of $\sigma$ are homoclinic connections and are included in the same connected component of  $W^s(\sigma)\setminus W^{ss}(\sigma)$. 
 
  \begin{figure}[th]
\centering
\includegraphics[scale=0.17]{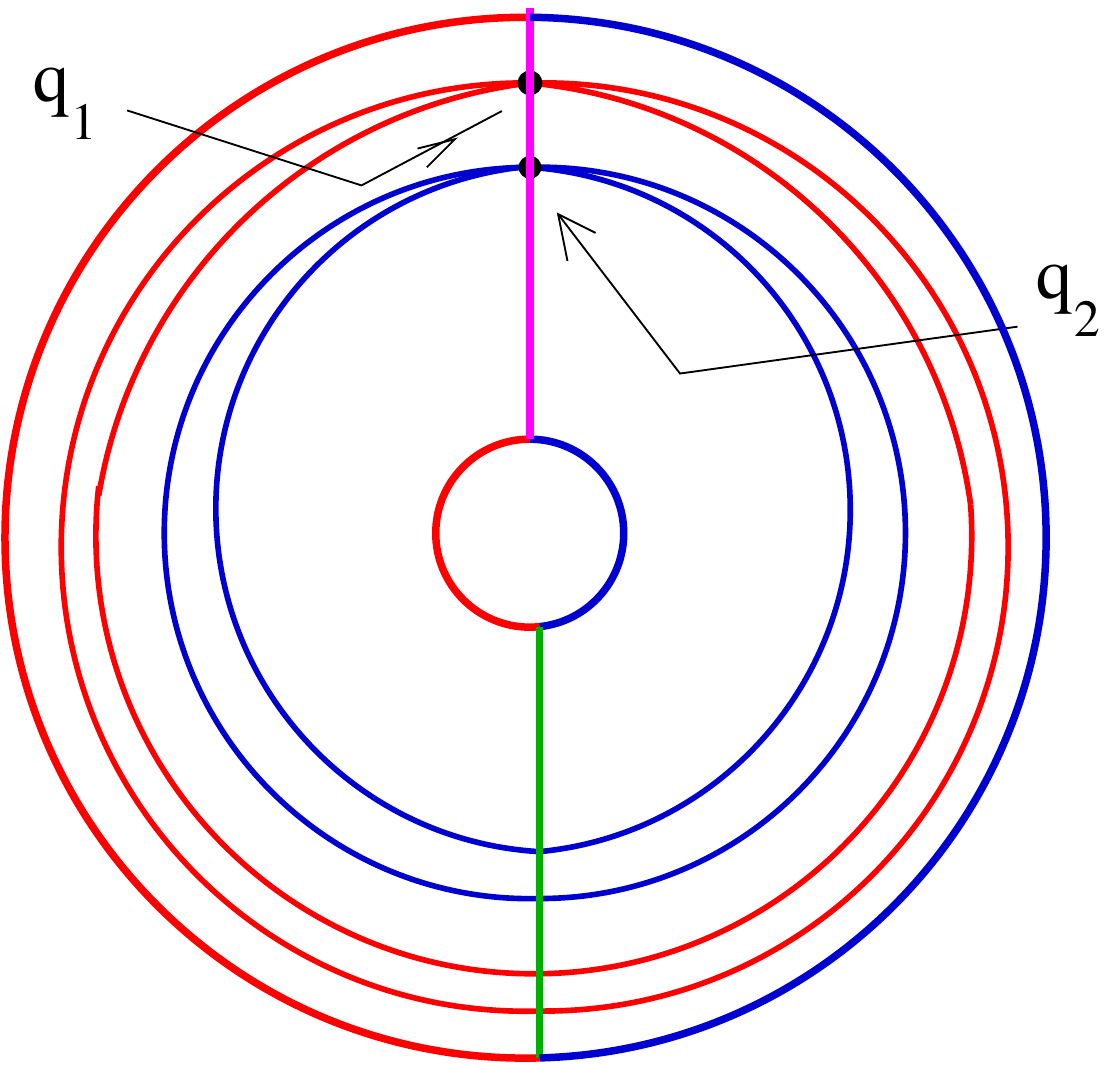}
\hspace{0.5cm}
\includegraphics[scale=0.17]{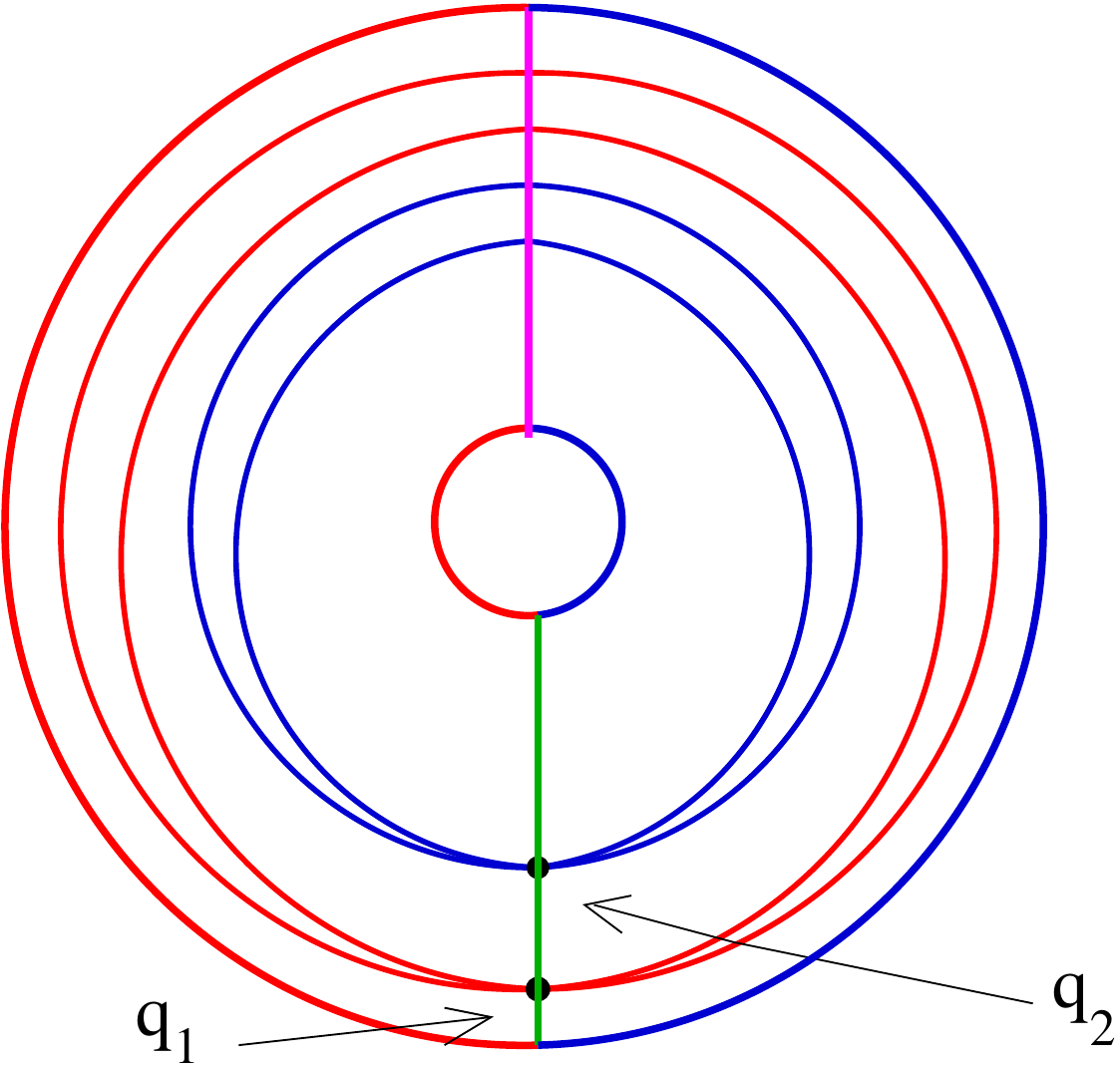}
\caption{$ X \in \cH^1\cup \cH^2$, and $q_1,q_2\in \gamma^s_+$ and  $q_1,q_2\in \gamma^s_-$}\label{1-f.p.Hi2}
\end{figure}
 
 The next lemma ensures that for $X\in\cH^i$ out of $\cH^{1,2}_+$ and $\cH^{1,2}_-$, the transitivity of the attractor is robust so that the attractor is a two-sided Lorenz attractor. 
 
 \begin{maintheorem}\label{l.Hi}
 	 If $X\in \cH^i\setminus (\cH^{1,2}_+ \cup\cH^{1,2}_-)$, 
 	there is a neighborhood $\cU(X)$ of $X$ such that the maximal invariant set for every $Y\in \cU(X)$ is a two-sided Lorenz attractor.
 \end{maintheorem}
 \begin{proof}The proof consists of unfolding the homoclinic loops and checking that all the possibilities lead to one of the $\cO^{-,-}_\varphi\cup\cO^{+,-}_\varphi\cup\cO^{-,+}_\varphi\cup\widetilde{\cO^{+,+}_\varphi}\cup \cH^1\cup\cH^2$.
 \end{proof}
 
%%%%%%%%%%%%%%%%%%%%%%%%%%%%%%%%%%%%%%%%%%%%%%%%
 \subsection{Two-sided Lorenz attractor in $\cL^{+-}$}
 %%%%%%%%%%%%%%%%%%%%%%%%%%%%%%%%%%%%%%%%%%%%
 
 Note that Propositions~\ref{p.--}, ~\ref{p.+-}, \ref{p.++} and \ref{p.Hi} and Lemma~\ref{l.Hi} together prove
 
 \begin{maintheorem}\label{p.L+-} For any $X\in\cL^{+,-}$, the maximal invariant set in $U$ is a transitive singular hyperbolic attractor meeting both {stable sets} $W^{s}_{-}(\sigma)$ and $W^s_{+}(\sigma)$; that is, it is a two-sided Lorenz attractor. 
 \end{maintheorem}

 \noindent The region $\cL^{+,-}$ has been defined as a union of many regions, and the propositions listed above prove the conclusion of Proposition~\ref{p.L+-} in each of these regions. Finally, some of these regions are not open, so Lemma~\ref{l.Hi} checks that the vector fields in these non-open regions admit neighbourhoods contained in the union of the other regions.

 %%%%%%%%%%%%%%%%%%%%%%%%%%%%%%%%%%%%%%%%%%%%
 \subsection{The collisions of the Lorenz attractor and a fake horseshoe: vector fields in $\cH\cE_i$}
 %%%%%%%%%%%%%%%%%%%%%%%%%%%%%%%%%%%%%%%%%%%%
 In this section, we consider  vector fields $X$ in $\cO^{+,+}_\varphi$, that is,  so that the return map $P$ has $2$ fixed point $p_i\in\Sigma^i$, $i=1,2$,  a heteroclinic connection between $\sigma$ and one of the points $p_i$. More precisely, $q_i$ belongs to the stable leaf $W^s_j$ through $p_j$; note that $j\neq i$ because $q_i$ is not in $\Sigma^i$ when $\Sigma^i$ contains a fixed point. 
 The case when both $q_1$ and $q_2$ belong to $W^s_1\cup W^s_2$ corresponds to $X\in\cH\cE_1\cap\cH\cE_2$, which is a codimension $2$ submanifold.

 \begin{maintheorem}\label{p.HE1HE2} For any $X\in \cH\cE_1\cap\cH\cE_2$, there is a unique chain recurrence class, which is not transitive. Both $\Si_-$ and $\Si_+$ are invariant by $P$.  
 The maximal invariant set of the restriction of $P$ to $\Sigma_i$ is transitive: every unstable segment in $\Sigma_i$ has its iterates which cut every segment in $\Sigma_i$. 
{The attracting set $U$ splits into $2$ connected compact sets, each containing a  \emph{full Lorenz} intersecting along $W^s(p_1)\cup W^s(p_2)\cup W^u(\sigma)$.}  

 \end{maintheorem}
 \begin{proof}The first return map is illustrated in Figure \ref{f.p.HE1HE2}. The study of the first return map in each rectangle $\overline \Sigma_i$ is classical for the expansion rate larger than  $\sqrt2<\varphi$. 
 \end{proof}
\begin{figure}[th]\label{figu-p.HE1HE2}
\centering
\includegraphics[scale=0.17]{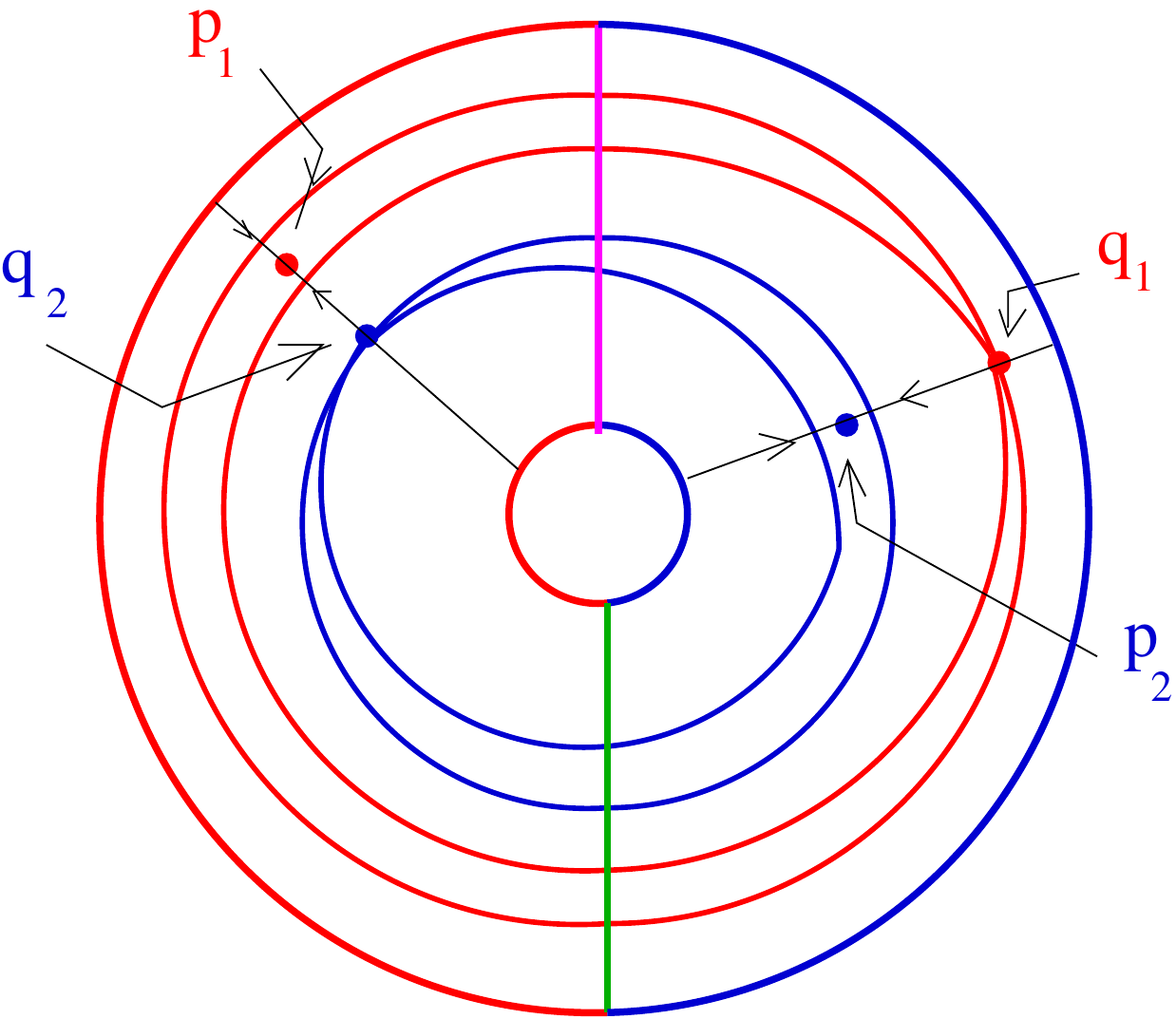}
\caption{$X\in \cH\cE_1\cap\cH\cE_2$}\label{f.p.HE1HE2}
\end{figure}
The submanifold  $\cH\cE_1\cap\cH\cE_2$ cuts $\cH\cE_1$ into two (relative) open subsets, as $q_2$ belongs either to $\Sigma_+$ or $\Sigma_-$ (connected components of $\Sigma\setminus (W^s_1\cup W^s_2)$. In the same way, the submanifold  $\cH\cE_2\cap\cH\cE_2$ cuts $\cH\cE_2$ into two (relative) open subsets, as $q_1$ belongs either to $\Sigma_+$ or $\Sigma_-$ (connected components of $\Sigma\setminus (W^s_1\cup W^s_2)$. 

We will consider  $X\in \cH\cE_1\setminus (\cH\cE_1\cap\cH\cE_2)$, and  $q_2\in \Sigma_+$, the other cases are similar. 

Then  the stable leaf through $q_2$ cuts $\Sigma_+$ in two components, one, denoted as $\Sigma^+_1$ bounded by the stable leaf $W^s_1$ trough $p_1$ and the other, denoted as $\Si^+_2$ bounded by $W^s_2$ (which contains $q_1$ by definition of $\cH\cE_1$);  notice that $\Sigma^+_2$ contains the stable leaf $\gamma^s_+$ , because $q_2\in\Sigma^1$. 
 Note that  if $X\in \cH\cE_1\setminus (\cH\cE_1\cap\cH\cE_2)$ then $q_1\in W^s(p_2)$ and
 $q_2 \notin W^s(p_1)$. 
 
 \begin{prop}\label{p.HE} Consider $X\in \cH\cE_1\setminus (\cH\cE_1\cap\cH\cE_2)$, 
 and assume $q_2\in \Sigma_+$. Then :
 \begin{itemize}
  \item  $X$ has a unique chain recurrence class in $U$, but this class is not transitive and consists of two (singular) homoclinic classes $K_-,K_+$ containing $\sigma$. 
  \item The iterates by $P$ of any unstable segment in $\Sigma$ cut any stable leaf in $\Sigma^+_2$ but finitely many of them.  
  Furthermore, the return map $P$ restricted to the rectangle $\overline{\Sigma^+_2}$ is a 
{Lorenz map }
  for the parameter corresponding to one homoclinic connection.  See Figure \ref{f.p.HE}.
  
  \begin{figure}[th]
\centering
\includegraphics[scale=0.15]{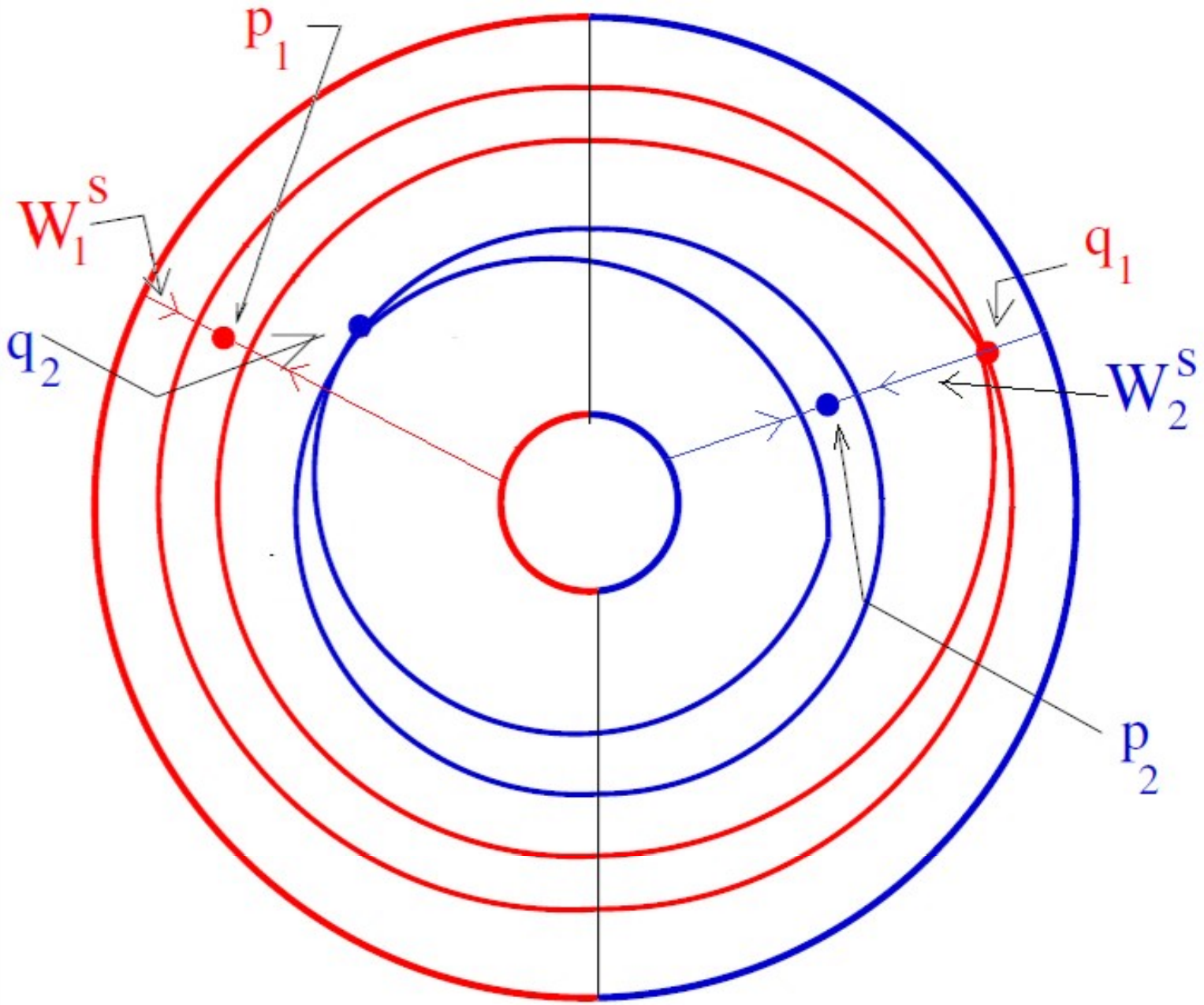}
		\hspace{0.5cm}
\includegraphics[scale=0.15]{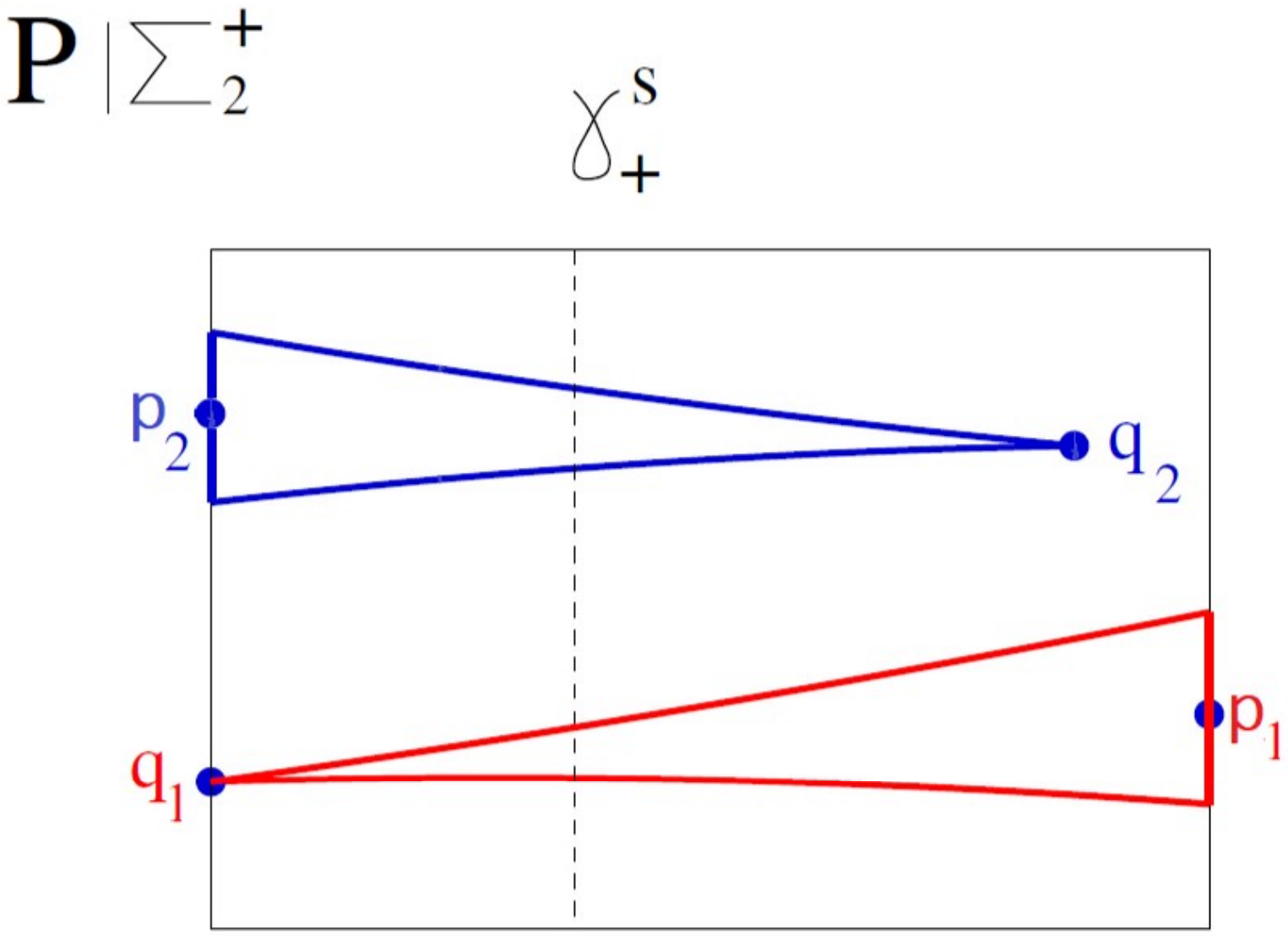}
\hspace{0.3cm}
\includegraphics[scale=0.15]{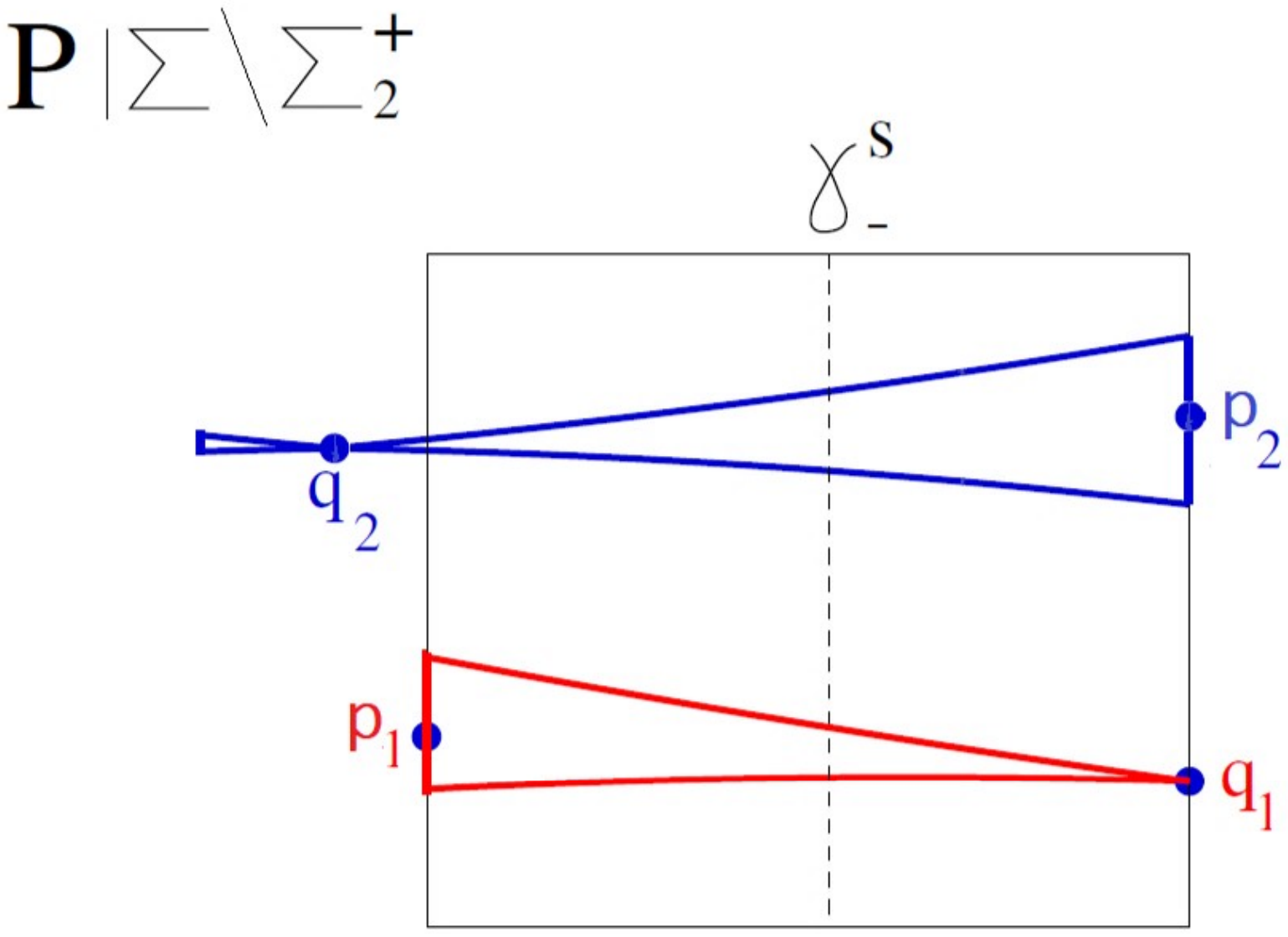}
\caption{{The left hand diagram displays $P(\Sigma)$ for  $X\in \cH\cE_1\setminus (\cH\cE_1\cap\cH\cE_2)$ and $q_2\in \Sigma_2^+$. The  diagram at the middle displays the restriction of $P|_{\Sigma^+_2}$  and the right hand diagram displays the restriction of $P$ to $\Sigma \setminus \Sigma^+_2$.} }\label{f.p.HE}
\end{figure}
 
  \item  The homoclinic class $K_-$ is a \emph{singular fake horseshoe} that intersects $\Sigma_-$ and is disjoint from $\Sigma_+$. 
  \item The intersection $K_-\cap K_+\cap \Sigma$ is contained in $W^s_2$ and consists of $p_2$ and the orbit of $q_1$. 
 \end{itemize}
 \end{prop}
 \begin{proof} We refer the reader to the pictures illustrating the restriction of the return map $P$ to $\Sigma^+_2$ and to $\Sigma\setminus \Sigma^+_2$. See Figure \ref{f.p.HE}.
 \end{proof}
 
 {A similar result holds  for $X \in \mathcal{HE}_{i} \setminus (\mathcal{HE}_1 \cap \mathcal{HE}_2)$ and $q_i \in \Sigma_j$, $i \in \{1,2\}$ and $j \in \{-,+\}$.

 %%%%%%%%%%%%%%%%%%%%%%%%%%%%%%%%%%%%%%%%%%%%
 \subsection{Fat Lorenz attractor:  vector fields in $\cH^{1,2}_+$  and $\cH^{1,2}_-$}
 %%%%%%%%%%%%%%%%%%%%%%%%%%%%%%%%%%%%%%%%%%%%

{We start this section setting the concept of "fat Lorenz attractor".}
\begin{defi} \label{d-fat} Consider a vector field $X\in \cO_\varphi$ and  $\Lambda$ its  maximal invariant (compact) set in $U$.  We say that $\Lambda$ is a \emph{fat Lorenz Attractor} if  every unstable segment in $\Sigma\setminus\gamma^s_+$  (resp. $\Sigma\setminus\gamma^s_-$) has an iterate by $P$ which  cuts every stable leaf  excepted $\gamma^s_+$ (resp. $\gamma^s_-$),  See Figure \ref{f.p.Hi2}.
\end{defi}

Notice that Definition~\ref{d-fat} implies that $\Lambda$ is a transitive attractor. 

 \begin{figure}[th]
\centering
\includegraphics[scale=0.17]{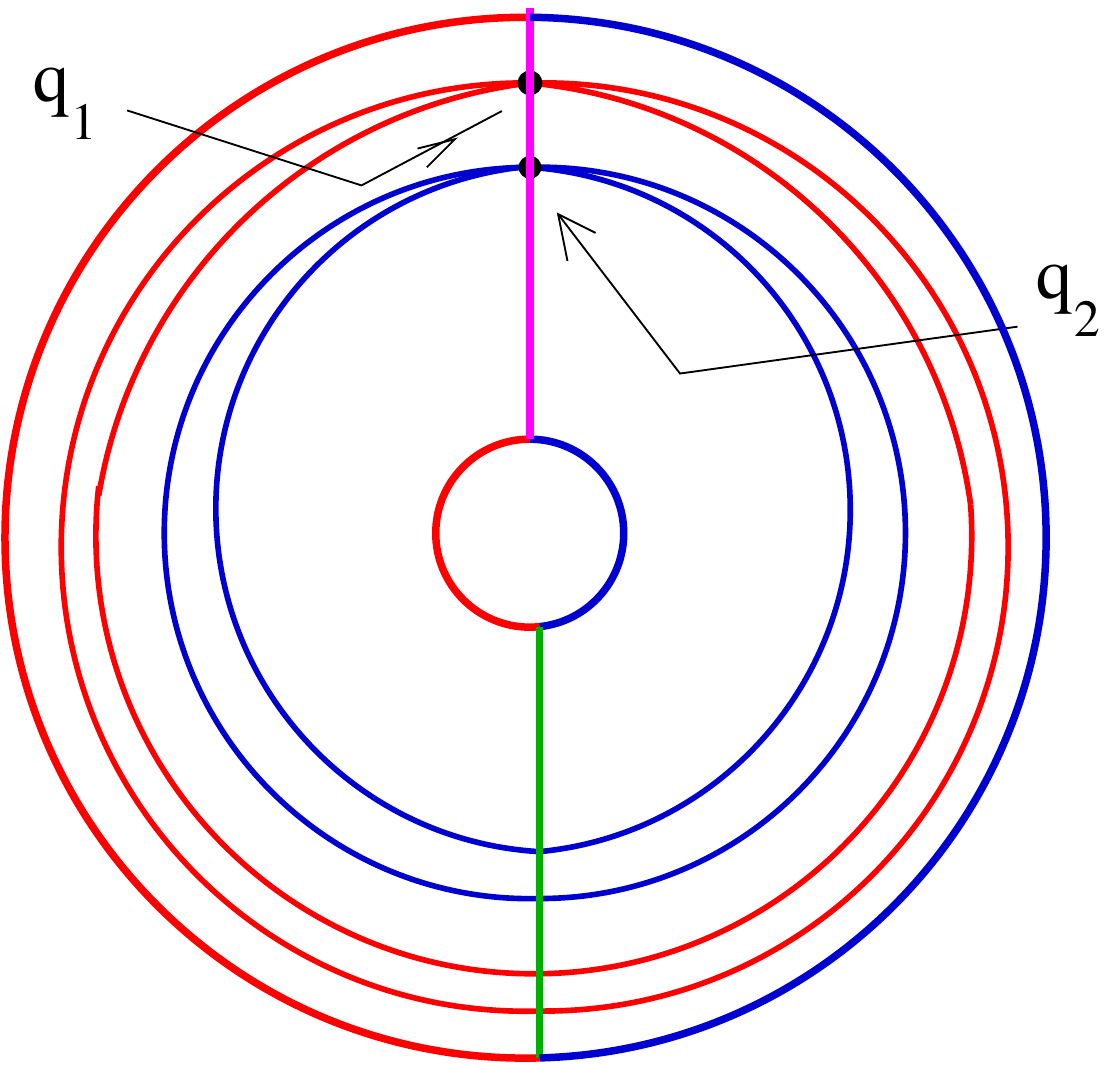}
\hspace{0.2cm}
\includegraphics[scale=0.17]{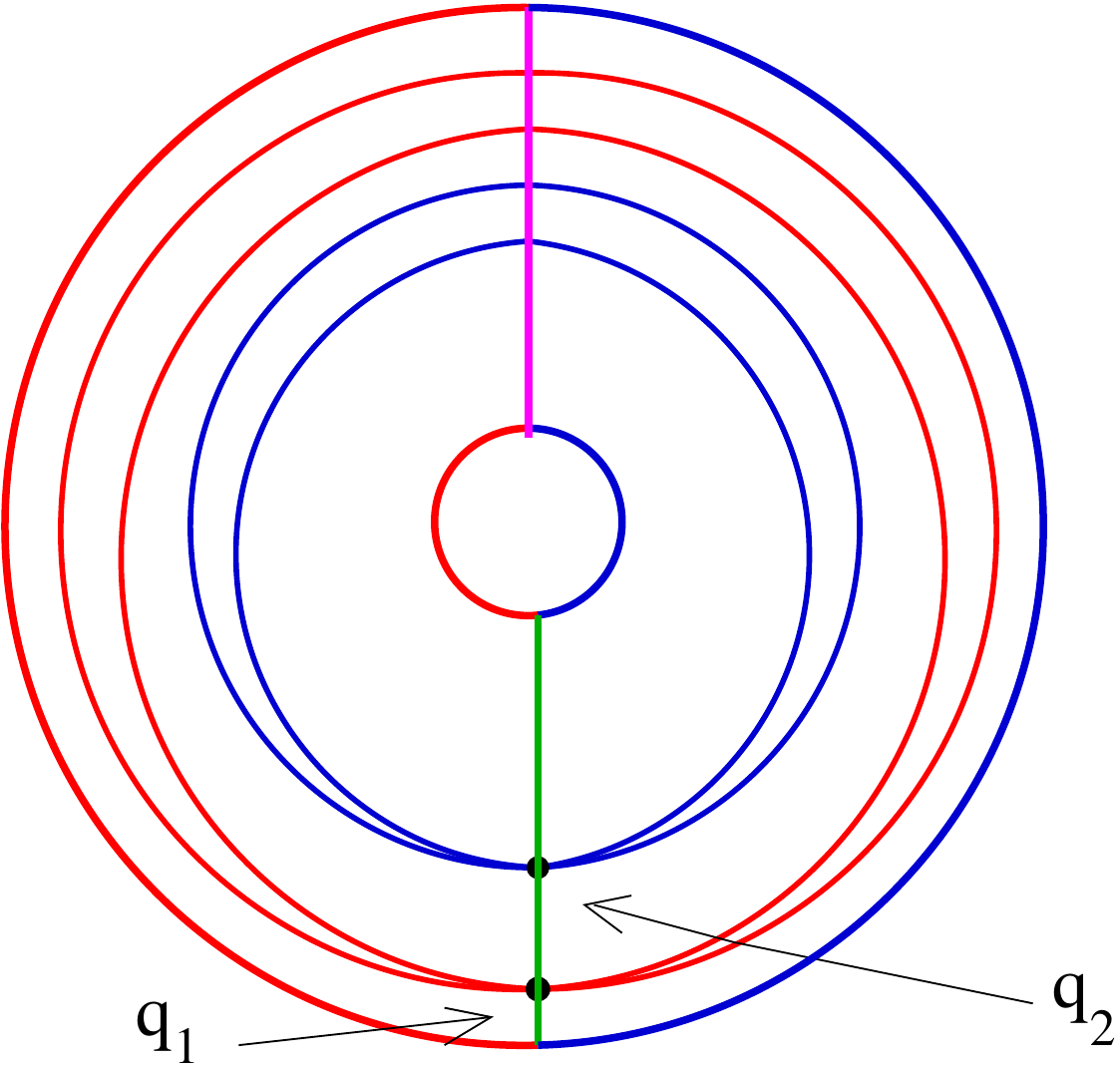}
\caption{The maximal invariant set in $U$ of $X \in \cH^{1,2}_+ \cup \cH^{1,2}_-$ is a fat Lorenz.}\label{f.p.Hi2}
\end{figure}

{ 
\begin{prop}\label{p.H12+} If $X\in\cH^{1,2}_+$ (resp. $X\in\cH^{1,2}_-$) then the maximal invariant set in $U$ is a 
fat Lorenz attractor. 
 \end{prop}
 }

\begin{proof} We just refer the reader to the  pictures illustrating the return map $P$,  Figures \ref{f.p.Hi1} and~\ref{f.p.Hi2}. 
 \end{proof}

%%%%%%%%%%%%%%%%%%%%%%%%%%%%%%%%%%%%%%
\section{Collisions and collapses}\label{s-collisions}
%%%%%%%%%%%%%%%%%%%%%%%%%%%%%%%%%%%%%%%%
%%%%%%%%%%%%%%%%%%%%%%%%%%%%%%%%%%%%%
This section describes  the drastic changes of behaviour appearing in the topological dynamics when a family $X_\mu$ crosses the boundary of the regions we described in Section~\ref{s.cartografia}.

%%%%%%%%%%%%%%%%%%%%%%%%%%%%%%%%%
\subsection{Collisions}\label{ss.colisao}
%%%%%%%%%%%%%%%%%%%%%%%%%%%%%%%%%f

Let us consider a $1$ parameter family $X_\mu\in\cO^{+,+}_\varphi$, $\mu\in[-1,1]$  with the following properties: 

\begin{itemize}\item  The family  crosses the hypersurface transversely  $\cH\cE^1_+$ at $\mu=0$:  this means that $X_0$ exhibits a heteroclinic connection $q_1\in W^s_2$ and $q_2\in\Sigma_+$.  
 \item for $\mu<0$ the vector field $X_\mu$ belongs to $\cL^+$:  it has an up Lorenz attractor and $\Sigma_-$ contains a fake horseshoe for $P$. 
 \item  for $\mu>0$ the vector field $X_\mu$ has a two-sided Lorenz attractor. 
\end{itemize}

In other words, $X_\mu$ consists of a generic unfolding of the heteroclinic connection $q_1\in W^s_2$.  
The cuspidal point $q_{1,\mu}$ is moving with the parameter $\mu$ and is crossing transversely $W^s_{2,\mu}$ (stable leaf through $p_{2,\mu}$) for $\mu=0$. 
For $\mu<0$, the upper-Lorenz attractor is intersecting exactly every stable leaf in $\Sigma^+$ between  $q_{2,\mu}$ and $q_{1,\mu}$, and the horseshoe in $\Sigma_-$ is bounded by  the stable leaves $W^s_{1,\mu}$, $W^s_{2,\mu}$ and contains the periodic points $p_{1,\mu}$ $p_{2,\mu}$. 

When $\mu$ tends to $0$, the point $q_{1,\mu}$ tends to a point $q_{1,0}$ in $W^s_2$, and this point $q_{1,0}$ is also the limit of the intersection of $W^s_2$ with one of the rectangles (the one containing $p_1$) of the Markov partition of the fake horseshoe. This corresponds to a Cantor set in $W^s_2$ of points of the fake horseshoe for $X_\mu$, $\mu<0$ whose diameter tends to $0$ as $\mu\to 0$: all the Cantor set tends to $q_{1,0}$. 

For the parameter $0$, the Lorenz attractor no longer admits an attracting neighborhood and is no longer robust. The fake horseshoe becomes singular and intersects the Lorenz attractor along $\sigma$ and the orbits of $p_{2,0}$ and 
of $q_{1,0}$. 
When $\mu>0$ the Lorenz attractor and the (singular) fake horseshoe merge into a two-sided Lorenz attractor. 

Notice that if all the vector fields $X_\mu$ are assumed to be of class $C^2$, then all along the parameter, $X_\mu$ admits a unique SRB measure $\nu_\mu$ (see \cite{APPV09}). For $\mu\leq 0$ the support of $\nu_\mu$ is the Lorenz attractor, and in particular, does not intersect  $\Sigma_-$.  
For $\mu>0$, the support of $\nu_\mu$ intersects any stable leaf in $\Sigma$. 
Thus the support of $\nu_\mu$ changes drastically at  the collision point. 

\begin{ques} Is the map $\mu\mapsto\nu_\mu$ continuous (for the weak topology) at $\mu=0$?
\end{ques}

%%%%%%%%%%%%%%%%%%%%%%%%%%%%%%%%%%%%%%%%%%%%%
\subsection{The collapse of the horseshoe: parameter families crossing $\cH^{1,2}_+$ or $\cH^{1,2}_-$}
%%%%%%%%%%%%%%%%%%%%%%%%%%%%%%%%%%%%%%%%%%%%%

Recall that $\cH^{1,2}_+$ is the codimension $2$ submanifold corresponding to the double homoclinic connection $q_1,q_2\in\gamma^s_+$.  We consider here a $2$-parameter family $X_{\mu}$, $\mu=(\mu_1,\mu_2)\in[-1,1]^2$ which is a generic unfolding of this double homoclinic connection: in other words, the family cuts transversely $\cH^{1,2}_+$ at $\mu=(0,0)$. We assume furthermore that, for any fixed $\mu_2$, the $1$ parameter family $X_{\mu_1,\mu_2}$ is transverse to $\cH^1$ at $\mu_1=0$ and, reciprocally, for any fixed $\mu_1$, the $1$ parameter family $X_{\mu_1,\mu_2}$ is transverse to $\cH^1$ at $\mu_2=0$. More precisely, one may assume that 
$X_{\mu_1,\mu_2}\in\cO_\varphi^{\omega_1,\omega_2}$ where $\omega_i\in\{-,+\}$ is the sign of $\mu_i$. 
When one unfolds this double homoclinic connection, if $q_{1,\mu}$ enters in $\Sigma_2$ (that is, $\mu_1>0$), then a fixed point $p_{1,\mu}\in\Sigma_1$ is created and tends to $q_{1,0}\in \gamma^s_+$ as $\mu$ tends to $0$.   In the same way, if $\mu_2>0$  then $q_2$ enters in $\Sigma_1$ and the fixed point $p_{2,\mu}$ is created in $\Sigma_2$ and tends to $q_{2,0}\in\gamma^s_+$. 
The stable leaves $W^s_{1,\mu}$ and $W^s_{2,\mu}$ bound the small strip $\Sigma^+$ in $\Sigma$ containing $\gamma^s_+$  (tending to the segment $\gamma^s_+$ as $\mu\to0$), and a large strip $\Sigma^-$. 
The vector field $X_\mu$ in $\cO^{+,+}_\varphi$ belongs to $\cL^+\cup\cL^-$ if and only if both $q_1$ and $q_2$ belong to the same component  $\Sigma^+$ or $\Sigma^-$.  
\begin{lema}\label{l.doubleconnection} For small $\mu$, $\{q_1,q_2\}$ is not contained in $\Sigma^+$. In other words, the closure of $\cL^+$ is disjoint from $\cH^{1,2}_+$. 
 $\{\mu|\{q_{1,\mu},q_{2,\mu} \}\subset \Sigma_-\}$ is an open subset containing $(0,0)$ in its closure. More precisely, for any $\alpha >0$, there is $\mu_0$ so that for any $0<\mu<\mu_0$   the vector field $X_{\mu,\alpha\mu}\in\cL^-.$ 
 \end{lema}
\begin{proof} 

Assume (arguing by contradiction) that  $\{q_1,q_2\}\subset \Sigma^+$.  Thus $X\in\cL^+$. Then  $\Sigma^+$ is invariant under the action of $P$.  However, as $\Sigma^+$ is contained in an arbitrarily small neighbourhood of $\gamma^s_+$, the rate of expansion of $P$ in the unstable cone is arbitrarily large, in particular $\gg 2$. So $\Sigma^+$ cannot be invariant, ending the proof. 

The proof of the second point is as follows: consider one half-line $X_{\mu,\alpha\mu}$ and consider $\mu>0$ very small. 
Thus the point $q_1$ (resp. $q_2$) belongs to $\Sigma_2$  (resp. $\Sigma_1$)  and its distance to $\gamma^s_+$ is  $\mu$ (resp. $\alpha\mu$). 
Consider any point $p$ in $\Sigma_1$ at a small distance $>\frac12\alpha\mu$ from $\gamma^s_+$. Consider an unstable segment $\gamma_p$ realizing this distance so that $\gamma_p$ starts at $\gamma^s_+$ and ends at $p$.  Then $P(\gamma_p)$ is an unstable segment starting at $q_1$  and ending at $P(p)$.  The length  
$\ell(P(\gamma_p))$ is arbitrarily larger than $\ell(\gamma_p)$ (say, larger than $100 (1+\alpha)\mu\ell(\gamma_p)$ as  $\mu$ tends to $0$ (as the expansion rate close to $\gamma^s_+$ tends to infinity). 

One deduces that the point $P(p)$ belongs to $\Sigma_1$ and is at a distance larger than $ 99(1+\alpha)\mu$ of $\gamma^s_+$.  One deduces that $P(p)$ (and thus $p$) does not belong to the stable leaf through the fixed point $p_1\in \Sigma_1$.  In particular $q_2$ belongs to $\Sigma_-$.  
One proves in the same way as for $\mu>0$ small $q_1\in\Sigma_-$, and thus the vector filed belongs to $\cL^-$, ending the proof. 
Figure \ref{f.l.doubleconn} displays a local bifurcation diagram in this case.
\end{proof}

The lemma implies that, when one considers a segment in the parameter plane crossing $(0,0)$ and entering in $\cO^{+,+}_\varphi$ transversely to both $\cH^1$ and $\cH^2$, then one creates a down Lorenz attractor, which cuts every stable leaf in $\Sigma^-$.  When the parameter tends to $(0,0)$,  then $\Sigma^-$ tends to $\Sigma$ and the Lorenz attractor tends to what we call the fat Lorenz attractor.  The horseshoe, corresponding to $\Sigma^+$ collapses into the double homoclinic connection. 

%%%%%%%%%%%%%%%%%%%%%%%%%%%%%%%%%%%%%%%%%%%%%
\subsection{Switching from upper to down:  family crossing $\cH\cE_1\cap \cH\cE_2$}
%%%%%%%%%%%%%%%%%%%%%%%%%%%%%%%%%%%%%%%%%%%%%

In this section, we consider a two-parameter family $X_\mu\in \cO_\varphi^{+,+}$, $\mu=(\mu_1,\mu_2)$ crossing transversely $\cH\cE_1\cap \cH\cE_2$ at $\mu=0$. 
In other words, $X_{(0,0)}$ exhibits two heteroclinic connections $q_1\in W^s_2$ and $q_2\in W^s_1$ (recall that $W^s_i$ is the stable leaf through the fixed point $p_i$). We have seen that this behaviour implies that $X_{(0,0)}$ has two \emph{full Lorenz} which intersect along  $W^s(p_1)\cup W^s(p_2)\cup W^u(\sigma)$.  See Figure \ref{f.p.HE1HE2}.
Up to reparametrize, the family one may assume that 
$$\begin{array}{c}
\left(\mu_1=0 \Leftrightarrow q_1\in W^s_2\right) \mbox{ and }\left(\mu_2=0 \Leftrightarrow q_2\in W^s_1\right)\\
\left(\mu_1>0 \Leftrightarrow q_1\in \Sigma^+\right) \mbox{ and }\left(\mu_2>0 \Leftrightarrow q_2\in \Sigma^+\right)
\end{array}$$

\begin{figure}[th]
\centering
	\includegraphics[scale=0.17]{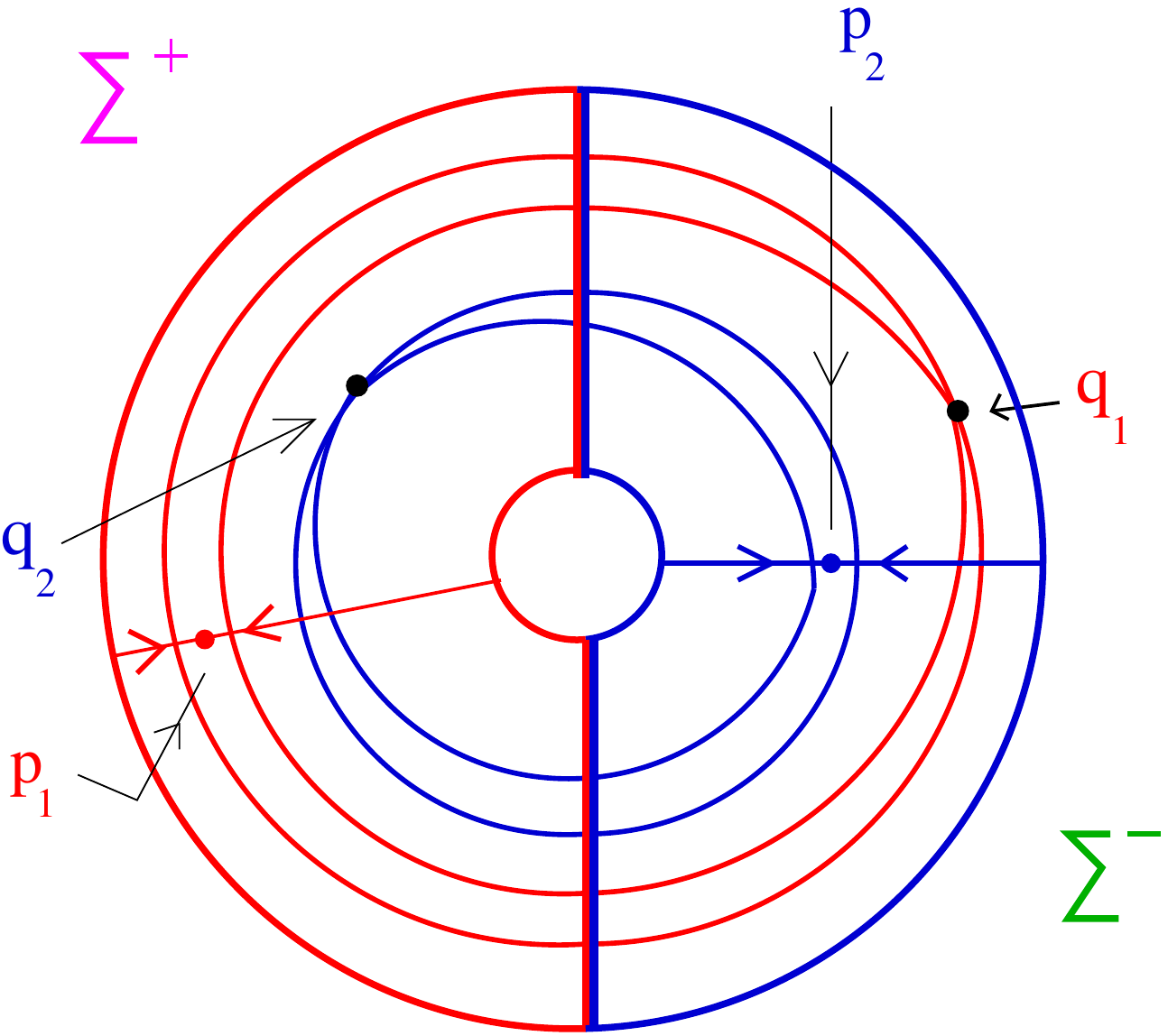}
		\hspace{0.3cm}
\includegraphics[scale=0.17]{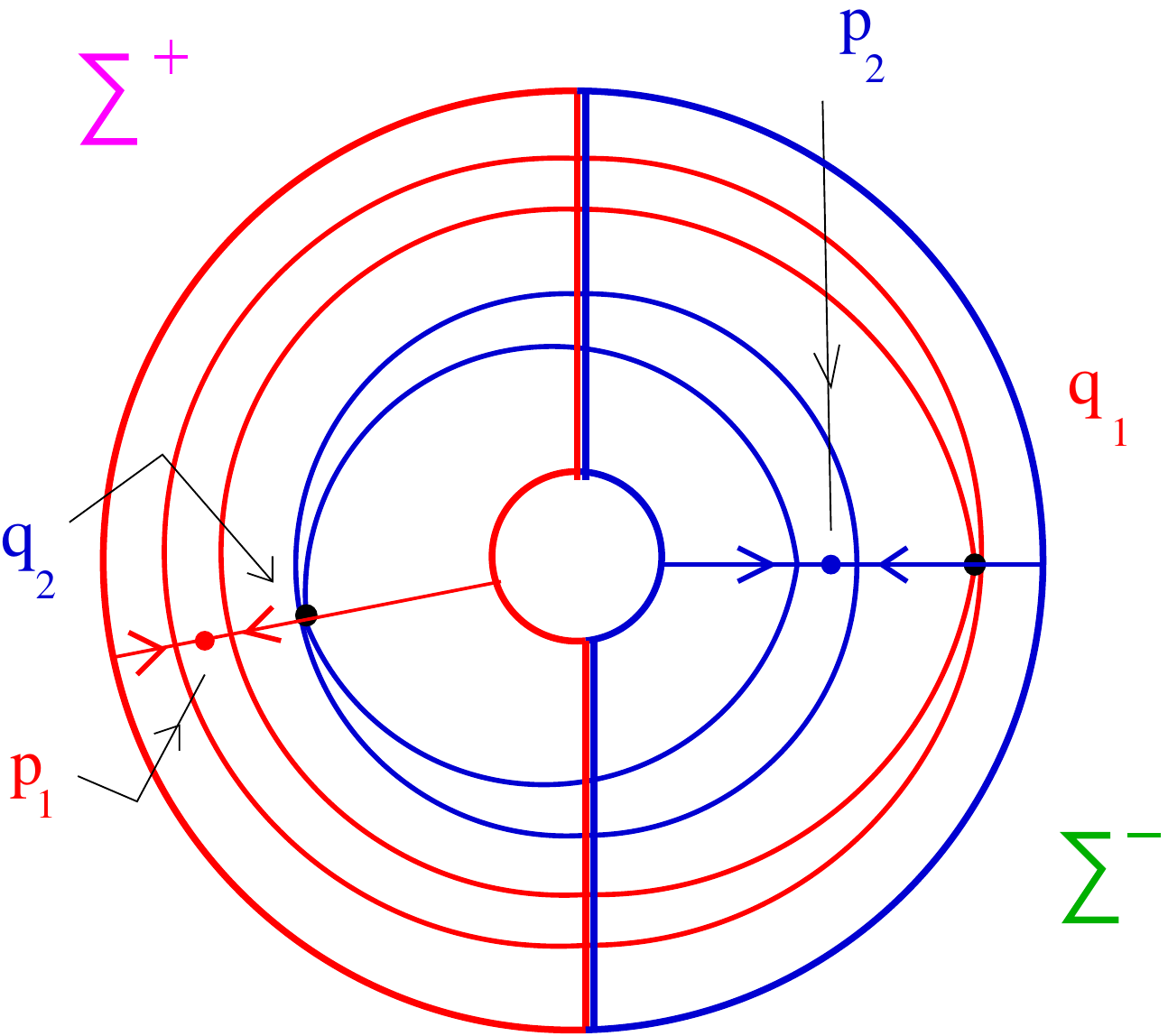}\
\hspace{0.3cm}
\includegraphics[scale=0.17]{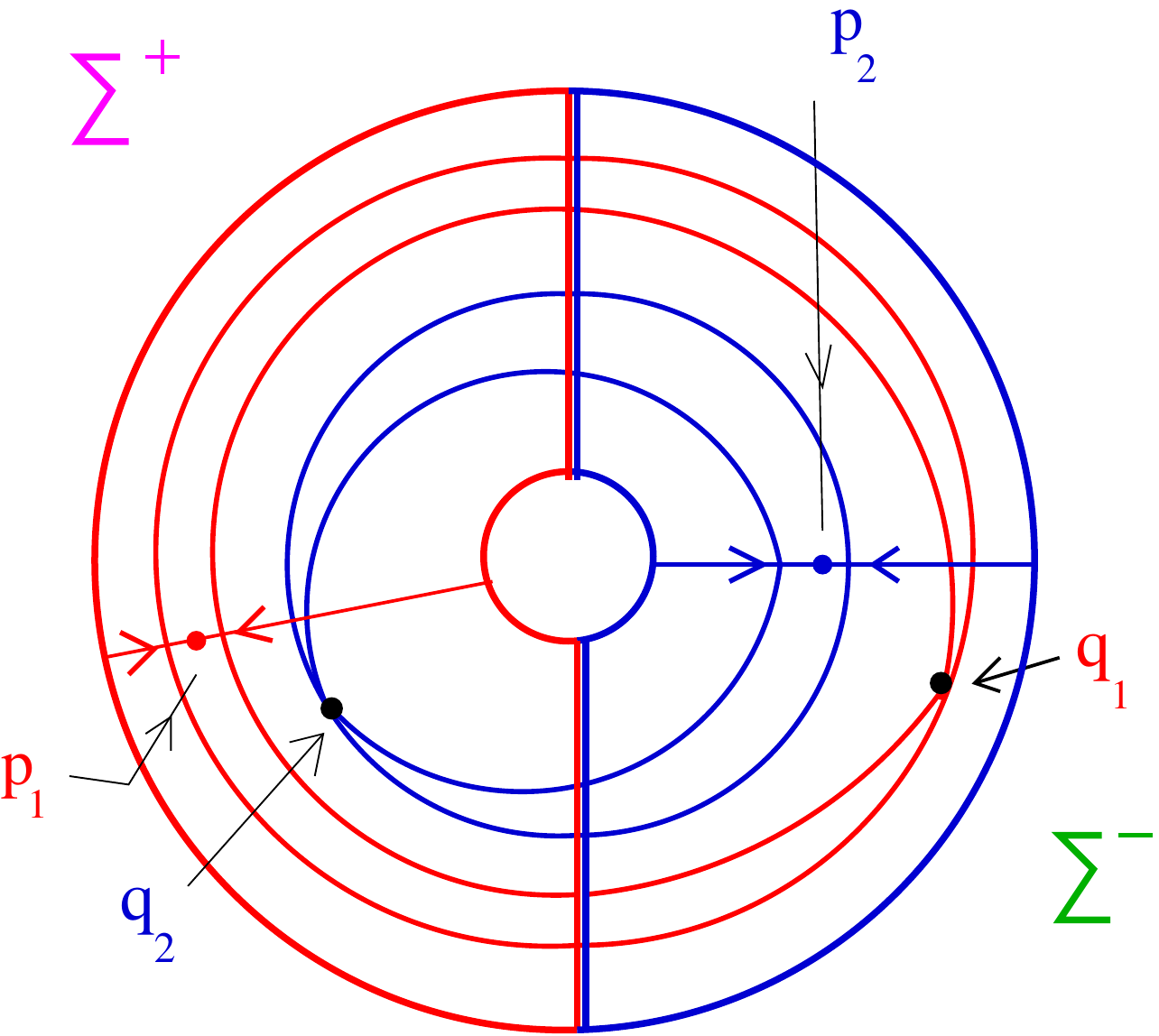}
\caption{Local bifurcation: (a) $(\mu_1,\mu_2), \mu_i >0$, $(\mu_1,\mu_2), \mu_i =0$, (c) $(\mu_1,\mu_2), \mu_i < 0$.}
\label{f-parametro}
\end{figure}

With these notations, note that Theorems~\ref{igual-a-l.switch} and \ref{1-igual-a-l.switch} are a reformulation of the results in the previous sections. 

%%%%%%%%%%%%%%%%%%%%%%%%%%%%%%%%%%%%%%%%%%%%%
\section{Parameter families $X_\mu\in\cO_\varphi$, with parameters in the torus}
%%%%%%%%%%%%%%%%%%%%%%%%%%%%%%%%%%%%%%%%%%%%%
\subsection{$\TT^2$-parameter families}
%%%%%%%%%%%%%%%%%%%%%%%%%%%%%%%%%%%%%%%%%%%%%

\begin{defi} Let $\pi\colon \RR^2\to\TT^2=\RR^2/\ZZ^2$ be the canonical projection. 
Let $V\subset \RR^2$ be an open subset so that the projection 
$\pi(V)$ is the whole torus $\TT^2$. 
We say that a family $\{X_\mu\in\cO_1\}_{\mu\in V}$ is a $C^r, r\geq 0,$ family of vector fields in $\cO_1$ parametrized by $\TT^2$ if 
\begin{itemize}
\item the map $\mu\to X_\mu$ is continuous for the $C^1$-topology. 
 \item the map $(p,\mu)\mapsto X_\mu(p)$ is of class $C^r$. 
 \item for any $\mu,\mu'$ so that $\mu'-\mu\in\ZZ^2$ one has:  the return maps 
 $P,P'$ of $X,X'$ on the transverse cross-section $\Sigma$ coincide: 
 $P=P'$.  In particular, the restriction of $X$ and $X'$  to the attracting region $U$  are smoothly topologically equivalent, by an equivalence whose restriction to the cross-section $\Sigma$ is the identity map.  
\end{itemize}

 Shortly, we say that $X_\mu$ is a $\TT^2$-parameter family. 
\end{defi}

%%%%%%%%%%%%%%%%%%%%%%%%%%%%%%%%%%%%%%%%%%%%%
\subsection{Essential families}
%%%%%%%%%%%%%%%%%%%%%%%%%%%%%%%%%%%%%%%%%%%%%
This section aims  to define the notion of essential $\TT^2$ families.
Roughly, we don't want  the family $X_\mu$ to be contained in a small neighbourhood of a given vector field $X_0$.  
We want that, when $\mu$ follows a simple closed  path in $\TT^2$, non-homotopic to a point, then the images $P(\Sigma^1)$ or $P(\Sigma^2)$ give a turn in $\Sigma$ in an essential way. 

Let us first present a non-intrinsic definition of  these phenomena. 
Afterwards, we will see  an intrinsic definition, showing that the definition does not depend on the choices. 
Consider a $\TT^2$-parameter family $\{X_\mu\}$ and let $\gamma^s_{+,\mu}$ and $\gamma^s_{-,\mu}$ the associated stable leaves corresponding to the discontinuities of the first return map. The leaves $\gamma^s_{+, \mu}$ vary continuously with $\mu$.  
This allows us to choose a parametrization of $\Sigma$, depending on $\mu$, and so that $\gamma^s_{+,\mu}$ is the segment $\{0\}\times [-1,1]$ in the annulus $\Sigma=\RR/\ZZ\times [-1,1]$. 

Consider now the points $q_{1,\mu}$ and $q_{2,\mu}$ (first intersection points of the unstable separatrices of $\sigma_\mu$ with $\Sigma$).  This defines two 
continuous maps 
$q_1\colon \TT^2\to \RR /\ZZ\times [-1,1] \mbox{ and } q_2\colon \TT^2\to \RR/\ZZ\times [-1,1]$. 
Then, composing $q_1$ and $q_2$ by the projection $\psi\colon \SS^1\times [-1,1]\to \SS^1$ one gets two continuous maps   $\psi \circ q_i\colon \TT^2\to \SS^1$, $i=1,2$. 
$\mbox{Let us denote}\quad Q=\left(\psi\circ q_1, \psi\circ q_2\right) \colon \TT^2\to \TT^2.$
\begin{defi}\label{d.grau} With the notation above, we say that the family $X_\mu$ is essential if the topological degree of $Q$ is  $1$ (or else, if $Q$ is homotopic to an orientation preserving homeomorphism). 
\end{defi}

{ In Definition~\ref{d.grau}, the homotopy class of $Q$ depends on the parametrization of $\Sigma$ that we choose.  Let us convince you that the topological degree of $Q$ does not depend on the choice of  parametrization of $\Sigma$.}

{ Consider the hypersurface  $\Gamma^s_+\subset \Sigma \times \TT^2$ defined by $\Gamma^s_+= \bigcup_{\mu} \gamma^s_{+,\mu}\times\{\mu\}$. In the parametrization above it is   $\{0\}\times [-1,1]\times \TT^2\subset\Sigma\times \TT^2$.  Thus it is a compact $3$-manifold with a boundary homeomorphic to $[-1,1]\times \TT^2$, whose boundary is contained in the boundary of $\Sigma\times \TT^2$.  As a consequence, the algebraic intersection number of a closed curve in $\Sigma\times \TT^2$ with $\Gamma^s_+$  is well defined. }

To every closed path $c\colon\SS^1\to \TT^2$ let us consider the closed paths
$q_{i,c}\colon  \SS^1\to \Sigma\times \TT^2$, $i\in\{1,2\}$  so that $q_{i,c}(\theta)= \left( q_{i,c(\theta)},c(\theta)\right)$ for $\theta\in\SS^1$.  

 Notice that $q_{i,c}$ depends continuously on the closed path $c$.  In particular, its algebraic intersection number with the hypersurface $\Gamma^s_+$ only depends on the homology class $[c]$.  We denote it $[q_{i,c}]\cdot\Gamma^s_+$. 
 One gets a map 
 $$[c]\mapsto \left([q_{1,c}]\cdot\Gamma^s_+,[q_{2,c}]\cdot\Gamma^s_+\right)\in\ZZ^2$$  defined on $H_1(\TT^2,\ZZ)\simeq \ZZ^2$ with values in $\ZZ^2$,  and this map is a linear map given by a $2$ by $2$ matrix with $\ZZ$-entries.  The topological degree of $Q$  is  the determinant of this linear map:  the family is essential if and only if the determinant is $1$, which does not depend on the parametrization  choice.  
 $\square$

\subsection{Building  $\TT^2$-families}

Consider a vector field $X\in\cO_1$ on $\mathbb{R}^3$. Thus, by definition, it is transversal to the annuli $\Sigma$, $D_1$ and $D^2$.  
Furthermore, the first return map  %of $D_i$ to $\Sigma\cup D_1\cup D_2$ is  smooth 
$R\colon D_1\cup D_2\to\Sigma$ is smooth and maps $D_1$ and $D_2$ on two disjoint essential annuli in  the interior of $\Sigma$. 

Consider an annuli $\tilde D_i$, $i=1,2$, containing $D_i$ in its interior, and so that $R$ extends in a diffeomorphism $ \tilde R\colon \tilde D_1\cup \tilde D_2\to \Sigma$ which is the first return map from $\tilde D_i$ to $\Sigma\cup \tilde D_1\cup \tilde D_2$. 
We denote by $\Delta_i$ the union of the $X$-orbit segment joining points $p\in \tilde D_i$ to $\tilde R(p)$. Thus $(\Delta_i,X)$ is smoothly orbitally equivalent to $(\tilde D_i\times [0,1], \frac\partial{\partial t}).$ 
The next lemma expresses that one can realize by a continuous family of vector fields in $\mathbb{R}^3$ any continuous deformation of the return map $R$. 

\begin{lema} \label{l.family}Consider a continuous family $R_\mu\colon D_1\cup D_2\to \tilde R(\tilde D_1\cup \tilde D_2)$, $\mu\in V$, where $V\subset \RR^2$ is an open  disk containing $0$. One assumes $R_0=R$. 
Then there is a family of vector fields $X_\mu$ with the following properties:
\begin{itemize}
\item $X_0=X$
 \item for any $\mu$, $X_\mu$ satisfies all the topological properties (items 1 to 7)) of the definition of the set $\cO_1$.
 \item for any $\mu$, $X_\mu$ coincides with $X$ out of $\Delta_1\cup\Delta_2$. 
 \item for any $\mu$,  the restriction of $X_\mu$ to $\Delta_i$ is transverse to the fibers $\tilde D_i\times \{t\}$
 \item the return map of $X$ from $D_1\cup D_2$ to $\Sigma$ is $R_\mu$. 
\end{itemize}
\end{lema}
\begin{proof}
 One extends $R_\mu$ to $\tilde D_i$ in a diffeomorphism $\tilde R_\mu$  so that all $\tilde R_\mu$ coincide with $\tilde R$ in a neighbourhood of the boundary $\partial \tilde D_i$. 
 Then we replace the restriction of $X$ in $\Delta_i$ with a vector field that coincides with $X$ in a neighbourhood of $\partial \Delta_i$ and which entrance exit map is $\tilde R_\mu$. 
\end{proof}

Assume now that the projection of $V$ on $\RR^2/\TT^2$ covers the whole torus $\TT^2$, and assume that the family $R_\mu$ is $\ZZ^2$-periodic  in the following sense:
\begin{itemize}
 \item for any $\mu_1,\mu_2\in V$ so that $\mu_2-\mu_1\in \ZZ^2$ then $R_{\mu_1}=R_{\mu_2}$. 
\end{itemize}
Then  the family $X_\mu$ defined in Lemma~\ref{l.family} is a $\TT^2$ parameter family of vector fields having $U$ as an attracting region. 

%%%%%%%%%%%%%%%%%%%%%%%%%%%%%%%%%%%%%%%%%%%%%
\subsection{An essential $\TT^2$-family of vector fields in $\cO_\varphi$} 
%%%%%%%%%%%%%%%%%%%%%%%%%%%%%%%%%%%%%%%%%%%%%

Consider $\Sigma\simeq \SS^1\times [-1,1]$ with the coordinates $\theta, t$ and consider the constant cone field $\cC$ defined by 
$$\cC(p)=\{u=\alpha\frac{\partial}{\partial \theta}+\beta\frac{\partial}{\partial t}\in T_p\Sigma \mbox{ so that } |\alpha|\geq |\beta|\}.$$

\noindent {We denote by $\cR_\alpha\colon \Sigma\to\Sigma$ the rotation of angle $\alpha \in\SS^1$, that is, $(\theta, t)\mapsto (\theta+\alpha,t)$. Recall that $R\colon D_1\cup D_2\to\Sigma$ is the first return map  of $D_i$ to $\Sigma\cup D_1\cup D_2$.}

Consider a vector field $X\in \cO_\varphi$ with the  following extra properties
\begin{itemize}
 \item The constant cone field $\cC$ is the unstable cone field $\cC^u$.
 \item The images $R(D_i)$ are product annulus $\SS^1\times I_i \subset\Sigma=\SS^1\times [-1,1]$.
\end{itemize}

Now  consider the $\ZZ^2$-family of maps $R_{\alpha,\beta}\colon D_1\cup D_2\to \Sigma$, $\alpha,\beta\in \TT^2=\RR^2/\SS^2$ defined by 
$$R=\cR_\alpha\circ R\mbox { on } D_1, \mbox{ and } R=\cR_\beta\circ R\mbox{ on } D_2$$

According to Lemma~\ref{l.family}, one can realize the periodic family $R_\mu$ by a $\TT^2$-parameter family of vector fields $X_\mu$, and one easily check

\begin{lema} The family $X_\mu$ is a $\TT^2$-parameter essential family in $\cO_\varphi$. 
 
\end{lema}

%%%%%%%%%%%%%%%%%%%%%%%%%%%%%%%%%%%%%%%%%%%%%
\section{Reduction to a $1$-dimensional dynamics}
%%%%%%%%%%%%%%%%%%%%%%%%%%%%%%%%%%%%%%%%%%%%%
\subsection{The action of the return map on the space of stable leaves}
%%%%%%%%%%%%%%%%%%%%%%%%%%%%%%%%%%%%%%%%%%%%%

Any $C^1$ vector field $X\in\cO_1$ is singular hyperbolic in the attracting region $U$, with a continuous strong stable direction. 
There is a well-defined \emph{stable foliation} (also called \emph{strong stable foliation} tangent to the stable distribution with leaves having the same regularity as $X$. 
It admits, therefore, a well-defined $2$-dimensional \emph{(weak)stable foliation} (also called \emph{center-stable foliation}) out of the strong stable manifold of the singularity. The leaves of the weak stable foliation are the orbits for the flow of the leaves of the strong stable: out of the strong stable manifold of $\sigma$, the vector is not tangent to the stable direction, so  these orbits are $2$-dimensional. 
Along the strong stable manifold of $\sigma$, the vector field is tangent to the one-dimensional strong stable leaf, so the foliation is singular. 
The strong stable foliation is not tangent to $\Sigma$. However, the $2$ dimensional center-stable foliation  cuts the annulus $\Sigma$ transversely along a  one-dimensional foliation $\cF^s$, which is  the stable foliation of the return map $P$. 
 The segments $\gamma^s_+$ and $\gamma^s_-$ are leaves of $\cF^s$.

{The foliation $\cF^s$ is transverse to the unstable cone field $\cC^u$.  
 The leaves of the foliation $\cF^s$ are segments crossing $\Sigma$ (connecting the two boundary components of $\Sigma$). 
  The leaves space  $\Sigma/\cF^s$ is a (topological) circle $\SS^1_X$.  The leaves $\gamma^s_+$ and $\gamma^s_-$ induce each a point $c_+$ and $c_-$ respectively, on $\SS^1_X$. 
Note that the flow $X^t$ preserves the center-stable foliation, and thus the first return map $P$ preserves the foliation $\cF^s$.
 As a consequence, $P$ passes to the quotient as a map $f=f_X$ defined from $\SS^1_X\setminus\{c_+,c_-\}$ to $\SS^1_X$. 
 As $P(\Sigma^i)$ is an essential pinched annulus in $\Sigma$, we deduce that $f_X$ restricted to each interval of $\SS^1_X\setminus\{c_+,c_-\}$ is a diffeomorphism on  a punctured circle. }
 
 %%%%%%%%%%%%%%%%%%%%%%%%%%%%%%%%%%%%%%%%%%%%
 \subsection{Increasing the regularity of the foliation}

 In a recent work \cite{ararujomelbourne}, Ara\'ujo and Melbourne adapt to our setting  a condition from \cite{HPS},  ensuring the smoothness of the strong stable foliation of $X$.
 \begin{itemize}\setcounter{enumi}{9}
 \item \label{i.am}  there exists $t>0$ such that
		\begin{equation}\|DX^t|_{E^s_x}\|\cdot \|DX^{-t}|_{E^{cu}_{X^t(x)}}\|\cdot \|DX^t|_{E_x^{cu}}\|<1 \label{eq.cont.}\quad \mbox{ for all $x \in \Lambda$}. \end{equation}
\end{itemize}
 \noindent As a consequence of Theorem \ref{holonomia},
 %item~\ref{i.am}), 
 the weak stable foliation of $X$, and 
 the stable foliation $F^s$ of the first return map $P$ is of class $C^1$.  Therefore the circle $\SS^1_X$ is endowed with a natural $C^1$-structure, and $f_X$ is of class $C^1$. 

%%%%%%%%%%%%%%%%%%%%%%%%%%%%%%%%%%%%%%%%%%%%%
\section{Symbolic dynamics and topological classification}\label{ss-symbolic}
%%%%%%%%%%%%%%%%%%%%%%%%%%%%%%%%%%%%%%%%%%%%%
As for the classical geometric model of Lorenz attractor, we will see in that  section that, to any vector field in $\cO_1$ one can associate  combinatorial data, called the itinerary of the discontinuities. Furthermore, these itineraries provide a topological classification of the vector fields in the attracting region $U$. 
We fix a vector field $X\in\cO_1$ and its return map 
$P\colon\Sigma\to\Sigma$ in this section. 

%%%%%%%%%%%%%%%%%%%%%%%%%%%%%%%%%%%%%%%%%%%%%
\subsection{Itineraries for the return map $P$ on $\Sigma$}
%%%%%%%%%%%%%%%%%%%%%%%%%%%%%%%%%%%%%%%%%%%%%

Recall that $\Sigma$ is endowed with two specific stable leaves $\gamma^s_+$ and $\gamma^s_-$, which split  $\Sigma$ in $\Sigma_1$ and $\Sigma_2$.
Note that $\gamma^ s_+$ cuts both pinched annuli $P(\Sigma_i)$ along one stable leaf, except in the case where $q_i\in \gamma^s_+$. 
Thus $P^{-1}(\gamma^s_+)$ cut $\Sigma_i$ into two components (one of them being empty if $q_i\in\gamma^s_+$): 

\begin{itemize}
 \item we denote by $A_0$ and $A_1$ the two components of $\Sigma_1\setminus P^{-1}(\gamma^s_+)$  where $A_0$  starts at $\gamma^s_+$ (for the positive orientation of the circle $\SS^1$ of $\Sigma=\SS^1\times [-1,1]$. 
 If $q_1\in \gamma^s_+$ then $A_1=\emptyset$. 
 \item we denote by $B_0$ and $B_1$ the two components of $\Sigma_2\setminus P^{-1}(\gamma^s_+)$  where $B_0$ is starting at $\gamma^s_-$.  
 If $q_1\in \gamma^s_+$ then $B_1=\emptyset$. 
\end{itemize}

\noindent We consider $\XX=\{A_0,A_1,B_0,B_1\}^\NN$ the space of  infinite positive words in the alphabet $\{A_0,A_1,B_0,B_1\}$. 
A finite word of length $k$ is an element $\{\omega_0,\dots,\omega_{k-1}\}$ of $\{A_0,A_1,B_0,B_1\}^{k}$.

Given $\omega=\{\omega_i\}_{i\in\NN} \in \XX$ we denote by $[\omega]_k$ the initial word $\{\omega_0,\dots,\omega_{k-1}\}$  
of length $k$ of~$\omega$. 

\begin{rema} For any $k>0$ and any $\varepsilon\in +,-$, $P^{-k}(\gamma^s_\varepsilon)$ consists in at most $2^k$ stable leaves. 
\end{rema}

\begin{defi}For every $p\in \Sigma\setminus \bigcup_{j=0}^{k} P^{-j}(\gamma^s_+\cup\gamma^s_-)$ we denote by $[\omega(p)]_k=\{\omega_0(p),\dots,\omega_{k-1(p)}\}$ the word defined by $f^i(p)\in \omega_i(p)$ (recall that $\omega_i(p)$ is one of the four regions $A_0, A_1,B_0,B_1$).
$[\omega(p)]_k$ is called the \emph{$k$-itinerary} of $p$. 
\end{defi}

Figure \ref{f.itinerario} displays the choice of the alphabet above for the
$1$-dimensional dynamics as in the picture. 

\begin{figure}[th]
\centering
	\includegraphics[scale=0.22]{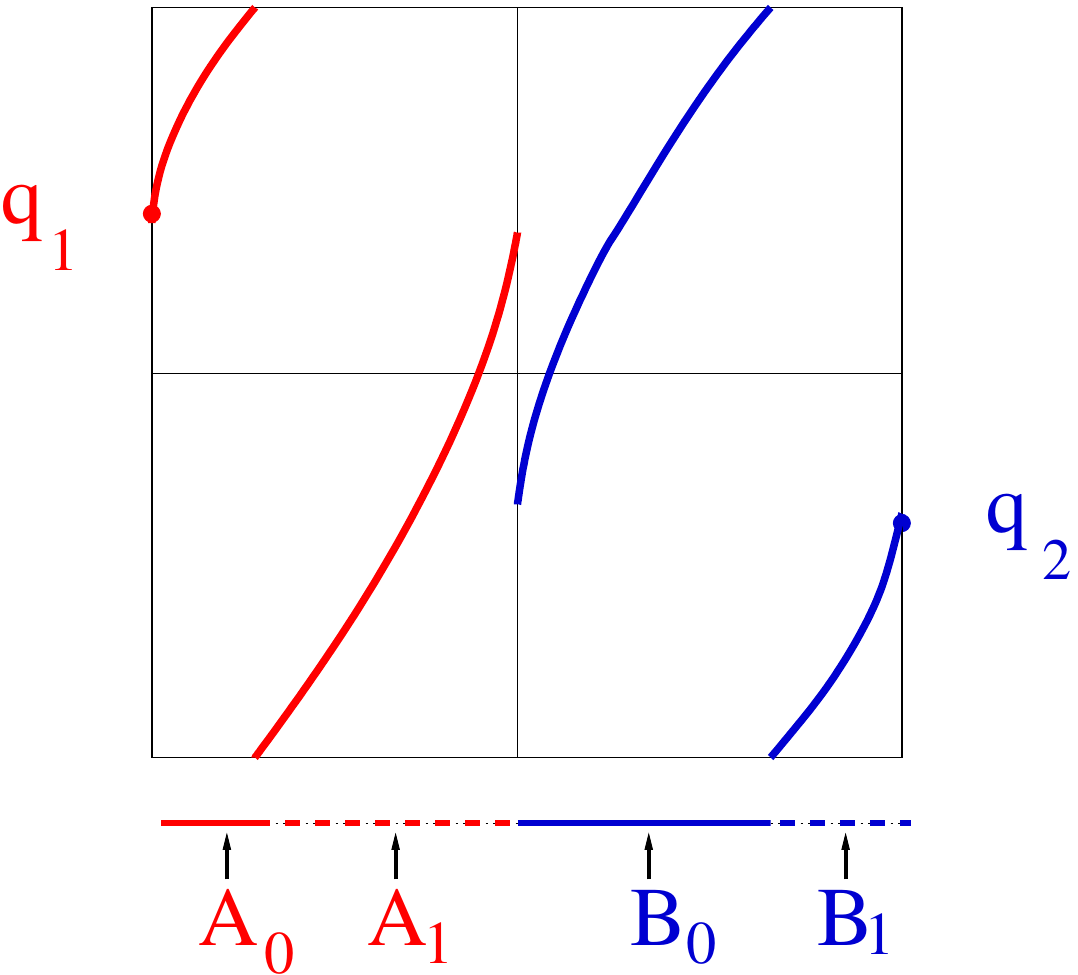}
	\caption{The chosen alphabet.}\label{f.itinerario}
\end{figure}
\begin{lema}\label{l.itinerary}  Consider any point $p\in \Sigma$, and $S\colon[-1,1]\to \Sigma$ a positively oriented unstable segment centered at $p$ (i.e., $S(0)=p$). 
Then for any $k\geq 0$,  the itinerary $[\omega(S(t))]_k$ is well defined and constant for $t>0$  (resp. $t<0$) small enough.  This itinerary is independent of the choice of $S$.  We denote them 
$[\omega_-(p)]_k \quad \mbox{and} \quad [\omega_+(p)]_k.$ 
 If $p_1$ and $p_2$ belong to the same stable leaf then 
 $[\omega_\pm(p_1)]_k = [\omega_\pm(p_2)]_k.$
 In other words, the itinerary depends only on the stable leaf, and not on the point in the leaf. 
\end{lema}

For any $p\in\Sigma$, one denotes by $\omega_-(p)$ and $\omega_+(p)$ the infinite words whose first segments of length $k$ are, respectively,  $[\omega_-(p)]_k$ and $[\omega_+(p)]_k$.  They are called the \emph{down-} and \emph{upper-itinerary} of $p$ respectively.  The \emph{itinerary }of $p$ is the pair of sequences $\omega(p)=(\omega_-(p),\omega_+(p)).$

\begin{rema} If $p\notin W^s(\sigma)$ (that is,  $P^k(p)\notin \gamma^s_\pm$ for all $k>0$),  then $\omega_-(p)=\omega_+(p)$ and the segment of length $k$ is \emph{$k$-itinerary} $[\omega(p)]_k$ of $p$.  This shows that our terminology and notations are consistent. 
\end{rema}

The following remark says that the itineraries $\omega_+$ and $\omega_-$ induce, in some sense, a semi-conjugacy of $(\Sigma, P)$ with $(\XX,\mathfrak{S})$, where $\mathfrak{S}$ is the shift on $\XX$. 
Being rigorous, $P$ is defined on $\Sigma\setminus (\gamma^s_+\cup \gamma^s_-)$, which is not invariant under $P$. 
Thus $\omega_{\pm}$ is a conjugacy in the restriction of $\Sigma\setminus W^s(\sigma)$. 
As $P$ is discontinuous along $\gamma^s+$ and $\gamma^s_-$, we also explain the itinerary of these points. {Recall the {\em{ star map}} (denoted by $\star$) defined on $\XX$ as follows:
given a sequence $w =(w_0, w_1,\cdots ) \in \XX$ and a letter $L \in \{A_0,A_1,B_0,B_1\}$, 
$
L \star w \eqdef (L, w_0, w_1, \cdots)$.}

\begin{rema}\begin{itemize}\item For any $x\in \Sigma\setminus (\gamma^s_+\cup \gamma^s_-)$ (that is, $P(x)$ is defined), then 
$$\omega_-(P(x))=\mathfrak{S}(\omega_-(x))\quad\mbox{and}\quad \omega_+(P(x))=\mathfrak{S}(\omega_+(x))$$

 \item For $x\in \gamma^s_+$   one has:
 $\quad \,\, \omega_+(x)= A_0\star\omega_+(q_1)$
 and
 $$\begin{array}{l}
 \omega_-(x)= B_1\star\omega_-(q_2), \mbox{ if } q_2\notin \gamma^s_+\\
\omega_-(x)= B_0\dots B_0\dots, \mbox{ if } q_2\in \gamma^s_+ \quad (\mbox{and then } B_1=\emptyset)
   \end{array}
$$
\item For $x\in \gamma^s_-$   one has:
 $\quad \,\, \omega_+(x)= B_0\star\omega_+(q_2)$
 and
 $$\begin{array}{l}
 \omega_-(x)= A_1\star\omega_-(q_1), \mbox{ if } q_1\notin \gamma^s_+\\
\omega_-(x)= A_0\star\omega_-(q_1) \mbox{ if } q_1\in \gamma^s_+ \quad (\mbox{and then } A_1=\emptyset)
   \end{array}
$$

\end{itemize}
\end{rema}
According to Lemma~\ref{l.itinerary}, the itineraries $\omega_-$ and $\omega_+$ are functions of the stable leaf.  So they pass to the quotient on the leaves space $\SS^1_X$. We still denote by $\omega_-$ and $\omega_+$ the quotient maps
$\omega_\pm\colon \SS^1_X\to\XX.$

%%%%%%%%%%%%%%%%%%%%%%%%%%%%%%%%%%%%%%%%%%%%%
\subsection{Order and topology}\label{ss.order}
%%%%%%%%%%%%%%%%%%%%%%%%%%%%%%%%%%%%%%%%%%%%%
We endow the alphabet $\{A_0,A_1,B_0,B_1\}$ with the total order 
$A_0<A_1<B_0<B_1<1$,
which corresponds to the order in an unstable segment   starting at $\gamma^s_+$ and crosses the corresponding regions in $\Sigma$. 
We endow $\XX$ with the corresponding lexicographic order that we denote by $\prec$ (and $\preccurlyeq$ for the non-strict order). 

\begin{prop}\label{p.order}Let $S\colon [0,1]\to \Sigma$ be an unstable segment, positively oriented,  whose interior is contained in 
$\Sigma\setminus \gamma^s_+$. 
Then:
\begin{itemize}
 \item for any $t\in(0,1)$  one has $\omega_-(S(t))\preccurlyeq\omega_+(S(t)),$
 \item for any $t_1,t_2\in[0,1]$ so that  $t_1<t_2$ one has $\omega_+(S(t_1))\prec\omega_-(S(t_2)).$
\end{itemize}
\end{prop}
\noindent This proposition has a straightforward translation for the itineraries associated with the $1$-dimensional dynamic $f$.  

\begin{proof} For the first item, we already have seen that if $S(t)$ does not belong to $W^s(\sigma)$, then $\omega_-(S(t))=\omega_+(S(t)$, and there is nothing to prove. 

Consider $t_1<t_2$ so that $S(t_i)\notin W^s(\sigma)$.  As $\omega_+$ and $\omega_-$ coincide on $S(t_i)$ we just note $\omega^i=\omega_+(S(t_i))=\omega_-(S(t_i))$. 
If the first letter is not the same, that is $S(t_1)$ and $S(t_2)$ are not in the same region $A_0,A_1,B_0,B_1$,  then $(\omega^1)_0<(\omega^2)_0$, by choice of the order on our alphabet, and so $\omega^1\prec\omega^2$.

\begin{clai}\label{c.k} {Assume that $(\omega^1)_j=(\omega^2_j)$ for $j=0,\dots,k-1$ but $(\omega^1)_k\neq  (\omega^2)_k$.} 
\begin{itemize}\item $P^j$, $0\leq j\leq k$ is defined on the unstable segment $S([t_1,t_2])$ 
 \item $P^j(S([t_1,t_2])$ is contained in one of the regions $A_0,A_1,B_0,B_1$ for $0\leq j<k$,
 \item $P^k(S(t_1))$ and $P^k(S(t_2))$ are not in the same region
\end{itemize}
\end{clai}
\begin{proof}The third item says $(\omega^1)_k\neq  (\omega^2)_k$, which is our hypothesis. 
The proof of the two first items goes together and by induction. 
As $S(t_1)$ and $S(t_2)$ belong to the same region, and as the interior of $S$ is disjoint from $\gamma^s_+$ and $S$ is an unstable segment (transverse to the fibration by stable leaves), this implies that $S([t_1,t_2])$ is contained in one of the regions of the alphabet. As $P^0=id$ is clearly defined in $S([t_1,t_2])$, we proved both items for $j=0$.
 
We assume now that both items have been proved for $0,\dots, j-1$ and let us prove them for $j$. 
Hence $P^{j-1}(S([t_1,t_1]))$ is contained in one of the regions. Thus $P$ is well defined on this interval, meaning that $P^j$ is well defined on $S([t_1,t_2])$, proving the first item. 

If $j\neq k$, then $(\omega^1)_j=(\omega^2)_j$: in other words, the endpoints of the unstable  segment  $P^{j}(S([t_1,t_1]))$ belong to the same region. If the whole segment is contained in that region, we are done. Otherwise, $P^{j}(S([t_1,t_1]))$ crosses $\gamma^s_+$. 
This means that $P^{j-1}(S([t_1,t_1]))$ is crossing $P^{-1}(\gamma^s_+$, and this (by definition of the regions $A_0,A_1,B_0,B_1$) contradicts the fact that $P^{j-1}(S([t_1,t_1]))$ is contained in one of these regions. This ends the proof of the claim.  
\end{proof}

\begin{clai} With the hypotheses above, $(\omega^1)_k <  (\omega^2)_k$.
\end{clai}
\begin{proof} According to Claim~\ref{c.k}, the  $P^{k-1}(S([t_1,t_1]))$ is an unstable segment contained in one of the regions, and we have seen in the proof that this implies that $P^{k}(S([t_1,t_1]))$  is an unstable segment that does not cross $\gamma^s_+$.   
As already seen, the choice of the order on $A_0,A_1,B_0,B_1$ implies that either $P^{k}(S([t_1,t_1]))$ is contained in one region (which contradicts $(\omega^1)_k\neq  (\omega^2)_k$) or  $(\omega^1)_k <  (\omega^2)_k$, ending the proof of the claim.
\end{proof}

The claims above show that for any  $t_1<t_2$ with $S(t_i)\notin W^s(\sigma)$ one has 
$\omega^1\preccurlyeq\omega^2.$

\begin{clai} Given $t_1<t_2$ so that $S(t_i)\notin W^s(\sigma)$, then 
$\omega^1\neq \omega^2$.
\end{clai}
\begin{proof}As in Claim~\ref{c.k}, if $\omega^1=\omega^2$ then $P^j$ is well defined on $S([t_1,t_2])$ for any $j\geq 0$ and  $P^j(S([t_1,t_2])$ is contained in $1$ of the regions $A_0,A_1,B_0,B_1$.  The length of these iterates increases exponentially, and these forbid (for large iterates) these segments to be contained in one region, ending the proof of the claim. 
\end{proof}

Consider now any $t_1<t_2$. We need to prove
$\omega^1_+=\omega_+(S(t_1))\prec\omega_-(S(t_2))=\omega^2_-.$
By definition of $\omega_-$, there is a decreasing sequence $t_{1,n}<t_2$ tending to $t_1$ and so that:
\begin{itemize}
 \item $S(t_{1,n})\notin W^s(\sigma)$
 \item $[\omega^1_+]_n=[\omega^{1,n}]_n$  (where $\omega^{1,n}$ is the itinerary of $S(t_{1,n})$. 
\end{itemize}
Note that the sequence $\omega^{1,n}$ is strictly decreasing for $\prec$ and tends to $\omega^1_-$.  In other words, 
$\omega^1_-=\inf_{n\to+\infty} \omega^{1,n}. $

In the same way, we fix an increasing sequence $t_{2,n}>t_{1,0}$ tending to $t_2$ and so that
\begin{itemize}
 \item $S(t_{2,n})\notin W^s(\sigma)$
 \item $[\omega^2_-]_n=[\omega^{2,n}]_n$  (where $\omega^{2,n}$ is the itinerary of $S(t_{2,n})$. 
\end{itemize}

\noindent Then $\omega^2_+=\sup_{n\to+\infty} \omega^{2,n}. $
As $\omega^{1,n}\prec \omega^{2,n}$, we conclude 
$\omega^1_+\prec\omega^2_-$
by proving the second item of the proposition.

To end the proof, it remains to show that $\omega_-(S(t))\preccurlyeq\omega_+(S(t))$ for any $t\in (0,1)$. 
For that, we consider sequences $t_{-,n}<t_{-,n+1}<\dots<t<\dots<t_{+,n+1}<t_{+,n}$ tending to $t$ as $n\to +\infty$ and so that 
\begin{itemize}
 \item $S(t_{\pm,n})\notin W^s(\sigma)$
 \item $[\omega_-(S(t))]_n=[\omega^{-,n}]_n$  (where $\omega^{-,n}$ is the itinerary of $S(t_{-,n})$) 
 \item $[\omega_+(S(t))]_n=[\omega^{+,n}]_n$  (where $\omega^{+,n}$ is the itinerary of $S(t_{-,n})$ )
\end{itemize}

We know that $\omega^{-,n}\prec \omega^{+,n}$ (the itinerary is strictly increasing on the point out of $W^s(\sigma)$).  
So for every $n$, one has 
$[\omega_-(S(t))]_n\preccurlyeq [\omega_+(S(t))]_n$, ending the proof. 
\end{proof}

%%%%%%%%%%%%%%%%%%%%%%%%%%%%%%%%%%%%%%%%%%%%%
\subsection{Admissible itineraries}
%%%%%%%%%%%%%%%%%%%%%%%%%%%%%%%%%%%%%%%%%%%%%
Given $X\in \cO_1$, we associate 4 itineraries:
$\omega^+_+=\omega_+(\gamma^s_+)$, $\omega^+_-=\omega_-(\gamma^s_+)$, $\omega^-_+=\omega_+(\gamma^s_-)$ and $\omega^-_-=\omega_-(\gamma^s_-)$. 

\begin{defi} We say that a $\omega\in \XX$ is admissible for the vector field $X$ (or, shortly, $X$-admissible) if it satisfies the following inequalities
\begin{itemize}
 \item $\omega^+_+\preccurlyeq \mathfrak{S}^n(\omega)\preccurlyeq \omega^+_-$. 
 \item If $(\omega)_i\in\{A_0,A_1\}$ then $\mathfrak{S}^i(\omega)\preccurlyeq\omega^-_-$.
 \item If $(\omega)_i\in\{B_0,B_1\}$ then $\omega^-_+\preccurlyeq
 \mathfrak{S}^i\omega$.
\end{itemize}

\noindent We denote  by $\cA_X\subset \XX$ the set of $X$-admissible itineraries. Note that
 $\cA_X$ is a $\mathfrak{S}$-invariant compact set. 
\end{defi}

%\begin{rema} The subset $\cA_X$ is a $\mathfrak{S}$-invariant compact set. 
%\end{rema}

{We note that if $X$ does not exhibit homoclinic loops, then $\underset{x \to c^-_-}{\lim} f^i_X(x)=\underset{x \to c^+_+}{\lim} f^i_X(x)$ and $\underset{x \to c^-_+}{\lim} f^i_X(x)=
\underset{x \to c^+_-}{\lim} f^i_X(x)$, for all $i \in \mathbb{N}$, and therefore} 
 
 \begin{equation}\label{e-itinerario}
\begin{array}{c}
\mathfrak{S}(\omega^+_+)=\mathfrak{S}(\omega^-_-)\\
\mathfrak{S}(\omega^+_-)=\mathfrak{S}(\omega^-_+)
\end{array}
\end{equation}

Remember that $(\omega^+_+)_0= A_0$ and $(\omega^-_+)_0=B_0$ one gets that these $4$ itineraries are determined by $\omega^+_-$ and $\omega^-_-$. 
{If $X$ exhibits a homoclinic loop, the equalities (\ref{e-itinerario})
are no longer true. 
For instance, this is the case when $W^u_+(\sigma) \cap W^s_+(\sigma) \neq \emptyset$ and $W^u_-(\sigma) \cap W^s(\sigma) = \emptyset$, $\omega_+^+$ is periodic and $\omega_-^-$ is not. See Figure \ref{f-itinerario-45}.
However, since $\underset{x \to c^-_-}{\lim} f_X(x)=\underset{x \to c^+_+}{\lim} f_X(x)$ and $\underset{x \to c^-_+}{\lim} f_X(x)=\underset{x \to c^+_-}{\lim} f_X(x)$, the set $\cA_X$ is still determined by 
$\{\omega_+^+,\omega_-^+\}$ or $\{\omega_-^-,\omega_+^-\}$.}

\begin{figure}[!h]
	\centering
	\includegraphics[scale=0.06]{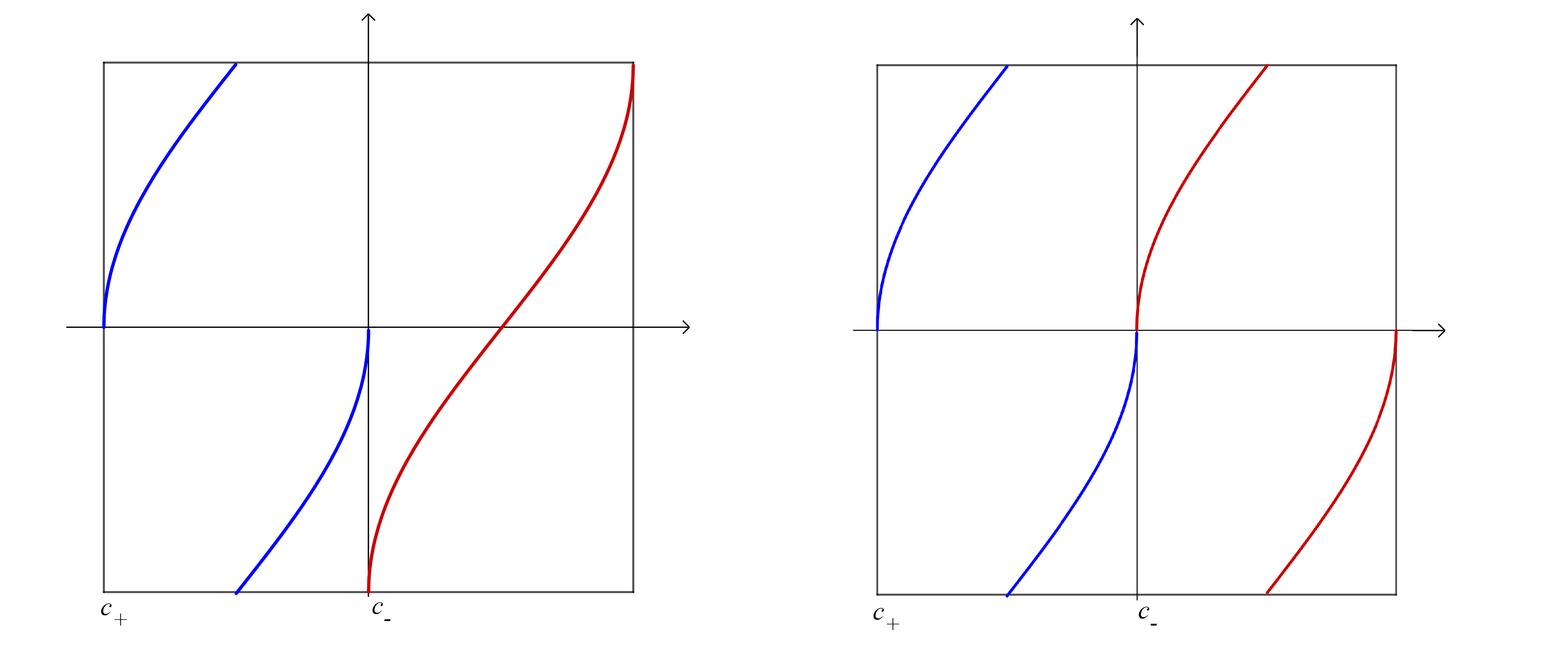}
	\caption{(a) $\mathcal{H}_-^1 \cap \mathcal{H}^2_+$,\hspace{2,5cm} (b)  $\mathcal{H}_-^1 \cap \mathcal{H}_-^2$.}\label{f-itinerario-45}
\end{figure}

\begin{eqnarray*} \mbox{Itineraries of} \quad \mathcal{H}_-^1 \cap \mathcal{H}_-^2 &\hspace{3cm}&  \mbox{Itineraries of} \quad \mathcal{H}_-^1 \cap \mathcal{H}_+^2\\
\omega_+^+=A_0B_0A_0B_0\ldots\,\,\,\,\, &&\,\,\,\,\, \omega_+^+=A_0B_0B_0B_0\ldots\\
\omega_-^-=A_1A_1A_1A_1\ldots \,\,\,\,\,&&\,\,\,\,\, \omega_-^-=A_1A_1A_1A_1\ldots \\
\omega_-^+=B_0A_0B_0A_0\ldots \,\,\,\,\,&&\,\,\,\,\, \omega_-^+=B_1A_1A_1A_1\ldots\\
\omega_+^-=B_0B_0B_0B_0\ldots \,\,\,\,\,&&\,\,\,\,\, \omega_+^-=B_0B_0B_0B_0\ldots	
\end{eqnarray*}	

\vspace{1cm}

\begin{lema}\label{lemma.ad}If $p\in\Sigma$ then $\omega_-(p)$ and $\omega_+(p)$ are $X$-admissible. 
\end{lema}
\begin{proof} Consider an unstable segment $S\colon [0,1]\to \Sigma$ whose interior is disjoint from $\gamma^s_+$ and  so that $S(0),S(1)\in \gamma^s_+$ and $p\in S([0,1])$.  Then Proposition~\ref{p.order} applies and implies that if $p\notin \gamma^s_+$ then 
$\omega^+_+\preccurlyeq \omega_-(p)\preccurlyeq\omega_+(p)\preccurlyeq \omega^+_-$. 
In particular, 
$$\begin{array}{c}
\omega^+_-=\max\{ \omega_-(p), \omega_+(p), p\in\Sigma\}\\
\omega^+_+=\min\{ \omega_-(p), \omega_+(p), p\in\Sigma\}
  \end{array}
$$
In particular, this shows that $\omega_-(p)$ and $\omega_+(p)$ satisfy the first item of the definition of $X$-admissibility. 

The other two items correspond to several cases whose proof is very similar. Let us present one of these cases.
Let $p=S(t)$ so that $(\omega_+(p))_0\in\{A_0,A_1\}$.  
Then there is a  decreasing sequence $t_n\to t$ so that $S(t_i)\notin W^s(\sigma)$ and 
$[\omega_+(p)]_n=[\omega^n]_n$
where $\omega^n$ denotes $\omega_-(S(t_n))=\omega_+(S(t_n))$.

Then $\omega_+(p)=\inf  \omega^n$. 
On the other hand, $S(t_n)$ is a point out of $W^s(\sigma)$ and contained in $A_0\cup A_1$, and thus in $\Sigma_1$.  Proposition~\ref{p.order} implies that $\omega^n\prec \omega^-_-$, finishing this case, and the proof. 
\end{proof}

%%%%%%%%%%%%%%%%%%%%%%%%%%%%%%%%%%%%%%%%%%%%%
\subsection{Realizing $X$-admissible itineraries}
%%%%%%%%%%%%%%%%%%%%%%%%%%%%%%%%%%%%%%%%%%%%%

This section aims to prove

\begin{prop}\label{p.realisacao} Given any $\omega\in \cA_X$, there is $p\in\Sigma$ so that
$$\omega\in\{\omega_+(p),\omega_-(p)\}.$$
\end{prop}
\begin{rema}\label{r.realisacao}  Proposition~\ref{p.order} implies that any two points satisfying the conclusion of Proposition~\ref{p.realisacao} belong to the same stable leaf. 
 
\end{rema}

Proposition~\ref{p.realisacao} is a direct consequence of Lemma~\ref{l.realisacao} below:

\begin{lema}\label{l.realisacao} Given any $\omega\in \cA_X$, the set $\Omega_n$ of points $p$ in $\Sigma$ so that 
$$[\omega]_n\in \{[\omega_+(p)]_n,[\omega_-(p)]_n\}\quad \quad \mbox{is a non empty compact subset of $\Sigma$.}$$
 
\end{lema}

Assuming that Lemma~\ref{l.realisacao} is true, then the sequence $\Omega_n$ is a nested sequence of non-empty compact sets, and any $p\in\bigcap\Omega_n$ satisfies that $\omega\in\{\omega_+(p),\omega_-(p)\}.$  Note that, indeed, 
this intersection is precisely  the stable leaf through $p$, according to Proposition~\ref{p.order}. 
It remains to prove Lemma~\ref{l.realisacao}. 

%%%%%%%%%%%%%%%%%%%%%%%%%%%%%%%%%%%%%%%%%%%%%
\subsection{Proof of Lemma~\ref{l.realisacao}}
%%%%%%%%%%%%%%%%%%%%%%%%%%%%%%%%%%%%%%%%%%%%%

For any itinerary $\omega\in\cA_X$, we denote by $\Omega_n(\omega)$ the set of points $q\in\Sigma$ so that $[\omega]_n\in \{[\omega_-(q)]_n,[\omega_+(q)]_n\}$.

\begin{lema}\label{l.itinerarios-contantes}  Let $\omega\in \cA_X$ so that there is some $p\in\Sigma$ for which $\omega\in\{\omega_-(p),\omega_+(p)\}$. 
Fix $n\in\NN$ and denote $\Omega_n(\omega)$ the set of points $q\in\Sigma$ so that $[\omega]_n\in [\omega_-(q)]_n,[\omega_+(q)]_n$. 
Then $\Omega_n(\omega)$ is the closure of a connected component of 
$$\Gamma_n\eqdef \Sigma\setminus\bigcup_{i=0}^{n-1} P^{-i}(\gamma^s_+\cup P^{-1}(\gamma^s_+)\cup\gamma^s_-).$$
 
\end{lema}
\begin{proof} First, notice that $\Gamma_n$ consists of the union of finitely many stable leaves. 
Consider an unstable segment $S\colon[0,1]\to\Sigma$ whose interior is disjoint from $\Gamma^n$ and has its endpoints on $\Gamma_n$.  Let $\Omega_n$ be the closure of the connected component of $\Sigma\setminus \Gamma_n$ containing $S((0,1))$.  
Then for any 
$0\leq i <n$, $P^i(S((0,1)))$ is well defined and disjoint from $\gamma^s_\pm$ and of $P^{-1}(\gamma^s_+)$. In other words, $P^i(S((0,1)))$ is contained in one of the regions representing  the alphabet. Thus $[\omega_-(S(t))]_n$ does not depend on $t\in(0,1)$ and is equal to $[\omega_+(S(0))]_n$ and $[\omega_-(S(1))]_n$.  
One deduces that $[\omega_-]_n$ and $[\omega_+]_n$ are equal and constant
on the interior of $\Omega_n$, and $[\omega_-]_n$ takes the same value on one of the boundary stable leaves, and $[\omega_+]_n$ on the other boundary stable leaf. 
This proves that $\Omega_n(\omega)$ is a union of such closures of connected components $\Omega_n$. 

Fix $\Omega_n\subset \Omega_n(\omega)$ and let $q\notin \Omega_n$. 
If $q$ is not in the same region $\{A_0,A_1,B_0,B_1\}$ as $\Omega_n$ then 
$[\omega_\pm(q)]_n$ is not $[\omega]_n$.  Otherwise, there is an unstable segment ( still denoted by $S$) in this region (hence disjoint from $\gamma^s_+\cup\gamma^s_-\cup P^{-1}(\gamma^s_+)$, joining $q$ to a point $p$ in the interior of $\Omega_n$. 

The interior  $\interior(S)$ of segment $S$ crosses the boundary of $\Omega_n$ that is  crossed $\Gamma_n$. 
Let $i$ be the smallest integer so that $\interior(S)\cap P^{-i}(\gamma^s_+\cup\gamma^s_-\cup P^{-1}(\gamma^s_+))$. 

Then $P^{i-1}(S)$ is  contained in the closure of  one of the regions
$\{A_0,A_1,B_0,B_1\}$ but not $P^i(S)$.  This implies that  the two endpoints of $P_i(S)$ are not in the same region $\{A_0,A_1,B_0,B_1\}$.  This 
implies  that $[\omega_+(q)]_i$ and $[\omega_-(q)]_i$ are different from $[\omega]_i$, proving that $q\notin \Omega_n(\omega)$ and  finishes the proof. 
\end{proof}

\begin{proof}[Prof of Lemma~\ref{l.realisacao}.] The proof goes by induction.
 We want to prove that, for every $n\geq 0$ and every $\omega\in\cA_X$,   $\Omega_n(\omega)$ is the closure of one connected component of $\Sigma\setminus \Gamma_n$.
 Let us check if this is true for $n=0$. Each itinerary of length $0$ is one letter of our alphabet, which corresponds to a connected component of $\Sigma\setminus\Gamma_0$, and its closure is the one we announced. 
  We now prove it also holds for $n=1$.  Assume, for instance, that $(\omega)_0=A_0$. Thus $\Omega_0(\omega)=\bar A_0$ and $P(\Omega_0(\omega))$ are a cuspidal triangle starting at $q_1$, and ending at $\gamma^s_+$. 
  By definition of $\cA_X$, one has  
 $\omega_+(\gamma_+^s)\preccurlyeq\omega\preccurlyeq\omega_-(\gamma^s_-).$ 
 As the first letter of $\omega$ is the same as the first letter of $\omega_+(\gamma^s_+) $ one gets 
 $\mathfrak{\omega_+}(\gamma^s_+)=\omega_+(q_1)\preccurlyeq\mathfrak{S}(\omega).$
 In particular, $(\omega)_1=(\mathfrak{S}(\omega))_0$  either is strictly bigger  than or is equal to $(\omega_+(q_1))_0$.  
 In both cases, $P(\Omega_0(\omega))$ intersects $\Omega_0(\mathfrak{S}(\omega))$, proving that $\Omega_2(\omega)$ is not empty. Now Lemma~\ref{l.itinerarios-contantes} asserts that it is the closure of a connecting component of $\Sigma\setminus \Gamma_2$ proving the induction hypothesis in that case. 
 The case $(\omega)_0=A_1,B_0,B_1$ is very similar. 
 
\noindent  We assume that the induction hypotheses have been proved for $i=0\dots n$.  
Consider $\omega\in\cA_X$. The induction hypothesis, $\Omega_n(\omega)$, is the closure of one connected component of $\Sigma\setminus \Gamma_n$.
 We split the proof into cases. 
 
 %%%%%%%%%%%%%%%%%%%%%%%%%%%%
 \noindent\underline{Case 1:} $\gamma^s_+$ and $\gamma^s_-$ are not contained in the compact set 
 $\Omega_n(\omega)$. 
 
 Then $P$ is defined on $\Omega_n(\omega)$, and the boundary
 $\partial P(\Omega_n(\omega))$ is contained in $\Gamma_{n-1}$. 
 This implies that $P(\Omega_n(\omega))$ crosses every stable leaf in the closure of a connected component of $\Sigma\setminus \Gamma_{n-1}$, which has to be $\Omega_{n-1}(\mathfrak{S}(\omega))$ (by Lemma~\ref{l.itinerarios-contantes}).
 By the induction hypothesis $\Omega_n(\mathfrak{S}(\omega))$ is a connected component of $\Sigma\setminus \Gamma_n$ and is contained in $\Omega_{n-1}(\mathfrak{S}(\omega))$. 
 This implies that $P(\Omega_n(\omega))$ intersects $\Omega_{n-1}(\mathfrak{S}(\omega))$.  Thus 
 $\Omega_{n+1}(\omega)$ is not empty and therefore is a connected component of $\Sigma\setminus \Gamma_{n+1}$ by Lemma~\ref{l.itinerarios-contantes}.
 
 %%%%%%%%%%%%%%%%%%%%%%%%%%%%
 \noindent\underline{Case 2:} We now assume that $1$ of the boundary components of  $\Omega_n(\omega)$ is $\gamma^s_+$ and the other is not $\gamma^s_-$. Up to reverse the orientation, we assume that the positively oriented unstable segments starting at $\gamma^s_+$ enter in $\Omega_n(\omega)$. 
 Then $P(\Omega_n(\omega))$ is a cuspidal triangle starting at $q_1$ and ending on a stable leaf in $\Gamma_{n-1}$. 
  Now $\Omega_{n-1}(\mathfrak{S}(\omega))$ is (induction hypothesis) the closure of a connected component in $\Sigma\setminus \Gamma_{n-1}$, which contains $P(\Omega_n(\omega))$ and thus contains $q_1$. 
 Now $\Omega_{n}(\mathfrak{S}(\omega))$ is (induction hypothesis) the closure of a connected component in $\Sigma\setminus \Gamma_{n}$. 
 
 \begin{clai}$P(\Omega_n(\omega))\cap \Omega_{n}(\mathfrak{S}(\omega) \neq \emptyset$
  
 \end{clai}
 \begin{proof}
 Note that the first letter of $\omega$ is $A_0$.  As $\omega\in\cA_X$ on has $\omega_+(\gamma^s_+)\preccurlyeq\omega$.  As their first letters are equal, this implies 
 $\mathfrak{\omega}_+(\gamma^s_+)=\omega_+(q_1)\preccurlyeq\mathfrak{S}(\omega).$
 Consider a positively oriented unstable segment $S\colon[0,1]\to \Sigma$ crossing every stable leaf in $\Omega_{n-1}(\mathfrak{S}(\omega))$ and containing $q_1=S(t_0)$. 
 Then $S$ is crossing $\Omega_n(\mathfrak{S}(\omega))$
 at point $S(t)$. 
  Recall that $\omega_+(q_1)\preccurlyeq\mathfrak{S}(\omega)$.  
 Recall that the function $[\omega_+(S(t))]n$ is non-decreasing with $t$, so that $S^{-1}(\Omega_n(\mathfrak{S}(\omega)))\subset [0,1]$
 is not inferior to $S^{-1}(\Omega_n(\omega_+(q_1))$. 
 As a consequence, $S^{-1}(\Omega_n(\mathfrak{S}(\omega)))\subset [0,1]$ either coincides with $S^{-1}(\Omega_n(\omega_+(q_1))$ or is strictly larger than $t_0$. 
 In both cases, $P(\Omega_n(\omega))$ intersects $\Omega_{n}(\mathfrak{S}(\omega)) $, concluding the proof of the Claim. 
\end{proof}
 
 This implies that $\Omega_{n+1}(\omega)$ is not empty, and Lemma~\ref{l.itinerarios-contantes} concludes that case.
 
 %%%%%%%%%%%%%%%%%%%%%%%%%%%%
 \noindent\underline{Case 3:} We now assume that $1$ of the boundary components of  $\Omega_n(\omega)$ is $\gamma^s_-$ and the other is not $\gamma^s_+$.
 
  Up to reverse the orientation, we assume that the positively oriented unstable segments starting at $\gamma^s_+$ enter in $\Omega_n(\omega)$. 
 Then $P(\Omega_n(\omega))$ is a cuspidal triangle starting at $q_2$ and ending on a stable leaf in $\Gamma_{n-1}$. 
 Note that the first letter of $\omega$ is $B_0$.  As $\omega\in\cA_X$ on has $\omega_+(\gamma^s_-)\preccurlyeq\omega$.  As their first letters are equal, this implies 
 $$\mathfrak{S}(\omega)\preccurlyeq\mathfrak{S}\omega_+(\gamma^s_+)=\omega_+(q_2).$$
 The proof follows now in a similar way to Case 2.  
 
 %%%%%%%%%%%%%%%%%%%%%%%%%%%%
 \noindent\underline{Case 4:} Finally, we assume that both boundary components of  $\Omega_n(\omega)$ are $\gamma^s_-$ and  $\gamma^s_+$.
 This implies that $\Omega_n(\omega)$ is the closure of $\Sigma_1$ or of $\Sigma_2$, and thus $P(\Sigma)$ crosses  $\Omega_n(\mathfrak{S}(\omega))$ concluding. 
 Now the proof of Lemma~\ref{l.realisacao} (and therefore of Proposition~\ref{p.realisacao}) is complete. 
\end{proof}

%%%%%%%%%%%%%%%%%%%%%%%%%%%%%%%%%%%%%%%%%%%%%
\subsection{Itineraries, conjugacy, and topological equivalence}
%%%%%%%%%%%%%%%%%%%%%%%%%%%%%%%%%%%%%%%%%%%%%

\begin{teo}\label{L.T} Consider $X,Y\in \cO_1$ and let $f_X,f_Y$ be the corresponding $1$-dimensional dynamics. 
Assume that the itineraries   
$\omega_-(\gamma^s_-,X)=\omega_-(\gamma^s_-,Y)$ and $\omega_-(\gamma^s_+,X)=\omega_-(\gamma^s_+,Y)$. 
Then there is an orientation preserving map of $\SS^1$, which is a conjugation between $f_X$ and $f_Y$. 
\end{teo}

\begin{proof} We have seen that the sets $\cA_X$ and $\cA_Y$ of admissible itineraries for $X$ and $Y$ are wholly determined by $\omega_-(\gamma^s_\pm, X)$ and $\omega_-(\gamma^s_\pm, Y)$, respectively.  Thus 
$\cA_X=\cA_Y\eqdef\cA.$

Now for any $\omega\in\cA$, Proposition~\ref{l.realisacao} and Remark~\ref{r.realisacao} imply that there are a unique point $x_\omega\in \SS^1_X$ and $y_\omega\in\SS^1_Y$ so that $\omega\in \{\omega_-(x_\omega,f_X),\omega_+(x_\omega,f_X)\}$ and $\omega\in \{\omega_-(y_\omega,f_Y),\omega_+(y_\omega,f_Y)\}$. 

We define $h(x_\omega)=y_\omega$.  This defines a bijection from $\SS^1_X$ to $\SS^1_Y$, which sends $\gamma^s_{i,X}$ on $\gamma^s_{i,Y}$.
The punctured circle is an interval endowed with an order (from the positive orientation of the unstable segments), and Proposition~\ref{p.order} implies that $\omega\mapsto x_\omega$ and $\omega\mapsto y_\omega$ are  increasing.
This implies that $h\colon x_\omega\mapsto y_\omega$ is an increasing bijection from $\SS^1_X\setminus\{\gamma^s_{+,X}\}$ onto $\SS^1_Y\setminus\{\gamma^s_{+,Y}\}$.  An increasing bijection between intervals is a homeomorphism and so $h$ is a homeomorphism. 
The fact that $h$ is a conjugacy now comes from the fact that $x_{\mathfrak{S}(\omega)}= f_X(x)$ for $x\notin \{\gamma^s_{\pm,X}\}$.
\end{proof}

{Recall that the discontinuities are fixed points for the conjugacy $h$ constructed above.}
We finish proving Theorem \ref{igual-a-t.conjugado}, which establishes that 
 the restriction to the maximal invariant set of $X,\, Y \in \mathcal{O}_1$ is topologically equivalent by a conjugacy close to identity if, and only if, $X$ and $Y$ have the same itineraries.
 
 We are ready to present a proof of Theorem \ref{igual-a-t.conjugado}:
 
\begin{proof} ({\em{of Theorem \ref{igual-a-t.conjugado}}})
{($\Rightarrow$) A conjugation $\mathbb{H}$ between $X|_{\Lambda_X}$ and $Y|_{\Lambda_Y}$ induces a topological conjugation $h:\SS^1 \to \SS^1$ between $f_X:\SS^1 \to \SS^1$ and $f_Y:\SS^1 \to \SS^1$. 
Let  $\varepsilon=\frac{1}{2}\min\{d(D_1,\Sigma;d(D_2,\Sigma))\}$ (see Section \ref{ss-the-open-set-1}
 for the definitions of $D_1$ and $D_2$). If $d(\mathbb{H},id)<\varepsilon$, $h$ is orientation preserving and $h(c_{i,X})=c_{i,Y}$ for $i \in \{-,+\}$, and therefore $X$ and $Y$ have the identical  itineraries.\\
($\Leftarrow$ ) Fix $p \in \Sigma \cap \Lambda_X$.  Because $X$ and $Y$ have the same itineraries, Lemma \ref{l.realisacao} together with Theorem \ref{L.T}  ensures the existence of $q \in \Sigma \cap \Lambda_Y$ such that for all $n \in \mathbb{N}$, points in $\gamma^s_{X}(P_X^{-n}(p))$ and  $\gamma^s_{Y}(P_Y^{-n}(q)$ have the same itineraries. Thus $C_n=\{P_Y^{n}(\gamma^s_{Y}(P_Y^{-n}(q))\}_{n \in \mathbb{N}}$ is a sequence of compact sets in $\gamma^s_Y(q)$ converging to a single point, and we define $H(p)=\underset{n \in \mathbb{N}}{\cap} C_n$. Lema \ref{lemma.ad} implies that $H$ is onto. For $p_1,p_2 \in \Sigma$ such that $\gamma^s_X(p_1) \neq \gamma^s_X(p_2)$, it is easy to see that $H(p_1)\neq H(p_2)$. For $p_1$ and $p_2$ in the same leave, there exists $n_1 \in \mathbb{N}$ for which $P^{-n_1}(p_1)$ and $P^{-n_1}(p_2)$ belongs to different connected component of $\Sigma \setminus \{\gamma^s_{-,X},\gamma^s_{+,X}\}$, which implies that $\gamma^s_{X}(P_X^{-n_1}(p_1)) \cap \gamma^s_X(P_X^{-n_1})(p_2)) = \emptyset$ and hence $H_n(p_1) \cap H_n(p_2) \neq \emptyset$ for all $n>n_1$, providing $H$ injectivity. The existence of unstable cone fields around $\Sigma$ and the continuity of $h$ and $h^{-1}$ give the continuity of $H$ and $H^{-1}$. To finish, for $p \in \Sigma \cap \Lambda_X$, consider $\alpha \subset \mathcal{O}(p)$ and $\beta \subset \mathcal{O}(H(p))$, curves parametrized by the arc length, joining $p$ to $D^1 \cup D^2$ and $H(p)$ to $D^1 \cup D^2$ respectively, in a way that $\alpha(t_1),\beta(t_2) \notin D^1 \cup D^2$ for all $0<t_1<\ell(\alpha)$ and $0<t_2<\ell(\beta)$). For $\rho$ being the ratio of the length of $\alpha$ to the length of $\beta$, we define $\mathbb{H}(t)=\beta(\rho t)$. Extend this map to segments of trajectories leaving $D^1 \cup D^2$ and  returning to $\Sigma$ in the same way as before. $\mathbb{H}$ defines a topological equivalence. Note that $\mathbb{H}(D_i)=D_i$ and $\mathbb{H}(\Sigma)=\Sigma$, which implies that $d(\mathbb{H},id)<\varepsilon$ and we have the result.}
	
\end{proof}

\noindent Diego Barros and Maria Jos\'e Pacifico\\
Instituto de Matem\'atica,
Universidade Federal do Rio de Janeiro \\
Rio de Janeiro, Brazil\\
(e-mail: $[D. B.]$ {\em diegosbarros@macae.ufrj.br} and
$[M.J.P.]$ {\em pacifico@im.ufrj.br})

\vspace{0.3cm}
\noindent Christian Bonatti\\
Institut de Math\'ematiques de Bourgogne, CNRS, Universit\'e de Bourgogne\\
 Dijon, France\\ 
(e-mail: bonatti@u-bourgogne.fr)

\end{document}